\documentclass[reqno,12pt,letterpaper]{amsart}
\usepackage{multirow}
\usepackage{amsfonts,amsrefs,latexsym,amsmath, amssymb, mathrsfs, verbatim, slashed}
\usepackage{url}
\usepackage{pifont}
\usepackage{upgreek}
\usepackage{fancyhdr}
\usepackage{calligra}
\usepackage{marvosym}
\usepackage[percent]{overpic}
\usepackage{pict2e}
\usepackage[hypcap=false]{caption}
\usepackage{wasysym}
\usepackage{accents}
\usepackage[usenames,dvipsnames]{color}
\usepackage[colorlinks=true,linkcolor=Red,citecolor=Green]{hyperref}

\topmargin - .23 in

\newtheorem{theorem}{Theorem}[section]
\newtheorem{proposition}[theorem]{Proposition}
\newtheorem{lemma}[theorem]{Lemma}
\newtheorem{corollary}[theorem]{Corollary}

\theoremstyle{definition}
\newtheorem{definition}[theorem]{Definition}
\newtheorem{remark}[theorem]{Remark}

\DeclareMathAlphabet{\mathcalligra}{T1}{calligra}{m}{n}
\DeclareFontShape{T1}{calligra}{m}{n}{<->s*[2.2]callig15}{}

\newcommand{\Tboot}{T_{(Boot)}}

\newcommand{\Ntop}{N_{Top}}
\newcommand{\Nmid}{N_{Mid}}

\newcommand{\noshock}{v}
\newcommand{\noshockuparg}[1]{v^{#1}}
\newcommand{\diffnoshock}{V}
\newcommand{\diffnoshockdownarg}[1]{V_{#1}}
\newcommand{\diffnoshockdoublearg}[2]{V_{#1}^{#2}}

\newcommand{\blowupcoeff}{\mathcal{G}}

\newcommand{\Trandatasize}{\mathring{A}}

\newcommand{\TranminusdatasizeWithFactor}{\mathring{A}_{\ast}}

\newcommand{\Psiep}{\mathring{\upalpha}}

\newcommand{\covL}{H}

\newcommand{\GdVar}{\upgamma}
\newcommand{\BadVar}{\underline{\upgamma}}

\newcommand{\Fullset}{\mathscr{Z}}
\newcommand{\Tanset}{\mathscr{P}}

\newcommand{\Singletan}{P}

\newcommand{\Lunit}{L}
\newcommand{\Lunitback}{\widetilde{L}}

\newcommand{\Radunit}{X}
\newcommand{\Radunitback}{\widetilde{X}}
\newcommand{\Rad}{\breve{X}}

\newcommand{\CoordAng}[1]{{^{(#1) \mkern-3mu} \Theta}}
\newcommand{\CoordAngcomp}[2]{{^{(#1)} \Theta^{#2}}}
\newcommand{\CoordAngSmallcomp}[2]{{^{(#1)} \Theta_{(Small)}^{#2}}}

\newcommand{\vol}{\varpi}
\newcommand{\tvol}{\underline{\varpi}}
\newcommand{\conevol}{\overline{\varpi}}

\newcommand{\torusvol}{\vartheta}

\newcommand{\shocken}{\mathbb{E}^{(Shock)}}
\newcommand{\noshocken}{\mathbb{E}^{(Regular)}}
\newcommand{\noshockfl}{\mathbb{F}^{(Regular)}}
\newcommand{\totmax}{\mathbb{Q}}

\newcommand{\smoothfunction}{\mathrm{f}}

\numberwithin{equation}{subsection}

\begin{document}
\title{Multidimensional nonlinear geometric optics 
for transport operators 
with applications to stable shock formation}
\author[JS]{Jared Speck$^{*\dagger}$}

\thanks{$^{\dagger}$JS gratefully acknowledges support from NSF grant \# DMS-1162211,
from NSF CAREER grant \# DMS-1454419,
from a Sloan Research Fellowship provided by the Alfred P. Sloan foundation,
and from a Solomon Buchsbaum grant administered by the Massachusetts Institute of Technology.
}

\thanks{$^{*}$Massachusetts Institute of Technology, Cambridge, MA, USA.
\texttt{jspeck@math.mit.edu}}

\begin{abstract}
In $n \geq 1$ spatial dimensions,
we study the Cauchy problem for a quasilinear transport equation 
coupled to a quasilinear symmetric hyperbolic subsystem of a rather general type.
For an open set (relative to a suitable Sobolev topology) of regular initial data
that are close to the data of a simple plane wave,
we give a sharp, constructive proof of shock formation 
in which the transport variable remains bounded but its
first-order Cartesian coordinate partial derivatives blow up in finite time.
Moreover, we prove that the singularity does not propagate into 
the symmetric hyperbolic variables: they and their first-order Cartesian coordinate partial derivatives remain bounded,
even though they interact with the transport variable all the way up to its singularity.
The formation of the singularity is tied to the finite-time degeneration, relative to 
the Cartesian coordinates, of a system of geometric coordinates adapted to the characteristics of the transport operator.
Two crucial features of the proof are that relative to the geometric coordinates, 
all solution variables remain smooth,
and that the finite-time degeneration coincides with the intersection of the characteristics.
Compared to prior shock formation results in more than one spatial dimension, 
in which the blowup occurred in solutions to wave equations, 
the main new features of the present work are:
\textbf{i)} we develop a theory of nonlinear geometric optics 
for transport operators,
which is compatible with the coupling and which allows us to implement a quasilinear geometric vectorfield method, 
even though the regularity properties of the corresponding eikonal function are less favorable
compared to the wave equation case
and \textbf{ii)} we allow for a full quasilinear coupling, i.e., 
the principal coefficients in all equations are allowed to depend on all solution variables.


\bigskip


\bigskip

\end{abstract}

\maketitle

\centerline{\today}

\tableofcontents
\setcounter{tocdepth}{1}

\newpage

\section{Introduction}
\label{S:INTRO}
The study of quasilinear hyperbolic PDE systems is one of the most classical pursuits
in mathematics and, at the same time, among the most active.
Such systems are of intense theoretical interest, in no small part due to the fact that 
their study lies at the core of the revered field of nonlinear hyperbolic conservation laws 
(more generally ``balance laws''); we refer readers to Dafermos' work \cite{cD2010} for a detailed discussion 
of the history of nonlinear hyperbolic balance laws
as well as a comprehensive introduction to the main results
of the field and the main techniques behind their proofs, 
with an emphasis on the case of one spatial dimension.
The subject of quasilinear hyperbolic systems is of physical interest as well, 
since they are used to model a vast range of physical phenomena.
A fundamental issue surrounding the study of the initial value problem for such PDEs
is that solutions can develop singularities in finite time, 
starting from regular initial data.
In one spatial dimension, the theory is in a rather advanced state, and in many cases,
the known well-posedness results are able to accommodate the formation of shock singularities
as well as their subsequent interactions; see the aforementioned work of Dafermos. 
The advanced status of the one-space-dimensional theory 
is highly indebted to the availability of estimates in the space of functions of bounded variation (BV).
In contrast, Rauch \cite{jR1986} showed that for quasilinear hyperbolic systems in more than one spatial dimension, 
well-posedness in BV class \emph{generally does not hold}.
For this reason, energy estimates in 
$L^2$-based Sobolev spaces play an essential role in multiple spatial dimensions, 
and even the question of whether or not there is stable singularity formation 
(starting from regular initial data) can be exceptionally challenging. 
In particular, in order to derive a constructive shock formation result in
more than one spatial dimension, one cannot avoid 
the exacting task of deriving energy estimates that hold up to the singularity. 

In view of the above remarks, it is not surprising that the earliest blowup results for quasilinear hyperbolic PDEs 
in more than one spatial dimension without symmetry assumptions
were not constructive, but were instead based on proofs by contradiction, 
with influential contributions 
coming from, for example, John \cite{fJ1981} for a class of wave equations
and Sideris for a class of hyperbolic systems \cite{tS1984}
and later for the compressible Euler equations \cite{tS1985}.
The main idea of the proofs was to show that for smooth solutions with suitable initial data,
certain spatially averaged quantities verify ordinary differential inequalities 
that force them to blow up, contradicting the assumption of smoothness.

Although the blowup results mentioned in the previous paragraph are compelling,
their chief drawback is that they provide 
no information about the nature of the singularity, 
other than an upper bound on the solution's classical lifespan. 
In particular, such results are not useful if one aims
to extract sharp information about the blowup-mechanism and blowup-time, or if one 
aims to uniquely continue the solution past the singularity in a weak sense.
In contrast, many state-of-the-art blowup-results for hyperbolic PDEs
yield a detailed description of the singularity formation, even in the challenging 
setting of more than one spatial dimension. This is especially true for results on the formation of shocks
starting from smooth initial conditions, a topic that has enjoyed remarkable progress 
in the last decade, as we describe in Subsect.\ \ref{SS:PRIORWORKANDCONTEXT}.
Our main results are in this vein.
We recall that a shock singularity\footnote{The formation of a shock is sometimes referred to as ``wave breaking.''}
is such that some derivative of the solution
blows up in finite time while the solution itself remains bounded.
Shock singularities are of interest in part due to their
rather mild nature, which leaves open the hope that one might be able to
extend the solution uniquely past the shock, in a weak sense, under suitable selection
criteria. In the case of the compressible Euler equations in three spatial dimension, 
this hope has been realized in the form of
Christodoulou's recent breakthrough resolution \cite{dC2017} 
of the \emph{restricted shock development} problem without symmetry problems; 
see Subsubsect.\ \ref{SSS:MULTIDRESULTS} for further discussion.

We now provide a very rough statement of our results;
see Theorem~\ref{T:ROUGHMAINTHM} on pg.~\pageref{T:ROUGHMAINTHM} for a 
more detailed summary and Theorem~\ref{T:MAINTHM} on pg.~\pageref{T:MAINTHM}
for the complete statements.

\begin{theorem}[\textbf{Stable shock formation} (very rough version)]
	\label{T:EXTREMELYROUGH}
	In an \underline{arbitrary number} of spatial dimensions,
	there are many quasilinear hyperbolic PDE systems
	comprising a transport equation
	coupled to a symmetric hyperbolic subsystem such that the following occurs: 
	there exists an open set of initial data 
	\underline{without symmetry assumptions}
	such that the transport variable remains bounded but its first derivatives
	blow up in finite time. More precisely, the derivatives of the transport variable
	in directions \textbf{tangent} to the transport characteristics remain bounded,
	while any derivative in a \textbf{transversal} direction
	blows up.
	Moreover, the singularity does not propagate into
	the symmetric hyperbolic variables; they remain bounded, as do 
	their first derivatives in \textbf{all} directions.
\end{theorem}

\begin{remark}[\textbf{Rescaling the transversal derivative so as to ``cancel'' the blowup}]
	\label{R:RADPSIDOESNOTBLOWUP}
	We note already that a key part of the proof
	is showing the derivative of the transport variable in the transversal direction $\Rad$ also remains bounded.
	This does not contradict Theorem~\ref{T:EXTREMELYROUGH} for the following reason: 
	the vectorfield $\Rad$ is constructed so that its Cartesian components go
	to $0$ as the shock forms, in a manner that exactly compensates for the blowup of an ``order-unity-length''
	transversal derivative of the transport variable.
	Roughly, the situation can be described as follows, where $\Psi$ is the transport variable
	and the remaining quantities will be rigorously defined later in the article:
	$|\Radunit \Psi|$ blows up,\footnote{Here and throughout, 
if $Z$ is a vectorfield and $f$ is a scalar function, 
then $Z f := Z^{\alpha} \partial_{\alpha} f$ is the derivative of $f$ in the direction $Z$. \label{FN:SCALARVECTORFIELDIFFNOTATION}}
	$|\Rad \Psi|$ remains bounded,
	$\Rad = \upmu \Radunit$,
	and the weight $\upmu$ vanishes for the first time at the shock;
	one could say that $|\Radunit \Psi|$ blows up like $C/\upmu$ as $\upmu \downarrow 0$,
	where $C$ is the size of $|\Rad \Psi|$ at the shock;
	see Subsubsect.\ \ref{SSS:INTROSHOCKFORMS} for a more in-depth discussion of this point.
\end{remark}

\begin{remark}[\textbf{The heart of the proof and the kind of initial data under study}]
\label{R:HEARTOFPROOF}
The heart of the proof of Theorem~\ref{T:EXTREMELYROUGH} is to control the singular
terms and to show that the shock actually happens, 
i.e., that chaotic interactions do not prevent the shock from forming
or cause a more severe kind of singularity. 
In an effort to focus only on the singularity formation,
we have chosen to study the simplest non-trivial set of initial data to which 
our methods apply: perturbations of the data corresponding to simple plane symmetric waves
(see Subsect.\ \ref{SS:SIMPLEWAVS} for further discussion),
where we assume plentiful initial Sobolev regularity. The corresponding solutions do not
experience dispersion, so there are no time or radial weights in our estimates.
We will describe the initial data in more detail in Subsubsect.\ \ref{SSS:INTRODATABOOTSTRAPPOINTWISE}. 
\end{remark}

\begin{remark}[\textbf{Extensions to other kinds of hyperbolic subsystems}]
	\label{R:OTHERHYPERBOLICSUBSYSTEMS}
	From our proof, one can infer that the assumption of symmetric hyperbolicity for the subsystem 
	from Theorem~\ref{T:EXTREMELYROUGH} is in itself not important; 
	we therefore anticipate that similar shock formation results should hold for 
	systems comprising quasilinear transport equations coupled to many other types of hyperbolic subsystems,
	such as wave equations or regularly hyperbolic (in the sense of \cite{dC2000}) subsystems.
\end{remark}

\subsection{Paper outline}
\label{SS:PAPEROUTLINE}
\begin{itemize}
	\item In the remainder of Sect.\ \ref{S:INTRO},
		we give a more detailed description of our main results,
		summarize the main ideas behind the proofs,
		place our work in context by discussing prior works on shock formation,
		and summarize some of our notation.
	\item In Sect.\ \ref{S:SETUPOFPROBLEM}, 
		we precisely define the class of systems to which our main results apply.
	\item In Sect.\ \ref{S:GEOMETRICCONSTRUCTIONS}, we construct the majority of the
		geometric objects that play a role in our analysis. We also derive
		evolution equations for some of the geometric quantities.
	\item In Sect.\ \ref{S:ENERGYID}, we derive energy identities.
	\item In Sect.\ \ref{S:NUMDERIVSDATASIZEANDBOOTSTRAPASSUMPTIONS}, we
		state the number of derivatives that we use to close our estimates,
		state our size assumptions on the data,
		and state bootstrap assumptions that are useful for deriving estimates.
	\item In Sect.\ \ref{S:POINTWISEESTIMATESANDIMPROVEMENTOFAUX}, we derive
		pointwise estimates for solutions to the evolution equations and their derivatives,
		up to top order.
	\item In Sect.\ \ref{S:ESTIMATESFORCHOV}, we derive some properties of the change of variables map from
		geometric to Cartesian coordinates.
	\item In Sect.\ \ref{S:ENERGYESTIMATES}, which is the main section of the paper,
		we derive a priori estimates for all of the quantities under study.
	\item In Sect.\ \ref{S:CONTINUATION}, we provide some continuation criteria
		that, in the last section, we use to show that the solution survives up to the shock.
	\item In Sect.\ \ref{S:MAINTTHM}, we state and prove the main theorem.
\end{itemize}

\subsection{The role of nonlinear geometric optics in proving Theorem~\ref{T:EXTREMELYROUGH}}
\label{SS:NONLINEARGEOMERICOPTICS}
In prior stable shock formation results in more than one spatial dimension
(which we describe in Subsubsect.\ \ref{SSS:MULTIDRESULTS}),
the blowup occurred in a solution to a wave equation.
In the present work, the blowup occurs in the derivatives of the solution to the transport equation.
The difference is significant in that to obtain the sharp picture of shock formation,
one must rely on a geometric version of the vectorfield method 
that is precisely tailored to the family of characteristics whose intersection is tied to the blowup.
The key point is that the basic regularity properties of the characteristics
and the corresponding geometric vectorfields
are \emph{different in the wave equation and transport equation cases}; 
we will discuss this fundamental point in more detail below.
Although the blowup mechanism for solutions to the transport equations under study
is broadly similar to the Riccati-type 
mechanism that drives singularity formation in the simple one-space-dimensional example of Burgers' equation\footnote{The Riccati term appears after one spatial differentiation of the equation.} 
(see Subsect.\ \ref{SS:SIMPLEWAVS} for related discussion), 
the proof of our main theorem is much more complicated,
owing in part to the aforementioned difficulty of having to
derive energy estimates in multiple spatial dimensions.  

The overall strategy of our proof is to construct a system of geometric coordinates
adapted to the transport characteristics,
\emph{relative to which the solution remains smooth}, 
in part because the geometric coordinates ``hide''\footnote{In one spatial dimension, this is sometimes referred to
as ``straightening out the characteristics'' via a change of coordinates.} 
the Riccati-type term mentioned above. In more than one spatial dimension, 
the philosophy of constructing geometric coordinates to regularize the problem of shock formation
seems to have originated Alinhac's work \cites{sA1999a,sA1999b,sA2001b} 
on quasilinear wave equations; see Subsubsect.\ \ref{SSS:MULTIDRESULTS} for further discussion.
As will become abundantly clear, 
our construction of the geometric coordinates and other related 
quantities is tied to the following fundamental ingredient in our approach:
our development of a theory of 
\underline{nonlinear geometric optics for quasilinear transport equations}, tied to an eikonal function,
that is compatible with full quasilinear coupling to the symmetric hyperbolic subsystem.
We use nonlinear geometric optics to construct vectorfield differential operators adapted to the
characteristics as well as to detect the singularity formation. 
By ``compatible,'' we mean, especially, from the perspective of regularity considerations.
Indeed, in any situation in which one uses nonlinear geometric optics
to study a quasilinear hyperbolic PDE system,
\emph{one must ensure that the regularity of the corresponding eikonal function is consistent
with that of the solution}. By ``full quasilinear coupling,'' we mean that in the systems 
that we study, the
\emph{principal coefficients in all equations are allowed to depend on all solution variables}.

Upon introducing nonlinear geometric optics into the problem, 
we encounter the following key difficulty:
\begin{quote}
	Some of the geometric vectorfields that we construct have
	Cartesian components that are one degree less differentiable than the transport variable,
	as we explain in Subsubsect.\ \ref{SSS:REGULARITYCONSIDERATIONS}.
\end{quote}
On the one hand, due to the full quasilinear coupling,
it seems that we must use the geometric vectorfields when commuting 
the symmetric hyperbolic subsystem to obtain higher-order estimates;
this allows us to avoid generating uncontrollable commutator error terms 
involving ``bad derivatives'' (i.e., in directions transversal to the transport characteristics)
of the shock-forming transport variable.
On the other hand, the loss of regularity of the Cartesian components of the
geometric vectorfields leads,
at the top-order derivative level, 
to commutator error terms in the symmetric hyperbolic
subsystem that are uncontrollable in that they have insufficient regularity.
To overcome this difficulty, we employ the following strategy:
\begin{quote}
We never commute the symmetric hyperbolic subsystem a top-order number of times with
a pure string of geometric vectorfields; instead, 
we first commute it with a single Cartesian coordinate partial derivative, 
and then follow up the Cartesian derivative with commutations by the geometric vectorfields. 
\end{quote}
The above strategy allows us to avoid the loss of a derivative, 
but it generates commutator error terms depending on a single Cartesian coordinate partial derivative, 
which are dangerous because they are transversal to the transport characteristics. Indeed, the first 
Cartesian coordinate partial derivatives of the transport variable blow up at the shock.
Fortunately, by using a weight\footnote{The weight is the quantity $\upmu$ from Remark~\ref{R:RADPSIDOESNOTBLOWUP},
and we describe it in detail below.} 
adapted to the characteristics, 
we are able to control such error terms featuring a single Cartesian differentiation, 
all the way up to the singularity.

We close this subsection by providing some remarks on 
using nonlinear geometric optics to study the maximal development\footnote{The maximal development of the data is, roughly,
the largest possible classical solution that is uniquely determined by the data. 
Readers can consult \cites{jSb2016,wW2013} for further discussion.}  
of initial data for
quasilinear hyperbolic PDEs without symmetry assumptions.
The approach was pioneered by Christodoulou--Klainerman in their celebrated proof \cite{dCsK1993} of the stability of Minkowski
spacetime as a solution to the Einstein-vacuum equations.\footnote{Roughly, \cite{dCsK1993} is a small-data global
existence result for Einstein's equations.} Since perturbative global existence results for hyperbolic PDEs typically
feature estimates with ``room to spare,''
in many cases, it is possible to close the proofs by relying on a version of \emph{approximate} nonlinear geometric optics,
which features approximate eikonal functions whose level sets approximate the characteristics.
The advantage of using approximate eikonal functions is that 
is that their regularity theory is typically very simple.
For example, such an approach was taken by
Lindblad--Rodnianski in their proof of the stability of the Minkowski spacetime \cite{hLiR2010}
relative to wave coordinates. Their proof was less precise than Christodoulou--Klainerman's
but significantly shorter since,
unlike Christodoulou--Klainerman, 
Lindblad--Rodnianski relied on approximate eikonal functions 
whose level sets were standard Minkowski light cones.

The use of eikonal functions for proving shock formation for quasilinear wave equations
in more than one spatial dimension without symmetry assumptions
was pioneered by Alinhac in his aforementioned works \cites{sA1999a,sA1999b,sA2001b},
and his approach was later remarkably sharpened/extended by Christodoulou \cite{dC2007}.
In contrast to global existence problems, 
in proofs of shock formation without symmetry assumptions, 
the \emph{use of an eikonal function adapted to the true characteristics (as opposed to approximate ones) seems essential}, 
since the results yield that the singularity formation
exactly coincides with the intersection of the characteristics.
One can also draw an analogy between works on shock formation 
and works on low regularity well-posedness for quasilinear wave equations,
such as \cites{sKiR2003,sKiR2005d,hSdT2005,sKiRjS2015}, 
where the known proofs fundamentally rely on eikonal functions whose levels sets are true characteristics.

\subsection{A more precise statement of the main results}
\label{SS:ROUGHSTATEMENT}
For the systems under study,
we assume that the number of spatial dimensions is $n \geq 1$, where $n$ is arbitrary.
For convenience, we study the dynamics of solutions in spacetimes of the form
$\mathbb{R} \times \Sigma$, where 
\begin{align} \label{E:SPACEMANIFOLD}
	\Sigma & = \mathbb{R} \times \mathbb{T}^{n-1}
\end{align}
is the spatial manifold and $\mathbb{T}^{n-1}$ is the standard $n-1$ dimensional torus
(i.e., $[0,1)^{n-1}$ with the endpoints identified and equipped with the usual smooth orientation).
The factor $\mathbb{T}^{n-1}$ in \eqref{E:SPACEMANIFOLD} will correspond to perturbations away from plane symmetry.
Our assumption on the topology of $\Sigma$ is for technical convenience only;
since our results are localized in spacetime, one could derive similar stable blowup results for 
arbitrary spatial topology.\footnote{However, assumptions on the data that lead to shock formation
generally must be adapted to the spatial topology.}
Throughout, $\lbrace x^{\alpha} \rbrace_{\alpha = 0,\cdots,n}$
are a fixed set of Cartesian spacetime coordinates on $\mathbb{R} \times \Sigma$,
where $t := x^0 \in \mathbb{R}$ is the time coordinate,
$\lbrace x^i \rbrace_{i = 1,\cdots,n}$ are the spatial coordinates on $\Sigma$,
$x^1 \in \mathbb{R}$ is the ``non-compact space coordinate,'' 
and $\lbrace x^i \rbrace_{i = 2,\cdots,n}$ are standard locally
defined coordinates on $\mathbb{T}^{n-1}$ such that
$(\partial_2,\cdots,\partial_n)$ is a positively oriented frame. 
We denote the Cartesian coordinate partial
derivative vectorfields by 
$
\partial_{\alpha} := \frac{\partial}{\partial x^{\alpha}}
$,
and we sometimes use the alternate notation $\partial_t := \partial_0$.
Note that the vectorfields $\lbrace \partial_{\alpha} \rbrace_{\alpha = 0,\cdots,n}$
can be globally defined so as to form a smooth frame,
even though the $\lbrace x^i \rbrace_{i = 2,\cdots,n}$ 
are only locally defined.
For mathematical convenience, our main results are adapted to
\underline{nearly plane symmetric solutions}, where by our conventions,
exact plane symmetric solutions depend only on $t$ and $x^1$.
We now roughly summarize our main results; see Theorem~\ref{T:MAINTHM}
for precise statements.

\begin{theorem}[\textbf{Stable shock formation} (rough version)]
\label{T:ROUGHMAINTHM} \

\noindent \underline{\textbf{Assumptions}}:
	Consider the following coupled system\footnote{Throughout we use Einstein's summation convention.
	Greek lowercase ``spacetime'' indices vary over $0,1,\cdots,n$, 
	while Latin lowercase ``spatial'' indices vary over $1,2,\cdots,n$.} 
	with initial data
	posed on the constant-time hypersurface
	$\Sigma_0 := \lbrace 0 \rbrace \times \mathbb{R} \times \mathbb{T}^{n-1} \simeq \mathbb{R} \times \mathbb{T}^{n-1}$:
	\begin{align} 
	\Lunit^{\alpha}(\Psi,\noshock) \partial_{\alpha} \Psi & = 0,
		\label{E:INTROSHOCKEQN} \\
	A^{\alpha}(\Psi,\noshock) \partial_{\alpha} \noshock & = 0,
	\label{E:INTRONONSHOCKEQUATION} 
\end{align}
where $\Psi$ is a scalar function, 
$\noshock = (\noshockuparg{1},\cdots,\noshockuparg{M})$ is an array ($M$ is arbitrary),
and the $A^{\alpha}$ are symmetric $M \times M$ matrices.
Assume that $\Lunit^1(\Psi,\noshock)$ verifies a genuinely nonlinear-type condition tied to its dependence on $\Psi$
(specifically, condition \eqref{E:GENUINELYNONLINEAR})
and that for small $\Psi$ and $\noshock$, the constant-time hypersurfaces $\Sigma_t$ and the $\mathcal{P}_u$ are 
\textbf{spacelike}\footnote{This means that $A^{\alpha} \omega_{\alpha}$ is positive definite, where
the one-form $\omega$ is co-normal to the surface and satisfies $\omega_0 > 0$.} 
for the subsystem \eqref{E:INTRONONSHOCKEQUATION}.
Here and throughout, the $\mathcal{P}_u$ are $\Lunit$-characteristics,
which are the family of (solution-dependent) hypersurfaces 
equal to the level sets of the eikonal function $u$, that is,
the solution to the eikonal equation (see Footnote~\ref{FN:SCALARVECTORFIELDIFFNOTATION} regarding the notation)
$\Lunit u = 0$
with the initial condition
$u|_{\Sigma_0} = 1 - x^1$.

To close the proof, we make the following assumptions on the data, 
which we propagate all the way up to the singularity: 
\begin{quote}
$\bullet$ Along $\Sigma_0$,
$\noshock$, all of its derivatives,
and the $\mathcal{P}_u$-tangential derivatives of $\Psi$
are small \underline{relative}\footnote{We also assume an absolute smallness condition on 
$\| \Psi \|_{L^{\infty}(\Sigma_0)}$.} 
to quantities constructed out of
a first $\mathcal{P}_u$-transversal derivative of $\Psi$
(see Subsect.\ \ref{SS:SMALLNESSASSUMPTIONS} for the precise smallness assumptions, 
which involve geometric derivatives).
Moreover, along $\mathcal{P}_0$, all derivatives of $\noshock$
up to top order are relatively small.
\end{quote}

\medskip

\noindent \underline{\textbf{Conclusions}}: There exists an open set (relative to a suitable Sobolev topology)
of data that are close to the data of a simple plane wave
(where a simple plane wave is such that $\Psi = \Psi(t,x^1)$ and $\noshock \equiv 0$),
given along the unity-thickness subset $\Sigma_0^1$ of $\Sigma_0$ and a finite portion of $\mathcal{P}_0$, 
such that the solution behaves as follows:
\begin{quote}
$\bullet$ 
$\max_{\alpha=0,\cdots,n} |\partial_{\alpha} \Psi|$ 
blows up in finite time 
while $|\Psi|$, $\lbrace |\noshockuparg{J}| \rbrace_{1 \leq J \leq M}$, 
and $\lbrace |\partial_{\alpha} \noshockuparg{J}| \rbrace_{0 \leq \alpha \leq n, 1 \leq J \leq M}$
remain uniformly bounded.
\end{quote}
The blowup is tied to the intersection of the $\mathcal{P}_u$,
which in turn is precisely characterized by the vanishing of the inverse foliation density 
$\upmu := \frac{1}{\partial_t u}$
of the $\mathcal{P}_u$, which is initially near unity; 
see Fig.\ \ref{F:FRAME} for a picture in which a shock is about to form
(in the region up top, where $\upmu$ is small).
Moreover, one can complete $(t,u)$ to form a geometric coordinate system
$(t,u,\vartheta^2,\cdots,\vartheta^n)$ on spacetime with the following key property, central to the proof:
\begin{quote}
$\bullet$
No singularity occurs in $\Psi$, $\noshock$, $\partial_{\alpha} \noshock$,
or their derivatives with respect to the geometric coordinates\footnote{In practice, we will derive
estimates for the derivatives of the solution with respect to the vectorfields
depicted in Fig.\ \ref{F:FRAME}.}
up to top order.
\end{quote}
Put differently, the problem of shock formation can be transformed into an equivalent problem in which
one proves non-degenerate estimates relative to the geometric coordinates and, at the same time,
proves that the geometric coordinates degenerate in a precise fashion with respect to the
Cartesian coordinates as $\upmu \downarrow 0$.
\end{theorem}

\begin{center}
\begin{overpic}[scale=.25]{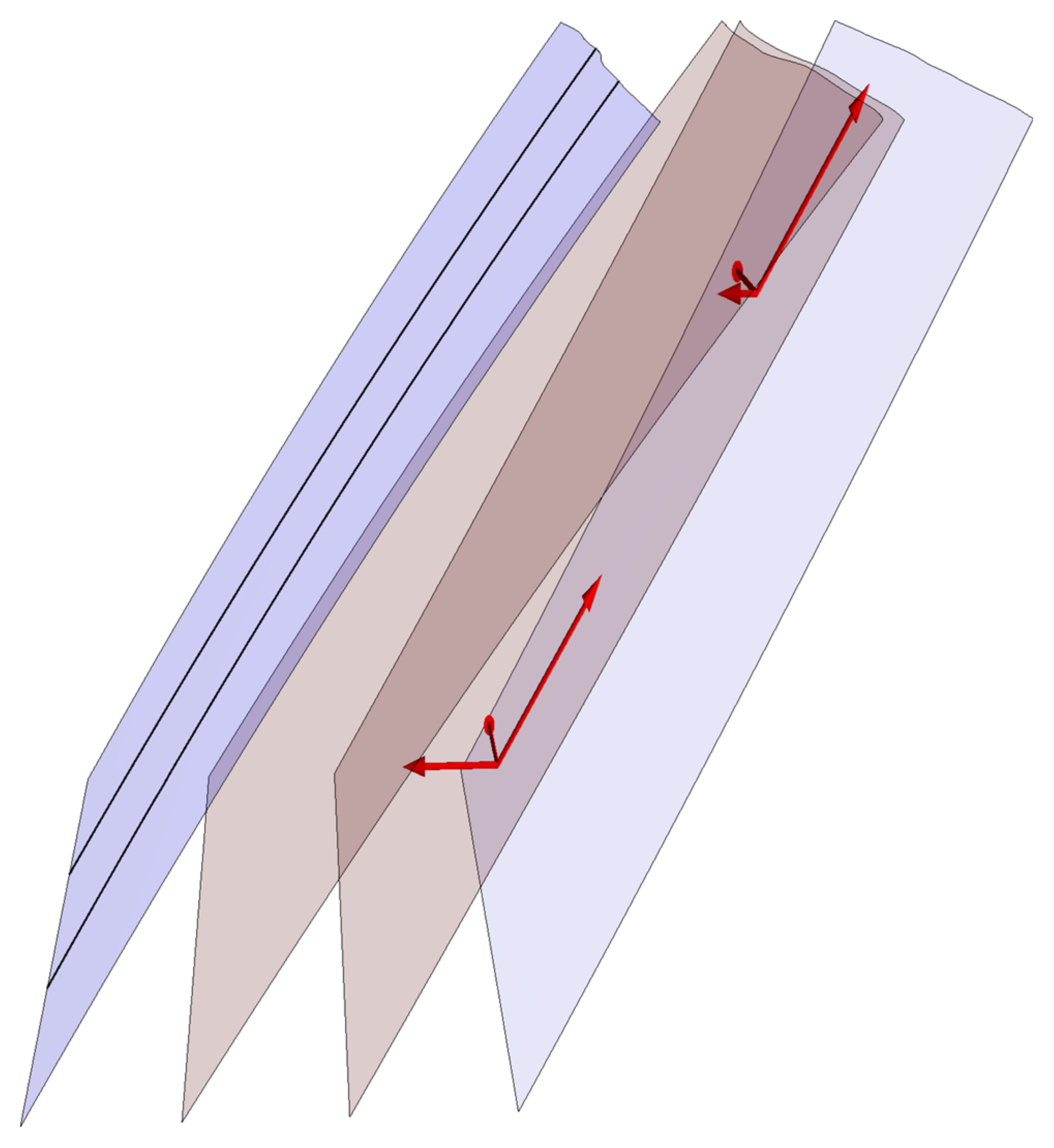} 
\put (71.5,82) {\large$\displaystyle \Lunit$}
\put (58,72) {\large$\displaystyle \Rad$}
\put (59.7,81.5) {\large$\displaystyle \frac{\partial}{\partial \vartheta^2}$}
\put (48,39.2) {\large$\displaystyle \Lunit$}
\put (30.5,31) {\large$\displaystyle \Rad$}
\put (38,42.3) {\large$\displaystyle \frac{\partial}{\partial \vartheta^2}$}
\put (49,14) {\large$\displaystyle \mathcal{P}_0^t$}
\put (34,14) {\large$\displaystyle \mathcal{P}_u^t$}
\put (6.5,14) {\large$\displaystyle \mathcal{P}_1^t$}
\put (20,20) {\large$\displaystyle \upmu \approx 1$}
\put (56,99) {\large$\displaystyle \upmu \ \mbox{\upshape small}$}
\put (12,35) {\rotatebox{57}{\large$\displaystyle \mbox{\upshape integral curves of } \Lunit$}}
%
\end{overpic}
\captionof{figure}{The dynamics until close to the time of the shock when $n=2$}
\label{F:FRAME}
\end{center}

\begin{remark}[\textbf{Non trivial interactions all the way up to the singularity}]
	We emphasize that in Theorem~\ref{T:ROUGHMAINTHM},
	$\noshock$ can be non-zero at the singularity in $\max_{\alpha=0,\cdots,n} |\partial_{\alpha} \Psi|$.
	This means, in particular, that the problem cannot be reduced to the study of
	blowup for the simple case of a decoupled scalar transport equation.
\end{remark}	

\begin{remark}[\textbf{Extensions to allow for semilinear terms}]
	\label{R:EXTENSIONSALLOWINGSEMILINEARTERMS}
	We expect that the results of Theorem~\ref{T:ROUGHMAINTHM} could be extended 
	to allow for the presence of arbitrary smooth semilinear terms 
	on RHSs~\eqref{E:INTROSHOCKEQN}-\eqref{E:INTRONONSHOCKEQUATION}
	that are functions of $(\Psi,\noshock)$. The extension would be straightforward to derive
	for semilinear terms that vanish when $\noshock = 0$ (for example, $\noshock \cdot \Psi$).
	The reason is that our main results imply that such semilinear terms 
	remain small, in suitable norms, up to the shock.
	In fact, such semilinear terms completely vanish for the exact simple waves 
	whose perturbations we treat in Theorem~\ref{T:ROUGHMAINTHM}; 
	see Subsect.\ \ref{SS:SIMPLEWAVS} for further discussion of simple waves.
	Consequently, a set of initial data similar to the one from Theorem~\ref{T:ROUGHMAINTHM} would also lead to the formation
	of a shock in the presence of such semilinear terms.
	In contrast, for semilinear terms that do not vanish when $\Psi = 0$ (for example, $\Psi^2$),
	the analysis would be more difficult and the assumptions on the data might have to be changed
	to produce shock-forming solutions. In particular, such semilinear terms can, 
	at least for data with $\Psi$ large, radically alter the behavior of some solutions.
	This can be seen in the simple model problem of the inhomogeneous
	Burgers-type equation $\partial_t \Psi + \Psi \partial_x \Psi = \Psi^2$. 
	This equation admits the family of ODE-type
	blowup solutions $\Psi_{(ODE);T}(t) := (T-t)^{-1}$, whose singularity is much more severe
	than the shocks that that typically form when the semilinear term $\Psi^2$ is absent.
\end{remark}

\begin{remark}[\textbf{Description of a portion of the maximal development}]
	\label{R:DESCRIPTIONOFMAXIMALDEVELOPMENT}
	We expect that the approach that we take 
	in proving our main theorem is
	precise enough that it can be extended to
	yield sharp information about
	the behavior of the solution up the boundary of the maximal development,
	as Christodoulou did in his related work \cite{dC2007}*{Chapter 15}
	(which we describe in Subsubsect.\ \ref{SSS:MULTIDRESULTS}).
	For brevity, we do not pursue this issue in the present article.
	However, in the detailed version of our main results (i.e., Theorem~\ref{T:MAINTHM}),
	we set the stage for the possible future study of the maximal development
	by proving a ``one-parameter family of results,''
	indexed by $U_0 \in (0,1]$; one would need to vary $U_0$
	to study the maximal development. Here and throughout, 
	$U_0$ corresponds to an initial data region $\Sigma_0^{U_0}$ of thickness $U_0$;
	see Fig.\ \ref{F:REGION} on pg.~\pageref{F:REGION}
	and Subsubsect.\ \ref{SSS:SPACETIMEREGION} for further discussion.
	For $U_0 = 1$, which is implicitly assumed in Theorem~\ref{T:ROUGHMAINTHM},
	a shock forms in the maximal development 
	of the data given along\footnote{Actually, as we explain in 
	Subsubsect.\ \ref{SSS:SPACETIMEREGION}, we only need to specify the data along the subset
	$\Sigma_0^{U_0} \cup \mathcal{P}_0^{2 \TranminusdatasizeWithFactor^{-1}}$
	of $\Sigma_0^{U_0} \cup \mathcal{P}_0$.} 
	$\Sigma_0^{U_0} \cup \mathcal{P}_0$.
	However, for small $U_0$, a shock does not necessarily form
	in the maximal development of the data given along
	$\Sigma_0^{U_0} \cup \mathcal{P}_0$
	within the amount of time that we attempt to control the solution.
\end{remark}

\subsection{Further discussion on simple plane symmetric waves}
\label{SS:SIMPLEWAVS}
Theorem~\ref{T:ROUGHMAINTHM} shows, roughly, that the well-known stable blowup of
$\partial_x \Psi$ in solutions to the one-space-dimensional Burgers' equation
\begin{align} \label{E:BURGERS1D}
	\partial_t + \Psi \partial_x \Psi
	& = 0
\end{align}
is stable under a full quasilinear coupling of \eqref{E:BURGERS1D} to other hyperbolic subsystems,
under perturbations of the coefficients in the transport equation,
and under increasing the number of spatial dimensions.
We now further explain what we mean by this.
A special case of Theorem~\ref{T:ROUGHMAINTHM} occurs when $\noshock \equiv 0$
and $\Psi$ depends only on $t$ and $x^1$ (plane symmetry).
In this simplified context, 
the blowup of $\max_{\alpha=0,1} |\partial_{\alpha} \Psi|$  
can be proved using a simple argument 
based on the method of characteristics, similar to the argument that is typically
use to prove blowup in the case of Burgers' equation.
Solutions with $\noshock \equiv 0$ are sometimes referred to as \emph{simple waves} since
they can be described by a single non-zero scalar component.
From this perspective, we see that Theorem~\ref{T:ROUGHMAINTHM}
yields the stability of simple plane wave blowup for the transport variable 
in solutions to the system \eqref{E:INTROSHOCKEQN}-\eqref{E:INTRONONSHOCKEQUATION}.

\subsection{The main new ideas behind the proof}
\label{SS:MAINNEWIDEASBEHINDPROOF}
The proof of Theorem~\ref{T:ROUGHMAINTHM} is based in part on ideas used in earlier
works on shock formation in more than one spatial dimension.
We review these works in Subsect.\ \ref{SS:PRIORWORKANDCONTEXT}.
Here we summarize the two most novel aspects behind the proof of Theorem~\ref{T:ROUGHMAINTHM}.

\begin{itemize}
	\item (\textbf{Nonlinear geometric optics for transport equations})	
		As in all prior shock formation results in more than one spatial dimension,
		our proof relies on nonlinear geometric optics, that is, 
		the eikonal function $u$. The use of an eikonal function 
		is essentially the method of characteristics
		implemented in more than one spatial dimension.
		All of the prior works were such that
		the blowup occurred in a solution to a quasilinear wave equation and thus the theory of nonlinear
		geometric optics was adapted to those wave characteristics. In this article,
		we advance the theory of nonlinear geometric optics for transport equations. Although the theory
		is simpler in some ways, compared to the case of wave equations,
		it is also more degenerate in the following sense: 
		\emph{the regularity theory for the eikonal function $u$ is less favorable
		in that $u$ is one degree less differentiable in some directions
		compared to the case of wave equations}.
		We therefore must close the proof of Theorem~\ref{T:ROUGHMAINTHM} under this decreased differentiability.
		We defer further discussion of this point until Subsubsect.\ \ref{SSS:REGULARITYCONSIDERATIONS}.
		Here, we will simply further motivate our use of nonlinear geometric optics in proving shock formation.
		
		First, we note that in more than one spatial dimension,
		it does not seem possible to close the proof using only the Cartesian coordinates;
		indeed, Theorem~\ref{T:ROUGHMAINTHM} shows that the blowup of $\Psi$ 
		precisely corresponds to the vanishing of the inverse foliation density of the characteristics,
		which is equivalent to the blowup of $\partial_t u$. Hence,
		it is difficult to imagine how a sharp, constructive proof of stable blowup
		would work without referencing an eikonal function.
		In view of these considerations, we construct a geometric coordinate system $(t,u,\vartheta^2,\cdots,\vartheta^n)$
		adapted to the transport operator vectorfield $\Lunit$ and prove that 
		$\Psi$, $\noshock$, $\diffnoshockdownarg{\alpha} := \partial_{\alpha} \noshock$,
		and their geometric coordinate partial
		derivatives remain regular all the way up to the singularity in $\max_{\alpha=0,\cdots,n} |\partial_{\alpha} \Psi|$. 
		The blowup of $\max_{\alpha=0,\cdots,n} |\partial_{\alpha} \Psi|$ occurs because
		the change of variables map between geometric and Cartesian coordinates \emph{degenerates}, 
		which is in turn tied to the vanishing of $\upmu$;
		the Jacobian determinant of this map is in fact proportional to $\upmu$; see Lemma~\ref{L:CHOVCALCULATIONS}.
		The coordinate $t$ is the standard Cartesian time function.
		The geometric coordinate function $u$ is the eikonal function described in Theorem~\ref{T:ROUGHMAINTHM}.
		The initial condition 
		$u|_{\Sigma_0} = 1 - x^1$
		is adapted to the approximate plane symmetry of the initial data.
		We similarly construct the ``geometric torus coordinates'' 
		$\lbrace \vartheta^j \rbrace_{j=2,\cdots,n}$ by solving 
		$\Lunit \vartheta^j = 0$ with the initial condition 
		$\vartheta^j|_{\Sigma_0} = x^j$.
		The main challenge is to derive regular estimates 
		relative to the geometric coordinates
		for all quantities, including the solution variables
		and quantities constructed out of the geometric coordinates.
	\item (\textbf{Full quasilinear coupling})
		Because we are able to close the proof with decreased regularity for $u$
		(compared to the case of wave equations),
		we are able to handle full quasilinear coupling between all solution variables.
		This is an interesting advancement over prior works, 
		where the principal coefficients in the evolution equation for the shock-forming variable
		were allowed to depend only on the shock-forming variable itself and on other solution
		variables that satisfy a wave equation with the \emph{same principal part} as the shock-forming variable;
		i.e., in equation \eqref{E:INTROSHOCKEQN}, we allow
		$L^{\alpha} = L^{\alpha}(\Psi,\noshock)$,
		where the principal part of the evolution equation
		\eqref{E:INTRONONSHOCKEQUATION} 
		for $\noshock$ is distinct (by assumption) from $\Lunit$.
	\end{itemize}

\subsection{A more detailed overview of the proof}
\label{SS:PROOFOVERVIEW}
In this subsection, we provide an overview of the proof of our main results.
Our analysis is based in part on some key ideas originating in earlier works,
which we review in Subsect.\ \ref{SS:PRIORWORKANDCONTEXT}.
Our discussion in this subsection is, at times, somewhat loose; 
our rigorous analysis begins in Sect.\ \ref{S:SETUPOFPROBLEM}.

\subsubsection{Setup and geometric constructions}
\label{SSS:INTROSETUPANDCONSTRUCTIONS}
In Sects.\ \ref{S:SETUPOFPROBLEM}-\ref{S:GEOMETRICCONSTRUCTIONS}, we construct the
geometric coordinate system $(t,u,\vartheta^2,\cdots,\vartheta^n)$ described in Subsect.\ \ref{SS:NONLINEARGEOMERICOPTICS},
which is central for all that follows. We also construct
many related geometric objects, 
including the inverse foliation density $\upmu$ 
(see Def.\ \ref{D:MUDEF} for the precise definition)
of the characteristics $\mathcal{P}_u$
of the eikonal function $u$.
As we mentioned earlier, our overall strategy
is to show that the solution remains regular with respect to 
the geometric coordinates, all the way up to the top derivative level,
to show that $\upmu$ vanishes in finite time, 
and to show that the vanishing of $\upmu$
is exactly tied to
the blowup of $\max_{\alpha=0,\cdots,n} |\partial_{\alpha} \Psi|$.
It turns out that when deriving estimates, 
it is important to replace the geometric coordinate partial derivative vectorfield
$
\frac{\partial}{\partial u}
$
with a $\Sigma_t$-tangent vectorfield that we denote by $\Rad$,
which is similar to
$
\frac{\partial}{\partial u}
$
but generally not parallel to it;
see Fig.\ \ref{F:FRAME} for a picture of $\Rad$.
In the context of the present paper, 
the main advantage of $\Rad$ is that it
enjoys the following key property:
the vectorfield $\Radunit = \upmu^{-1} \Rad$
has Cartesian components that remain uniformly bounded, 
all the way up to the shock. 
Put differently, we have
$\Rad = \upmu \Radunit$, where we will show that
$\Radunit$ is a vectorfield of order-unity Euclidean length
and thus the Euclidean length of $\Rad$ is $\mathcal{O}(\upmu)$.
We further explain the significance of this in Subsubsect.\ \ref{SSS:INTROSHOCKFORMS},
when we outline the proof that the shock forms.
In total, when deriving estimates for the derivatives of quantities, 
we differentiate them with respect to elements of the vectorfield frame
\begin{align} \label{E:FRAMEINTRO}
\Fullset
:=
\lbrace
\Lunit,
\Rad,
\CoordAng{2},
\cdots,
\CoordAng{n}
\rbrace,
\end{align}
which spans the tangent space at each point with $\upmu > 0$.
Here, $\CoordAng{i} := \frac{\partial}{\partial \vartheta^i}$,
(where the partial differentiation is with respect to the geometric coordinates),
$\Lunit$ is the vectorfield from \eqref{E:INTROSHOCKEQN}
and, by construction, we have
$
\Lunit = \frac{\partial}{\partial t}
$
(see \eqref{E:LISDDT}).
The vectorfields $\Lunit$ and $\CoordAng{i}$ are tangent to the $\mathcal{P}_u$,
while $\Rad$ is transversal and normalized by $\Rad u = 1$ (see \eqref{E:RADAPPLIEDTOUISONE});
see Fig.\ \ref{F:FRAME} on pg.~\pageref{F:FRAME} for a picture of the frame.
Note that since $\Rad$ is of length $\mathcal{O}(\upmu)$, the uniform boundedness of 
$|\Rad \Psi|$ is consistent with the formation of
a singularity in the Cartesian coordinate partial derivatives of $\Psi$ 
when $\upmu \downarrow 0$; see Subsubsect.\ \ref{SSS:INTROSHOCKFORMS} for further discussion of this point.

We now highlight a crucial ingredient in our proof:
we treat the Cartesian coordinate partial derivatives of $\noshockuparg{J}$
as independent unknowns $\diffnoshockdoublearg{\alpha}{J}$, defined by
\begin{align} \label{E:INTRODIFFERENTIATEDNOSHOCK}
\diffnoshockdoublearg{\alpha}{J}
& 
:= \partial_{\alpha} \noshockuparg{J}.
\end{align}
As we stressed already in Subsect.\ \ref{SS:NONLINEARGEOMERICOPTICS},
our reliance on
$\diffnoshockdoublearg{\alpha}{J}$ 
\emph{allows us to avoid commuting equation
\eqref{E:INTRONONSHOCKEQUATION} up to top order with elements of} $\Fullset$,
which allows us to avoid certain top-order commutator terms that would result in the loss of a derivative.
Moreover, as we noted in
Theorem~\ref{T:ROUGHMAINTHM}, 
a key aspect of our framework is to
show that the quantities $\diffnoshockdownarg{\alpha}$ remain bounded up to the singularity
in $\max_{\alpha=0,\cdots,n} |\partial_{\alpha} \Psi|$.
To achieve this, we will control $\diffnoshockdownarg{\alpha}$
by studying its evolution equation
$A^{\beta} \partial_{\beta} \diffnoshockdownarg{\alpha} 
= - (\partial_{\alpha} A^{\beta}) \diffnoshockdownarg{\beta}
$,
whose inhomogeneous terms
are controllable under the scope of our approach.

\subsubsection{A more precise description of the spacetime regions under study}
\label{SSS:SPACETIMEREGION}
For convenience, we study only the future portion of the solution 
that is completely determined by the data lying in the subset
$\Sigma_0^{U_0} \subset \Sigma_0$
of thickness $U_0$
and on a portion of the characteristic
$\mathcal{P}_0$,
where
$0 < U_0 \leq 1$
is a parameter, fixed until Theorem~\ref{T:MAINTHM}; see Fig.\ \ref{F:REGION} on pg.~\pageref{F:REGION}.
We will study spacetime regions such that $0 \leq u \leq U_0$,
where $u$ is the eikonal function described above.
We have introduced the parameter $U_0$ because one would need to allow $U_0$ to vary 
in order to study the behavior of the solution up the boundary of the maximal development,
as we mentioned in Remark~\ref{R:DESCRIPTIONOFMAXIMALDEVELOPMENT}.

In our analysis, we will use a bootstrap argument in which we
only consider times $t$
with $0 \leq t < 2 \TranminusdatasizeWithFactor^{-1}$, where
$\TranminusdatasizeWithFactor > 0$ is a data-dependent parameter
described in Subsubsect.\ \ref{SSS:INTRODATABOOTSTRAPPOINTWISE}
(see also Def.\ \ref{D:SHOCKFORMATIONQUANTITY}).
Our main theorem shows that if $U_0 = 1$, 
then a shock forms at a time equal to a small perturbation of
$\TranminusdatasizeWithFactor^{-1}$; see Subsubsect.\ \ref{SSS:INTROSHOCKFORMS}
for an outline of the proof.
For this reason, in proving our main results,
we only take into account only the portion
of the data lying in $\Sigma_0^{U_0}$ and
in the subset $\mathcal{P}_0^{2 \TranminusdatasizeWithFactor^{-1}}$
of the characteristic $\mathcal{P}_0$; from domain
of dependence considerations, one can infer that 
only this portion can influence
the solution in the regions under study.

\begin{remark}
For the remainder of Subsect.\ \ref{SS:PROOFOVERVIEW}, 
we will suppress further discussion of $U_0$ by setting $U_0 = 1$.
\end{remark}

\subsubsection{Data-size assumptions, bootstrap assumptions, and pointwise estimates}
\label{SSS:INTRODATABOOTSTRAPPOINTWISE}
In Sect.\ \ref{S:NUMDERIVSDATASIZEANDBOOTSTRAPASSUMPTIONS}, we state our assumptions on the data
and formulate bootstrap assumptions that are useful for deriving estimates.
Our assumptions on the data involve the parameters
$\Psiep > 0$,
$\mathring{\upepsilon} \geq 0$, 
$\Trandatasize > 0$, 
and 
$\TranminusdatasizeWithFactor > 0$,
where, for our proofs to close, 
$\Psiep$ must be chosen to be small in an absolute sense and
$\mathring{\upepsilon}$ must be chosen to be small
in a \emph{relative sense} 
compared to
$\Trandatasize^{-1}$
and 
$\TranminusdatasizeWithFactor$
(see Subsect.\ \ref{SS:SMALLNESSASSUMPTIONS} for a precise description of the required smallness).
The following remarks capture the main ideas behind the data-size parameters.
\begin{enumerate}
	\item $\Psiep = \| \Psi \|_{L^{\infty}(\Sigma_0^1)}$ is the size of $\Psi$.
	\item $\mathring{\upepsilon}$ is the size, in appropriate norms,
		of the derivatives of $\Psi$ up to top order in which 
		\emph{at least one $\mathcal{P}_u$-tangential differentiation occurs},
		and of $\noshock$, $\diffnoshock$
		and \emph{all of their derivatives} up to top order with respect to elements of $\Fullset$ from \eqref{E:FRAMEINTRO}.
		We emphasize that we will study perturbations of plane symmetric shock-forming solutions such that
		$\mathring{\upepsilon} = 0$.
		That is, the case $\mathring{\upepsilon} = 0$ corresponds to a plane symmetric simple wave in which $\noshock \equiv 0$.
		We state the total number of derivatives that we use to close the estimates
		in Subsect.\ \ref{SS:NUMBEROFDERIVATIVES} and Subsubsect.\ \ref{SSS:DATA}.
		We also highlight that to close our proof, we never need to differentiate any quantity
		with more than one copy of the $\mathcal{P}_u$-transversal vectorfield $\Rad$.
		This is possible in part because of the following crucial fact,
		proved in Lemma~\ref{L:VECTORFIELDCOMMUTATORIDENTITIES}:
		commuting the elements of the frame $\Fullset$ with each other
		yields a vectorfield belonging to $\mbox{\upshape span} \lbrace \CoordAng{2},\cdots,\CoordAng{n} \rbrace$.
	\item $\Trandatasize = \| \Rad \Psi \|_{L^{\infty}(\Sigma_0^1)}$ 
		is the size of the $\mathcal{P}_u$-transversal derivative of $\Psi$.
	\item $\TranminusdatasizeWithFactor = \sup_{\Sigma_0^1} [\blowupcoeff \Rad \Psi]_-$,
		is a modified measure of the size of the $\mathcal{P}_u$-transversal derivative of $\Psi$,
		where $\blowupcoeff \neq 0$ is a coefficient determined by the nonlinearities
		and $[f]_- := |\min{f,0}|$.
	\item When $t = 0$, other geometric quantities that we use in studying solutions
			obey similar size estimates, where any
			differentiation of a quantity 
			with respect to a $\mathcal{P}_u$-tangential vectorfield
			leads to $\mathcal{O}(\mathring{\upepsilon})$-smallness;
			see Lemma~\ref{L:ESTIMATESFORNONLINEARGEOMETRICOPTICS}.
			A crucial exception occurs for $\Lunit \upmu$, 
			which initially is of relatively large size $\mathcal{O}(\Trandatasize)$
			in view of its evolution equation
			$\Lunit \upmu \sim \Rad \Psi + \cdots$
			(see \eqref{E:LUNITUPMU} for the precise evolution equation).
	\item The relative smallness of $\mathring{\upepsilon}$
				corresponds to initial data that are close to that of a simple plane symmetric wave,
				as we described in Subsect.\ \ref{SS:SIMPLEWAVS}.
\end{enumerate}

One of the main steps in our analysis is to propagate the above size assumptions
all the way up to the shock.
To this end, 
on a region of the form $(t,u,\vartheta) \in [0,T_{(Boot)}) \times [0,U_0] \times \mathbb{T}^{n-1}$,
we make $L^{\infty}$-type bootstrap assumptions
that capture the expectation that the above size assumptions hold.
In particular, \emph{the bootstrap assumptions capture our expectation that no singularity will form
in any quantity relative to the geometric coordinates}.
Moreover, since $\diffnoshockdoublearg{\alpha}{J} = \partial_{\alpha} \noshockuparg{J}$,
the bootstrap assumptions for the smallness\footnote{We note that the bootstrap assumptions refer to a parameter $\varepsilon > 0$ that, in our main theorem,
we will show is controlled by $\mathring{\upepsilon}$; for brevity, we will avoid further discussion of
$\varepsilon$ until Subsubsect.\ \ref{SSS:FUNDAMENTALBOOT}. \label{FN:BOOTSTRAPEPSILON}} 
of $\diffnoshock$
\emph{capture our expectation that the Cartesian coordinate partial derivatives of $\noshock$ should remain bounded};
indeed, this is a key aspect of our proof that we use to control various error terms depending on $\diffnoshock$.
As we mentioned earlier, a crucial point is that
we have set the problem up so that the shock forms
at time  
$T_{(Lifespan)}
	< 
	2 \TranminusdatasizeWithFactor^{-1}
$.
Therefore, we make the assumption
\begin{align} \label{E:INTROTBOOTSIZEASSUMPTION}
	0 < T_{(Boot)} < 2 \TranminusdatasizeWithFactor^{-1},
\end{align}
which leaves us with ample margin of error to show that a shock forms.
In particular, in view of \eqref{E:INTROTBOOTSIZEASSUMPTION},
we can bound factors of
$t$, $\exp(t)$,
etc.\ by a constant $C > 0$ depending on $\TranminusdatasizeWithFactor^{-1}$,
and the estimates will
close as long as $\mathring{\upepsilon}$ is sufficiently small;
see Subsect.\ \ref{SS:NOTATIONANDINDEXCONVENTIONS} for further discussion
on our conventions regarding the dependence of constants $C$.

In Sect.\ \ref{S:POINTWISEESTIMATESANDIMPROVEMENTOFAUX}, 
with the help of the bootstrap assumptions and data-size assumptions described above,
we commute all evolution equations, including
\eqref{E:INTROSHOCKEQN}-\eqref{E:INTRONONSHOCKEQUATION} and evolution equations for
$\upmu$ and related geometric quantities,
with elements of the $\Fullset$ up to top order and derive pointwise estimates
for the error terms. Actually,
due to the special structures of the equations relative to the geometric coordinates,
\emph{we never need to commute the evolution equations
verified by $\noshock$, $\diffnoshock$, or $\upmu$ with the transversal vectorfield $\Rad$}.
Moreover, for the other geometric quantities, 
we need to commute their evolution equations \emph{at most once} with $\Rad$.
We clarify, however, that we commute all equations many times with the elements of the
$\mathcal{P}_u$-tangential subset
$
\Tanset
:=
\lbrace
\Lunit,
\CoordAng{2},
\cdots,
\CoordAng{n}
\rbrace
$.

\subsubsection{Sketch of the formation of the shock}
\label{SSS:INTROSHOCKFORMS}
Let us assume that the bootstrap assumptions and pointwise estimates described in Subsubsect.\ \ref{SSS:INTRODATABOOTSTRAPPOINTWISE}
hold for a sufficiently long amount of time. We will sketch how they can be used to give a simple proof
of shock formation, that is, that $\upmu \downarrow 0$ and $\partial \Psi$ blows up. 
The main estimates in this regard are provided by Lemma~\ref{L:CRUCIALESTIMATESFORUPMUANDRADPSI};
here we sketch them.
First, using equation \eqref{E:LUNITUPMU},
the bootstrap assumptions, and the pointwise estimates,
we deduce the following evolution equation for the inverse foliation density:
$\Lunit \upmu 
= 
[\blowupcoeff \Rad \Psi](t,u,\vartheta)
+ 
\cdots
$, where the ``blowup coefficient'' $\blowupcoeff \neq 0$ was described in Subsubsect.\ \ref{SSS:INTRODATABOOTSTRAPPOINTWISE}
and $\cdots$ denotes small error terms, which we ignore here.
Next, we note the following pointwise estimate, which falls under the scope of the discussion in 
Subsubsect.\ \ref{SSS:INTRODATABOOTSTRAPPOINTWISE}:
$\Lunit (\blowupcoeff \Rad \Psi) = \cdots$ 
(smallness is gained since $\Lunit$ is a $\mathcal{P}_u$-tangential differentiation).
Recalling that $\Lunit = \frac{\partial}{\partial t}$,
we use the fundamental theorem of calculus to 
deduce $[\blowupcoeff \Rad \Psi](t,u,\vartheta) = [\blowupcoeff \Rad \Psi](0,u,\vartheta) + \cdots$.
Inserting this estimate into the above one for $\Lunit \upmu$,
we obtain
$\Lunit \upmu(t,u,\vartheta)
= 
[\blowupcoeff \Rad \Psi](0,u,\vartheta)
+ 
\cdots
$.
From the fundamental theorem of calculus and the initial condition 
$\upmu(0,u,\vartheta) = 1 + \cdots$, we obtain
$\upmu(t,u,\vartheta) = 1 + t [\blowupcoeff \Rad \Psi](0,u,\vartheta) + \cdots$.
From this estimate and the definition of $\TranminusdatasizeWithFactor$, we obtain
$\min_{(u,\vartheta) \in [0,1] \times \mathbb{T}^{n-1}} \upmu(t,u,\vartheta) = 1 - t \TranminusdatasizeWithFactor + \cdots$.
Hence, $\upmu$ vanishes for the first time at $T_{(Lifespan)} = \TranminusdatasizeWithFactor^{-1} + \cdots$,
as desired. Moreover, the above reasoning can easily be extended to show that
$|\Rad \Psi|(t,u,\vartheta) \gtrsim 1$ at any point $(t,u,\vartheta)$ 
such that $\upmu(t,u,\vartheta) < 1/4$.
Recalling that $\Rad = \upmu \Radunit$ where $\Radunit$ has order-unity Euclidean length,
we see the following: 
\begin{quote}
$|\Radunit \Psi|$ must blow up like $C/\upmu$ as $\upmu \downarrow 0$. 
\end{quote}
This argument shows,
in particular, that the vanishing of $\upmu$ \emph{exactly coincides with the blowup of}
$\max_{\alpha=0,\cdots,n} |\partial_{\alpha} \Psi|$.

\subsubsection{Considerations of regularity}
\label{SSS:REGULARITYCONSIDERATIONS}
This subsubsection is an interlude in which we highlight some issues tied 
to considerations of regularity.
Our discussion will distinguish the problem of shock formation for transport equations
from the (by now) well-understood case of wave equations, 
which we further describe in Subsubsect.\ \ref{SSS:MULTIDRESULTS}.
To illustrate the issues, we will highlight some features of our analysis,
with a focus on derivative counts.
In Lemma~\ref{L:COORDANGCOMPONENTRECTANGULAR}, we derive the following evolution
equation for the Cartesian components of $\CoordAng{i}$:
$\Lunit \CoordAngcomp{i}{j} = \CoordAng{i} \Lunit^j$,
where $\CoordAng{i} = \frac{\partial}{\partial \vartheta^i}$.
Recalling that
$
\Lunit = \frac{\partial}{\partial t}
$,
that
$
\diffnoshockdoublearg{\alpha}{J}
= \partial_{\alpha} \noshockuparg{J}
$,
and that
$\Lunit^j$ is a smooth function of $(\Psi,\noshock)$,
we infer, from standard energy estimates for transport equations,
that $\CoordAngcomp{i}{j}$ should have the same degree of Sobolev differentiability
as $\partial \Psi$ and $\diffnoshock$. In particular, we expect that
$\CoordAngcomp{i}{j}$ should be \emph{one degree less differentiable} than $\Psi$.
For similar reasons, $\upmu$, $\diffnoshock$,
and some other geometric quantities that play a role in our analysis are also
one degree less differentiable than $\Psi$. The following point is crucial for our approach:
\begin{quote}
We are able to close the energy estimates for $\Psi$ up to top order
even though, upon commuting $\Psi$'s transport equation, we
generate error terms that depend on the ``less differentiable'' quantities.
\end{quote}
That is, in controlling $\Psi$, we must carefully ensure that all error terms
feature an allowable amount of regularity. Moreover, 
the same careful care must be taken throughout the paper, 
by which we mean that we must ensure 
that we can close the estimates for all quantities 
using a consistent number of derivatives.
In particular, we stress that it is precisely 
due to considerations of the regularity of the Cartesian components of $\CoordAng{i}$ and $\Rad$ that 
we have introduced the quantities
$\diffnoshockdoublearg{\alpha}{J}
= \partial_{\alpha} \noshockuparg{J}$,
as we explained in Subsubsect.\ \ref{SSS:INTROSETUPANDCONSTRUCTIONS}.

In the case of wave equations, the derivative counts are different.
For example, the inverse foliation density $\upmu$
enjoys the \emph{same} Sobolev regularity
as the wave equation solution variable in directions tangent to the characteristics, 
a gain of one tangential derivative compared to the present work. 
For wave equations, a similar gain in tangential differentiability 
also holds for some other key geometric objects,
which we will not describe here.
The gain is available because certain special combinations of 
quantities constructed out of the eikonal equation
and the wave equation solution variable satisfy an unexpectedly good evolution equation,
with source terms that have better than expected regularity; 
see Subsubsect.\ \ref{SSS:MULTIDRESULTS}
or the survey article \cite{gHsKjSwW2016}
for further discussion.
Moreover, this gain seems \emph{essential} for closing some of the top-order energy estimates
in the wave equation case, the reason being that one must commute the geometric vectorfields
through the \emph{second-order} wave operator, which eats up the gain.
As we explain in Subsubsect.\ \ref{SSS:MULTIDRESULTS}, one pays a steep price in gaining back the derivative:
the resulting energy estimates allow for possible energy blowup at the high derivative levels,
a difficulty which we do not encounter in the present work.

We close this subsubsection by again highlighting that we are able to handle systems with full 
quasilinear coupling (in the sense explained in the second paragraph of Subsect.\ \ref{SS:NONLINEARGEOMERICOPTICS})
precisely because we are able to close our estimates
using geometric quantities that are one degree less differentiable than $\Psi$.
In contrast, the special combinations of quantities mentioned in the previous paragraph,
which are needed to close the wave equation energy estimates, 
seem to be unstable under a full quasilinear coupling of multiple speed wave systems.
Here is one representative manifestation of this issue: 
the problem of multi-space-dimensional shock formation for
covariant wave equation systems (see Footnote~\ref{FN:COVWAVOP} on pg.~\pageref{FN:COVWAVOP} regarding the notation)
of the form
\begin{align*}
	\square_{g_1(\Psi_1,\Psi_2)} \Psi_1 
	& = 0,
		\\
	\square_{g_2(\Psi_1,\Psi_2)} \Psi_2 
	& = 0
\end{align*}
is open whenever $g_1 \neq g_2$, even though
shock formation for systems with $g_1 = g_2$ and for
scalar equations
$\square_{g(\Psi)} \Psi = 0$
is well-understood.

\subsubsection{Energy estimates}
\label{SSS:INTROENERGYESTIMATES}
In Sect.\ \ref{S:ENERGYESTIMATES}, we derive the main technical estimates of the article:
energy estimates up to top order for $\Psi$, $\noshock$, $\diffnoshock$, $\upmu$, and related geometric quantities.
Energy estimates are an essential ingredient in the basic regularity theory of quasilinear hyperbolic systems in multiple spatial dimensions,
and in this article, they are also important because they yield improvements of our bootstrap assumptions
described in Subsubsect.\ \ref{SSS:INTRODATABOOTSTRAPPOINTWISE}.
We now describe the energies, which we construct in Sect.\ \ref{S:ENERGYID}.
To control the transport variable $\Psi$,
we construct geometric energies along $\Sigma_t$.
To control the symmetric hyperbolic variables $\noshock$ and $\diffnoshock$,
we construct $\upmu$-weighted energies along $\Sigma_t$
as well as \emph{non-$\upmu$-weighted} energies along the characteristics $\mathcal{P}_u$.
With $\Sigma_t^u$ defined to be the subset of $\Sigma_t$ in which the eikonal function takes on values in between $0$ and $u$ 
and $\mathcal{P}_u^t$ defined to be the subset of $\mathcal{P}_u$ corresponding to times in between $0$ and $t$,
we have, with 
$\Singletan \in 
\Tanset
=
\lbrace
\Lunit,
\CoordAng{2},
\cdots,
\CoordAng{n}
\rbrace
$,
\begin{subequations}
\begin{align} \label{E:INTROSHOCKENERGY}
	\shocken[\Singletan \Psi](t,u)
		& := 
		\int_{\Sigma_t^u} 
			(\Singletan \Psi)^2
		\, d \torusvol du',
	&&
			\\
	\noshocken[\noshock](t,u)
		& \approx 
		\int_{\Sigma_t^u} 
			\upmu |\noshock|^2
		d \torusvol du',
	&
	\noshockfl[\noshock](t,u)
	 & \approx
		\int_{\mathcal{P}_u^t} 
			|\noshock|^2 
		\, d \torusvol dt',
		\label{E:INTRONOSHOCKENERGYFLUX}
			\\
	\noshocken[\diffnoshock](t,u)
		& \approx 
		\int_{\Sigma_t^u} 
			\upmu |\diffnoshock|^2
		d \torusvol du',
	&
	\noshockfl[\diffnoshock](t,u)
	 & \approx
		\int_{\mathcal{P}_u^t} 
			|\noshock|^2 
		\, d \torusvol dt'.
		\label{E:INTRODIFFNOSHOCKENERGYFLUX}
\end{align}
\end{subequations}
In our analysis, we of course must also control various higher-order energies, but
here we ignore this issue.
The degenerate $\upmu$ weights featured in
$\noshocken[\noshock]$ and $\noshocken[\diffnoshock]$ 
arise from expressing the standard energy
for symmetric hyperbolic systems in terms of the geometric coordinates.
For controlling certain error integrals that arise in the energy identities, 
\emph{it is crucial that the characteristic fluxes $\noshockfl[\noshock]$ and $\noshockfl[\diffnoshock]$
do not feature any degenerate $\upmu$ weight}.
These characteristic fluxes are positive definite only because our structural assumptions on the equations
ensure that the
propagation speed of $\noshock$ and $\diffnoshock$ is strictly slower than that
of $\Psi$ (see \eqref{E:POSITIVEDEFMATRICES} for the precise assumptions).
Readers can consult Lemma~\ref{L:COERCIVENESSOFNOSHOCKENERGY} and its proof to better understand
the role of these assumptions.

We now outline the derivation of the energy estimates; see Sect.\ \ref{S:ENERGYESTIMATES} for precise
statements and proofs.
Let us define\footnote{Our definition of $\mathbb{W}(t,u)$ given here is schematic. 
See Def.\ \ref{D:MAINCOERCIVEQUANT} for the precise definition of the controlling quantity,
which we denote by $\totmax(t,u)$. \label{FN:L2CONTROLLING}}
the controlling quantity  
$\mathbb{W}(t,u)$ 
to be the sum of the terms on LHSs~\eqref{E:INTROSHOCKENERGY}-\eqref{E:INTRODIFFNOSHOCKENERGYFLUX}
and their analogs up to the top derivative level 
(corresponding to differentiations with respect to the geometric vectorfields).
The initial data that we treat 
are such that $\mathbb{W}(0,1) \lesssim \mathring{\upepsilon}^2$ and 
$\mathbb{W}(2 \TranminusdatasizeWithFactor^{-1},0) \lesssim \mathring{\upepsilon}^2$,
with $\mathring{\upepsilon}$ the small parameter described in Subsubsect.\ \ref{SSS:INTRODATABOOTSTRAPPOINTWISE}.
We again stress that $\mathbb{W}(t,u) \equiv 0$ for simple plane waves.
Energy identities, based on applying the divergence theorem
on the geometric coordinate region $[0,t] \times [0,u] \times \mathbb{T}^{n-1}$,
together with the
pointwise estimates for error terms mentioned in Subsubsect.\ \ref{SSS:INTRODATABOOTSTRAPPOINTWISE},
lead to the following inequality:
\begin{align} \label{E:INTROSTANDARDERRORTERM}
	\mathbb{W}(t,u)
	& \leq 
		C \mathring{\upepsilon}^2
		+
		C
		\int_{t'=0}^t
		\int_{u'=0}^u
		\int_{\mathbb{T}^{n-1}}
			\left\lbrace
				|\Singletan \Psi|^2
				+
				|\noshock|^2
				+
				|\diffnoshock|^2
			\right\rbrace
			(t',u',\vartheta)
		\, d \torusvol
		du'
		dt'
		+
		\cdots,
\end{align}
where the terms $\cdots$ depend on other geometric quantities
and can be bounded using similar arguments similar to the ones we sketch here.
In view of the definition of $\mathbb{W}$, we deduce the following
inequality from \eqref{E:INTROSTANDARDERRORTERM}:
\begin{align} \label{E:GRONWALLREADYINTROSTANDARDERRORTERM}
	\mathbb{W}(t,u)
	& \leq 
		C \mathring{\upepsilon}^2
		+
		C
		\int_{t'=0}^t
			\mathbb{W}(t',u)
		\, dt'
		+
		C
		\int_{u'=0}^u
			\mathbb{W}(t,u')
		\, du'
		+
		\cdots.
\end{align}
Then from \eqref{E:GRONWALLREADYINTROSTANDARDERRORTERM} and Gronwall's inequality with respect to $t$ and $u$,
we conclude, ignoring the terms $\cdots$ and taking into account \eqref{E:INTROTBOOTSIZEASSUMPTION}, 
that the following a priori estimate holds for
$(t,u) \in [0,T_{(Boot)}) \times [0,U_0]$
(see Prop.\ \ref{P:MAINAPRIORI} for the details):
\begin{align} \label{E:INTROAPRIORENERGY}
	\mathbb{W}(t,u) 
	& \lesssim \mathring{\upepsilon}^2 \exp\left(C \TranminusdatasizeWithFactor^{-1} \right)
	\lesssim \mathring{\upepsilon}^2.
\end{align}
The estimate \eqref{E:INTROAPRIORENERGY} represents the realization of
our hope that the solution remains regular relative to the geometric coordinates,
up to the top derivative level.

We now stress the following key point: 
the characteristic fluxes $\noshockfl[\noshock]$
and $\noshockfl[\diffnoshock]$
are needed to control the terms
$|\noshock|^2
+
|\diffnoshock|^2$
on RHS~\eqref{E:INTROSTANDARDERRORTERM};
without the characteristic fluxes,
instead of the term
$
C
		\int_{u'=0}^u
			\mathbb{W}(t,u')
		\, du'
$
on RHS~\eqref{E:GRONWALLREADYINTROSTANDARDERRORTERM}, we would instead have the
term
$
C
\int_{t'=0}^t
	\frac{\mathbb{W}(t',u)}{\min_{\Sigma_{t'}^u} \upmu}
\, dt'
$,
whose denominator vanishes as the shock forms. 
Such a term would have led to a priori estimates allowing for the possibility
that at all derivative levels, the geometric energies blow up as the shock forms. This in turn
would have been inconsistent with the bootstrap assumptions described in Subsubsect.\ \ref{SSS:INTRODATABOOTSTRAPPOINTWISE} 
and would have obstructed 
our goal of showing that the solution
remains regular relative to the geometric coordinates.

\subsubsection{Combining the estimates}
\label{SSS:INTROCOMBININGTHEESTIAMTE}
Once we have obtained the a priori energy estimates, we can derive improvements of our $L^{\infty}$-type 
bootstrap assumptions via Sobolev embedding (see Cor.\ \ref{C:IMPROVEMENTOFFUNDAMANETALBOOTSTRAPASSUMPTIONS}).
These steps,
together with the estimates from Subsubsect.\ \ref{SSS:INTROSHOCKFORMS} showing that $\upmu$ vanishes in finite time,
are the main steps in the proof of the main theorem.
We need a few additional technical results to complete the proof, including
some results guaranteeing that the geometric and Cartesian coordinates
are diffeomorphic up to the shock (see Sect.\ \ref{S:ESTIMATESFORCHOV})
and some fairly standard continuation criteria
(see Sect.\ \ref{S:CONTINUATION}),
which in total ensure that the solution survives up to the shock.
We combine all of these results in Sect.\ \ref{S:MAINTTHM},
where we prove the main theorem.

\subsection{Connections to prior work}
\label{SS:PRIORWORKANDCONTEXT}
Many aspects of the approach outlined in Subsect.\ \ref{SS:PROOFOVERVIEW}
have their genesis in earlier works, which we now describe.

\subsubsection{Results in one spatial dimension}
\label{SSS:1DRESULTS}
In one spatial dimension and in symmetry classes whose PDEs are effectively one-dimensional,
there are many results, by now considered classical,
that use the method of characteristics to exhibit 
the formation of shocks in initially smooth solutions to
various quasilinear hyperbolic systems. 
Important examples include
Riemann's work \cite{bR1860} (in which he developed the method of Riemann invariants),
Lax's proof \cite{pL1964} of stable blowup for $2 \times 2$ genuinely nonlinear systems
via the method of Riemann invariants,
Lax's blowup results \cites{pL1972,pL1973}
for scalar conservation laws,
John's extension \cite{fJ1974} of Lax's work to systems in one spatial dimension
with more than two unknowns (which required the development of new ideas since the method of Riemann invariants does not apply),
and the recent work \cite{dCdRP2016} of Christodoulou--Raoul Perez, in which they significantly
sharpened John's work \cite{fJ1974}.
The main obstacle to extending the above results to more than one spatial dimension
is that one must complement the method of characteristics with 
an ingredient that, due to the singularity formation, is often 
accompanied by enormous technical complications:
energy estimates that are adapted to and that hold up to the singularity. 
We further explain these technical complications in the next subsubsection.

  \subsubsection{Results in more than one spatial dimension}
\label{SSS:MULTIDRESULTS}
The first breakthrough results on shock formation in more than one spatial dimension without symmetry assumptions 
were proved by Alinhac \cites{sA1999a,sA1999b,sA2001b} for 
small-data solutions to 
scalar quasilinear wave equations
of the form 
\begin{align} \label{E:ALINHACWAVE}
	(g^{-1})^{\alpha \beta}(\partial \Phi) \partial_{\alpha} \partial_{\beta} \Phi = 0
\end{align}
that fail to satisfy the null condition. Here, 
$g(\partial \Phi)$ is a Lorentzian metric\footnote{That is, the matrix of Cartesian components of $g(\partial \Phi)$
has signature $(-,+,\cdots,+)$.}
equal to the Minkowski metric plus an error term of size $\mathcal{O}(\partial \Phi)$.
As we do in this paper, 
Alinhac constructed a set of geometric coordinates tied to an eikonal function $u$, 
which in the context of his problems
was a solution the fully nonlinear eikonal equation
\begin{align} \label{E:WAVEEIKONAL}
	(g^{-1})^{\alpha \beta}(\partial \Phi) \partial_{\alpha} u \partial_{\beta} u
	& = 0.
\end{align}
Much like in our work here, the level sets of $u$ are characteristic hypersurfaces for equation \eqref{E:ALINHACWAVE}.
They are also known as \emph{null hypersurfaces} in 
the setting of Lorentzian geometry in view of their intimate connection to 
the $g$-null\footnote{That is, if $\Lunit^{\alpha} := - (g^{-1})^{\alpha \beta} \partial_{\beta} u$, then
by \eqref{E:WAVEEIKONAL}, we have $g(\Lunit,\Lunit) = 0$.} 
vectorfield $-(g^{-1})^{\alpha \beta} \partial_{\beta} u$.
In his works, Alinhac identified a set of small compactly supported initial data verifying a non-degeneracy condition
such that $\max_{\alpha,\beta=0,\cdots,n} |\partial_{\alpha} \partial_{\beta} \Phi|$ 
blows up in finite time due to the intersection of the characteristics
while $|\Phi|$ and $\max_{\alpha=0,\cdots,n} |\partial_{\alpha} \Phi|$ remain bounded. Moreover, relative to the geometric coordinates,
$\Phi$ and $\lbrace \partial_{\alpha} \Phi \rbrace_{\alpha = 0,\cdots,n}$ remain smooth, except possibly at the very high derivative levels 
(we will elaborate upon this just below).

In proving his results, Alinhac faced three serious difficulties.
We will focus only on the case of three spatial dimensions though
Alinhac obtained similar results in two spatial dimensions.
The first difficulty is that for small data, solutions to \eqref{E:ALINHACWAVE} 
experience a long period of dispersive decay, 
which seems to work against the formation of a shock and which
necessitated the application of Klainerman's
commuting vectorfield method \cites{sK1985,sK1986} in which the vectorfields
have time and radial weights.
We stress that such dispersive behavior is not exhibited by the solutions that we study in this article
and hence our vectorfields do not feature time or radial weights.
Alinhac showed that after an era\footnote{Roughly for a time interval of length $\exp(c/\epsilon)$, 
with $\epsilon$ the size of the data in a weighted Sobolev norm.} 
of dispersive decay,
the nonlinearity in equation \eqref{E:ALINHACWAVE}  
takes over and drives the formation of the shock.
The second main difficulty faced by Alinhac is that to follow the solution up the singularity, 
it seems necessary to commute the equations with geometric vectorfields
constructed out of the eikonal function, 
and these vectorfields seem to lead to the loss of a derivative 
when commuted through the wave operator.
Specifically, the geometric vectorfields $Z$ 
have Cartesian components that depend on $\partial u$, 
and hence commuting them through the wave equation \eqref{E:WAVEEIKONAL}
leads to an equation of the schematic form
$(g^{-1})^{\alpha \beta}(\partial \Phi) \partial_{\alpha} \partial_{\beta} (Z \Phi) = \partial^2 Z \cdot \partial \Phi + \cdots$.
The difficulty is that standard wave equation energy estimates suggest that, due to the source term
$\partial^2 Z$, $\Phi$ enjoys only the same Sobolev regularity 
as $Z \sim \partial u$, whereas standard energy estimates for the eikonal equation 
\eqref{E:WAVEEIKONAL} only allow one to prove that $\partial u$ enjoys the same Sobolev regularity as $\partial^2 \Phi$;
this suggests that the approach of using vectorfields constructed out of an eikonal function will lead to the loss of a derivative. 
To overcome this difficulty,
Alinhac obtained the nonlinear solution, up to the shock, as the limit of iterates that solve
singular linearized problems, and he used a rather technical Nash--Moser iteration scheme
featuring a free boundary in order to recover the loss
of a derivative. For technical reasons, his reliance on the Nash--Moser iteration allowed him to follow 
``most'' small-data solutions to the constant-time hypersurface of first blowup, and not further.
More precisely, his approach only allowed him to treat ``non-degenerate'' data such that the first singularity is \emph{isolated}
in the constant-time hypersurface of first blowup. We stress that in our work here, 
we encounter a similar difficulty concerning the regularity of the geometric vectorfields, 
but since our PDE systems are first-order, we are able to overcome it in a different way,
without relying on a Nash--Moser iteration scheme; 
see Subsect.\ \ref{SS:NONLINEARGEOMERICOPTICS} and Subsubsect.\ \ref{SSS:REGULARITYCONSIDERATIONS}.
The third and most challenging difficulty encountered by Alinhac is the following: 
when proving energy estimates relative to the geometric coordinates,
it seems necessary to rely on energies that feature degenerate weights 
that vanish as the shock forms; the weights are direct analogs of the inverse foliation density $\upmu$ from Theorem~\ref{T:ROUGHMAINTHM}.
These weights make it difficult to control certain error terms in the energy identities,
which in turn leads to a priori estimates allowing for the following possibility:
as the shock forms,
the high-order energies might blow up at a 
rate tied to $1/\upmu$. 
We stress that the possible high-order energy blowup encountered by Alinhac occurs 
relative to the geometric coordinates and is \emph{distinct} from the formation of the shock singularity
(in which $\max_{\alpha,\beta=0,\cdots,n} |\partial_{\alpha} \partial_{\beta} \Phi|$ blows up).
To close the proof, Alinhac had to show that the possible high-order geometric energy blowup does not propagate
down too far to the lower geometric derivative levels, i.e., that the solution remains smooth relative to the geometric coordinates
at the lower derivative levels. This ``descent scheme'' costs many derivatives,
and for this reason, the data must belong to a Sobolev
space of rather high order for the estimates to close.
We stress that although the energies that we use in the present paper also contain the same degenerate $\upmu$ weights, 
we encounter different kinds of error terms in our energy estimates, 
tied in part to the fact that our systems are first-order
and tied in part to our strategy of estimating the quantity
$\diffnoshockdoublearg{\alpha}{J}$ defined by \eqref{E:INTRODIFFERENTIATEDNOSHOCK}.
For this reason, our a priori estimates energy relative to the geometric coordinates are regular in that
\emph{even the top-order geometric energies remain uniformly bounded up to the shock}.

In Christodoulou's remarkable work \cite{dC2007}, he significantly sharpened Alinhac's 
shock formation results for the quasilinear wave equations of irrotational (i.e., vorticity-free)
relativistic fluid mechanics in three spatial dimensions, which form a sub-class of wave equations of type 
\eqref{E:ALINHACWAVE}. These wave equations arise 
from formulating the relativistic Euler equations
in terms of a fluid potential
$\Phi$, which is possible when the vorticity vanishes. 
The equations studied by Christodoulou enjoy special features that he exploited in his proofs,
such as having an Euler-Lagrange formulation with a Lagrangian 
that is invariant under the Poincar\'{e} group.
The main results proved by Christodoulou are as follows:
\textbf{i)} there is an open (relative to a Sobolev space of high, non-explicit order) 
set of small\footnote{In the context of \cite{dC2007}, ``small'' means
a small perturbation of the non-trivial constant-state solutions,
which take the form $\Phi = k t$, where $k > 0$ is a constant.} 
data such that the only possible singularities that can form in the solution are shocks
driven by the intersection of the characteristics;
\textbf{ii)} there is an open subset of the data from \textbf{i)}, not restricted
by non-degeneracy assumptions of the type imposed by Alinhac, such that
a shock does in fact form in finite time; 
and \textbf{iii)} for those solutions that form shocks, Christodoulou gave a complete description
of the maximal classical development of the data near the singularity,
which intersects the future of the constant-time hypersurface of first blowup.
His sharp description of the maximal development
seems necessary for even properly setting up the \emph{shock development problem}. 
This is the problem of uniquely locally continuing the solution
past the singularity to the Euler equations in a \emph{weak sense}, 
a setting in which one must also construct the ``shock hypersurface,''
across which the solution jumps (being smooth on either side of it). The shock development problem
in relativistic fluid mechanics was solved in spherical symmetry by Christodoulou--Lisibach in \cite{dCaL2016}
and, by Christodoulou in yet another breakthrough work \cite{dC2017}, 
for the non-relativistic compressible Euler equations without symmetry assumptions 
in a restricted case 
(known as the restricted shock development problem)
such that the jump in entropy across the shock hypersurface is ignored.

Compared to Alinhac's approach, 
the main technical improvement afforded by Christodoulou's approach \cite{dC2007} to proving shock formation
is that it avoids the loss of a derivative through a sharper, more direct method;
instead of using Alinhac's Nash--Moser scheme, 
Christodoulou found special combinations of geometric quantities
that satisfy good evolution equations, and he combined them with elliptic estimates on co-dimension two
spacelike hypersurfaces. This approach to avoiding the loss of a derivative in 
wave equation eikonal functions originated in the aforementioned proof \cite{dCsK1993} of the stability of Minkowski spacetime,
and it was extended by Klainerman--Rodnianski \cite{sKiR2003}
to the case of general scalar quasilinear wave equations 
in their study of low-regularity well-posedness for wave equations of
the form $- \partial_t^2 \Psi + g^{ab}(\Psi) \partial_a \partial_b \Psi = 0$.
In total, this allowed Christodoulou to control the solution up to the shock
using a traditional ``forwards'' approach,
without the free boundary found in Alinhac's iteration scheme.
However, as in Alinhac's work,
Christodoulou' energy estimates allowed for the possibility that the high-order energies might blow up.
Christodoulou therefore had to give a separate, technical argument to show 
that any high-order energy singularity does not propagate down too far to the lower geometric derivative levels.

In \cite{jS2016b}, we extended
Christodoulou's sharp shock formation results 
to the case of general quasilinear wave equations
of type \eqref{E:ALINHACWAVE} in three spatial dimensions 
that fail to satisfy the null condition,
to the case of covariant wave equations 
of the type\footnote{Here, $\square_g$ is the covariant wave operator of $g$.
Relative to arbitrary coordinates, $\square_g \Psi = \frac{1}{\sqrt{|\mbox{\upshape{det}} \mbox{$g$}}|} \partial_{\alpha}(\sqrt{|\mbox{\upshape{det}} g|}(g^{-1})^{\alpha \beta} 	\partial_{\beta} \Psi)$. \label{FN:COVWAVOP}} 
$\square_{g(\Psi)} \Psi = 0$ that fail to satisfy the null condition,
and to inhomogeneous versions of these wave equations featuring ``admissible'' semilinear terms.
Similar results were proved in \cite{dCsM2014} for a subset of these equations,
namely those wave equations arising from non-relativistic compressible fluid mechanics
with vanishing vorticity. All of the results mentioned so far in this subsubsection 
are explained in detail in the survey
article \cite{gHsKjSwW2016}.

In the wake of the above results, there have been significant further advancements, 
which we now describe. In \cite{jSgHjLwW2016}, we extended the shock formation results of \cite{jS2016b}
to a new, physically relevant regime of initial conditions in two spatial dimensions
such that the solutions are close to simple outgoing plane symmetric waves,
much like the setup of the present article.
For the initial conditions studied in \cite{jSgHjLwW2016}, 
the solutions do not experience dispersive decay. Hence,
we used a new analytic framework to control the solution up to the shock,
based on ``close-to-simple-plane-wave''-type smallness assumptions on the data that
are similar in spirit to the assumptions that we make on the data in the present article.
For special classes of wave equations
in three spatial dimensions with cubic nonlinearities,
Miao--Yu \cite{sMpY2017} proved similar shock formation results
for a class of large initial data
featuring a single scaling parameter, similar to 
the short pulse ansatz exploited by Christodoulou in his breakthrough work \cite{dC2009}
on the formation of trapped surfaces in solutions to the Einstein-vacuum equations.
For the same wave equations studied in \cite{sMpY2017}, 
Miao \cite{sM2016} recently
used a related but distinct ansatz for the initial data
to prove the existence of an open set of solutions 
that blow up at time $T_{(Shock)} \approx - 1$
but exist classically on the time interval $(-\infty,T_{(Shock)})$.

All of the above works concern systems that feature relatively simple characteristics:
those corresponding to a single wave operator. We now describe some recent shock formation
results in which the systems have more complicated principal parts, leading to multiple speeds
of propagation and distinct families of characteristics.
The first result of this type without symmetry assumptions 
was our joint work \cite{jLjS2016b} with J.~Luk, which
concerned the compressible Euler equations 
in two spatial dimensions under an arbitrary\footnote{There is one exceptional equation of state, known as that of the Chaplygin gas, to which the results of \cite{jLjS2016b} do not apply. In one spatial dimension, the resulting PDE system is \emph{totally linearly degenerate}, and many
experts believe that shocks do not form in solutions to such systems.} 
barotropic\footnote{A barotropic equation of state is such that the pressure is a function of the density.} 
equation of state. Specifically, in \cite{jLjS2016b},
we extended the shock formation results of
\cite{dCsM2014} for the compressible Euler equations to allow for the presence of
small amounts of vorticity at the location of the singularity. The vorticity satisfies
a transport equation and, as it turns out, remains Lipschitz with respect to the Cartesian coordinates,
all the way up to the shock. More precisely, the shock occurs in the ``sound wave part'' of the system rather than in the vorticity,
and, as in all prior works, it is driven by the intersection of a family of characteristic hypersurfaces
corresponding to a Lorentzian metric (known as the \emph{acoustical metric} in the context of fluid mechanics).
In particular, \cite{jLjS2016b} yielded the first proof of stable shock formation without symmetry assumptions
in solutions to a hyperbolic system featuring multiple speeds, where all solution variables were allowed to interact
up to the singularity. 

The results proved in \cite{jLjS2016b} were based on a new wave-transport-div-curl 
formulation of the compressible Euler equations under a barotropic equation of state,
which we derived in \cite{jLjS2016a}.
The new formulation exhibits remarkable null structures and regularity properties, tied in part 
to the availability of elliptic estimates for the vorticity 
in three spatial dimensions 
(vorticity stretching does not occur in two spatial dimensions, 
and in its absence,
one does not need elliptic estimates to control the vorticity).
In a forthcoming work, we will extend the shock formation results of 
\cite{jLjS2016b} to the much more difficult case of three spatial dimensions,
where to control the vorticity up to top order in a manner compatible with the wave part of the system,
one must rely on the elliptic estimates,
which allow one to show that the vorticity is exactly as differentiable as the
velocity with respect to geometric vectorfields adapted to the sound wave characteristics.
In \cite{jS2017a}, we extended the results of \cite{jLjS2016a} to allow for an arbitrary equation of state
in which the pressure depends on the density and entropy. The formulation of the equations in
\cite{jS2017a} exhibits further remarkable properties that, in our forthcoming work,
we will use to prove a stable shock formation result in three spatial dimensions in which the vorticity and entropy
are allowed to be non-zero at the singularity.
In \cite{jS2017b}, in two spatial dimensions,
we proved the first stable shock formation result for systems of quasilinear wave equations
featuring \underline{multiple wave speeds} of propagation, i.e., the systems featured more than one distinct 
quasilinear wave operator.
The main result provided an open set of data
such that the ``fastest'' wave forms a shock in finite time while the remaining solution variables
remain regular up to the singularity in the fast wave, much like in Theorem~\ref{T:ROUGHMAINTHM}.
The initial conditions were perturbations of simple plane waves, similar to the setup for the case of the scalar wave equations
studied in \cite{jSgHjLwW2016} and similar to the setup of the present article. 
The main new difficulty that we faced in \cite{jS2017b} is that 
the geometric vectorfields adapted to the shock-forming fast wave,
which seem to be an essential ingredient for following the fast wave all the way to its singularity,
exhibit very poor commutation properties
with the slow wave operator. Indeed, commuting the geometric vectorfields all the way through the slow wave
operator produces error terms that are uncontrollable both from the point of view of regularity
and from the point of view of the strength of the singular commutator terms that this generates.
To overcome this difficulty, we relied on a first-order reformulation of the 
slow wave equation which, though somewhat limiting in the precision it affords,
allows us to avoid commuting all the way through the slow wave operator and 
hence to avoid the uncontrollable error terms.

\subsection{Notation, index conventions, and conventions for ``constants''}
\label{SS:NOTATIONANDINDEXCONVENTIONS}
We now summarize some our notation. Some of the concepts referred to here
are defined later in the article.
Throughout, $\lbrace x^{\alpha} \rbrace_{\alpha =0,1,\cdots,n}$
denote the standard Cartesian coordinates
on spacetime $\mathbb{R} \times \Sigma$,
where $x^0 \in \mathbb{R}$ is the time variable and 
$(x^1,x^2,\cdots,x^n) \in \Sigma = \mathbb{R} \times \mathbb{T}^{n-1}$ are the space variables.
We denote the corresponding Cartesian partial derivative vectorfields by
$
\displaystyle
\partial_{\alpha}
=:
\frac{\partial}{\partial x^{\alpha}}
$
(which are globally defined and smooth even though $\lbrace x^i \rbrace_{i=2}^n$ are only locally defined)
and we often use the alternate notation $t := x^0$ and $\partial_t := \partial_0$.

\begin{itemize}
	\item Lowercase Greek spacetime indices 
	$\alpha$, $\beta$, etc.\
	correspond to the Cartesian spacetime coordinates 
	and vary over $0,1,\cdots,n$.
	Lowercase Latin spatial indices
	$a$,$b$, etc.\ 
	correspond to the Cartesian spatial coordinates and vary over $1,2,\cdots,n$.
	An exception to the latter rule occurs for the
	geometric torus coordinate vectorfields
	$\CoordAng{i}$
	from
	\eqref{E:GEOMETRCICOORDINATEPARTIALDERIVATIVEVECTORFIELDS},
	in which the labeling index $i$ varies over $2,\cdots,n$.
	Uppercase Latin indices such as $J$ correspond to the components
	$\noshockuparg{J}$ of the array of symmetric hyperbolic variables and
	typically vary from $1$ to $M$.
\item We use Einstein's summation convention in that repeated indices are summed over their respective ranges.
\item Unless otherwise indicated, 
	all quantities in our estimates that are not explicitly under
	an integral are viewed as functions of 
	the geometric coordinates $(t,u,\vartheta)$
	of Def.\ \ref{D:GEOMETRCICOORDINATES}.
	Unless otherwise indicated, quantities
	under integrals have the functional dependence 
	established below in
	Def.\ \ref{D:NONDEGENERATEVOLUMEFORMS}.
\item If $Q_1$ and $Q_2$ are two operators, then
	$[Q_1,Q_2] = Q_1 Q_2 - Q_2 Q_1$ denotes their commutator.
\item $A \lesssim B$ means that there exists $C > 0$ such that $A \leq C B$.
\item $A \approx B$ means that $A \lesssim B$ and $B \lesssim A$.
\item $A = \mathcal{O}(B)$ means that $|A| \lesssim |B|$.
\item Constants such as $C$ and $c$ are free to vary from line to line.
	\textbf{Explicit and implicit constants are allowed to depend in an increasing, 
	continuous fashion on the data-size parameters 
	$\Trandatasize$
	and $\TranminusdatasizeWithFactor^{-1}$
	from
	Subsect.\ \ref{SS:DATASIZE}.
	However, the constants can be chosen to be 
	independent of the parameters 
	$\Psiep$,
	$\mathring{\upepsilon}$,
	and $\varepsilon$ whenever the following conditions hold:
	\textbf{i)}
	$\mathring{\upepsilon}$
	and $\varepsilon$
	are sufficiently small relative to 
	$1$,
	relative to
	$\Trandatasize^{-1}$,
	and relative to $\TranminusdatasizeWithFactor$,
	and 
	\textbf{ii)}
	$\Psiep$ is sufficiently small relative to $1$}
	in the sense described in Subsect.\ \ref{SS:SMALLNESSASSUMPTIONS}.
\item Constants $C_{\star}$ are \textbf{universal 
		in that, as long as $\Psiep$ and $\mathring{\upepsilon}$ are sufficiently small relative to $1$,
		they do not depend on
		$\varepsilon$,
		$\mathring{\upepsilon}$,
		$\Trandatasize$,
		or $\TranminusdatasizeWithFactor$.}
\item $A = \mathcal{O}_{\star}(B)$ means that $|A| \leq C_{\star} |B|$,
	with $C_{\star}$ as above.
\item $\lfloor \cdot \rfloor$
	and $\lceil \cdot \rceil$
	respectively denote the standard floor and ceiling functions. 
\end{itemize}

\section{Rigorous setup of the problem and fundamental definitions}
\label{S:SETUPOFPROBLEM}
In this section, we state the equations that we will study
and state our basic assumptions on the nonlinearities.

\subsection{Statement of the equations}
\label{SS:EQUATIONS}
Our main results concern systems in 
$1 + n$ spacetime dimensions 
and
$1 + M$ unknowns of the following form:
\begin{align} 
	\Lunit \Psi & = 0,
		\label{E:SHOCKEQN} \\
	A^{\alpha} \partial_{\alpha} \noshock & = 0,
	\label{E:NONSHOCKEQUATION} 
\end{align}
where the scalar function $\Psi$ will eventually form a shock,
$M \geq 1$ is an integer,\footnote{Our results also apply in the case $M=0$, though we omit discussion of this simpler case.}
\begin{align} \label{E:NONSHOCKARRAY}
	\noshock := (\noshockuparg{J})_{J=1,\cdots,M}
\end{align}
denotes the ``symmetric hyperbolic variables'' 
(which will remain regular up to the singularity in $\max_{\alpha=0,\cdots,n} |\partial_{\alpha} \Psi|$),
$\Lunit$ is a vectorfield whose Cartesian components
are given smooth functions of $\Psi$ and $\noshock$, that is, $\Lunit^{\alpha} = \Lunit^{\alpha}(\Psi,\noshock)$,
and $A^{\alpha}$ are \emph{symmetric} $M \times M$ matrices
whose components $A_I^{\alpha;J} = A_J^{\alpha;I}$ are given smooth functions of $\Psi$ and $\noshock$.
Note that equation \eqref{E:NONSHOCKEQUATION} is equivalent to the $M$ scalar equations
$A_J^{\alpha;I} \partial_{\alpha} \noshockuparg{J} = 0$,
where $1 \leq I \leq M$ and with summation over $\alpha$ and $J$.
For convenience, we assume the normalization conditions
\begin{subequations}
\begin{align} \label{E:LUNIT0ISONE}
	\Lunit^0 & \equiv 1,
		\\
	\Lunit^1|_{(\Psi,\noshock)=(0,0)} 
	& = 1.
	\label{E:LUNIT1ISONEFORBACKGROUND}
\end{align}
\end{subequations}
More generally, if
$(\Lunit^0|_{(\Psi,\noshock)=(0,0)},\Lunit^1|_{(\Psi,\noshock)=(0,0)}) \neq (0,0)$,
then \eqref{E:LUNIT0ISONE}-\eqref{E:LUNIT1ISONEFORBACKGROUND}
can be achieved by performing a linear change of coordinates in the $(t,x^1)$ plane and 
then dividing equation \eqref{E:SHOCKEQN} by a scalar.

As we stressed in the introduction, an essential aspect of our analysis is that we
treat the Cartesian coordinate partial derivatives of
$\noshockuparg{J}$ as independent quantities. For this reason, we define
\begin{align} \label{E:CARTESIANDIFFNOSHOCK}
	\diffnoshockdoublearg{\alpha}{J}
	& := \partial_{\alpha} \noshockuparg{J},
	&
	\diffnoshockdownarg{\alpha} 
	&:= (\diffnoshockdoublearg{\alpha}{J})_{1 \leq J \leq M},
	&
	\diffnoshock 
	& := (\diffnoshockdoublearg{\alpha}{J})_{0 \leq \alpha \leq n, 1 \leq J \leq M}.
\end{align}
As a straightforward consequence of equation \eqref{E:NONSHOCKEQUATION} 
and definition \eqref{E:CARTESIANDIFFNOSHOCK},
we obtain the following evolution equation for $\diffnoshockdownarg{\alpha}$:
\begin{align} \label{E:CARTESIANDIFFERENTIATEDNONSHOCKEQUATION} 
	A^{\beta} \partial_{\beta} \diffnoshockdownarg{\alpha} 
	& = - (\partial_{\alpha} A^{\beta}) \diffnoshockdownarg{\beta}.
\end{align}

\subsection{The genuinely nonlinear-type assumption}
\label{SS:GENUINELYNONLINEAR}
To ensure that shocks can form in nearly plane symmetric solutions,
we assume that for $|\Psi| + |\noshock|$ sufficiently small, we have
\begin{align} \label{E:GENUINELYNONLINEAR}
	 \frac{\partial \Lunit^1}{\partial \Psi} \neq 0.
\end{align}

\subsection{Assumptions on the speed of propagation for the symmetric hyperbolic subsystem}
\label{SS:ASSUMPTIONSONTHESPEEDSOFPROPAGATION}
In this subsection, we state our assumptions on the speed of propagation
for the symmetric hyperbolic subsystem \eqref{E:NONSHOCKEQUATION}.
Specifically, we assume that the matrices
\begin{align} \label{E:POSITIVEDEFMATRICES}
	&A^0|_{(\Psi,\noshock)=(0,0)}
	\mbox{ and }
	A^0|_{(\Psi,\noshock)=(0,0)}
	-
	A^1|_{(\Psi,\noshock)=(0,0)}
	\mbox{ are positive definite}.
\end{align}
We now explain the significance of \eqref{E:POSITIVEDEFMATRICES}.
The positivity of $A^0|_{(\Psi,\noshock)=(0,0)}$
ensures that for solution values near the ``background state'' $(\Psi,\noshock)=(0,0)$,
the hypersurfaces $\Sigma_t$ are spacelike for equation \eqref{E:NONSHOCKEQUATION},
that is, for the evolution equation verified by the non-shock-forming variable $\noshock$.
By \eqref{E:LUNIT0ISONE}, the
$\Sigma_t$ are also spacelike for equation \eqref{E:SHOCKEQN},
i.e., $\Lunit$ is transversal to $\Sigma_t$.
The positivity of 
$
A^0|_{(\Psi,\noshock)=(0,0)}
-
A^1|_{(\Psi,\noshock)=(0,0)}
$
will ensure that for solution values near the background state,
hypersurfaces close to the flat planes $\lbrace t - x^1 = \mbox{\upshape} const \rbrace$
are spacelike for equation \eqref{E:NONSHOCKEQUATION}. This assumption is significant because
for the solutions that we will study, we will construct 
(in Subsect.\ \ref{SS:EIKONALFUNCTION})
a family $\lbrace \mathcal{P}_u \rbrace_{u \in [0,1]}$ of hypersurfaces that are characteristic
for equation \eqref{E:SHOCKEQN} (that is, for the operator $\Lunit$)
and that are close to the flat planes $\lbrace t - x^1 = \mbox{\upshape} const \rbrace$.
Put differently, the $\mathcal{P}_u$ will be characteristic for the evolution equation for $\Psi$ 
but spacelike for the evolution equation for $\noshock$,
which essentially means that for solution values near the background state,
$\Psi$ propagates at a strictly faster speed than $\noshock$
(and also strictly faster than $\diffnoshock$, since the principal coefficients 
in the evolution equations for $\noshock$ and $\diffnoshockdownarg{\alpha}$ are the same).

\section{Geometric constructions}
\label{S:GEOMETRICCONSTRUCTIONS}
In this section, we define/construct most of the geometric objects that we
use to analyze solutions. We defer the construction of the energies until
Sect.\ \ref{S:ENERGYID}.

\subsection{The eikonal function and the geometric coordinates}
\label{SS:EIKONALFUNCTION}
In this subsection, we construct the geometric coordinates that we use
to follow the solution all the way to the shock. The most important of these 
is the eikonal function.

\begin{definition}[\textbf{Eikonal function}]
The eikonal function is the solution $u$ to the following transport initial value problem,
where $\Lunit$ is the transport operator vectorfield from equation \eqref{E:SHOCKEQN}:
\begin{align} \label{E:EIKONAL}
		\Lunit u 
		& = 0, 
		&& u|_{\Sigma_0} = 1 - x^1.
\end{align}
\end{definition}

For reasons described in Remark~\ref{R:DESCRIPTIONOFMAXIMALDEVELOPMENT} and Subsubsect.\ \ref{SSS:SPACETIMEREGION},
we now fix a real parameter $U_0$ verifying
\begin{align} \label{E:FIXEDPARAMETER}
0 < U_0 \leq 1.
\end{align}
We will restrict out attention to spacetime regions
with $0 \leq u \leq U_0$.

Our analysis will take place on the following subsets of
spacetime, which are tied to the eikonal function;
see Fig.~\ref{F:REGION} for a picture of the setup.

\begin{definition} [\textbf{Subsets of spacetime}]
\label{D:HYPERSURFACESANDCONICALREGIONS}
We define the following subsets of spacetime:
\begin{subequations}
\begin{align}
	\Sigma_{t'} & := \lbrace (t,x^1,x^2,\cdots,x^n) \in \mathbb{R} \times \mathbb{R} \times \mathbb{T}^{n-1}  
		\ | \ t = t' \rbrace, 
		\label{E:SIGMAT} \\
	\Sigma_{t'}^{u'} & := \lbrace (t,x^1,x^2,\cdots,x^n) \in \mathbb{R} \times \mathbb{R} \times \mathbb{T}^{n-1} 
		 \ | \ t = t', \ 0 \leq u(t,x^1,x^2,\cdots,x^n) \leq u' \rbrace, 
		\label{E:SIGMATU} 
		\\
	\mathcal{P}_{u'}
	& := 
		\lbrace (t,x^1,x^2,\cdots,x^n) \in \mathbb{R} \times \mathbb{R} \times \mathbb{T}^{n-1} 
			\ | \ u(t,x^1,x^2,\cdots,x^n) = u' 
		\rbrace, 
		\label{E:PU} \\
	\mathcal{P}_{u'}^{t'} 
	& := 
		\lbrace (t,x^1,x^2,\cdots,x^n) \in \mathbb{R} \times \mathbb{R} \times \mathbb{T}^{n-1} 
			\ | \ 0 \leq t \leq t', \ u(t,x^1,x^2,\cdots,x^n) = u' 
		\rbrace, 
		\label{E:PUT} \\
	\mathcal{T}_{t',u'} 
		&:= \mathcal{P}_{u'}^{t'} \cap \Sigma_{t'}^{u'}
		= \lbrace (t,x^1,x^2,\cdots,x^n) \in \mathbb{R} \times \mathbb{R} \times \mathbb{T}^{n-1} 
			\ | \ t = t', \ u(t,x^1,x^2,\cdots,x^n) = u' \rbrace, 
			\label{E:LTU} \\
	\mathcal{M}_{t',u'} & := \cup_{u \in [0,u']} \mathcal{P}_u^{t'} \cap 
		\lbrace (t,x^1,x^2,\cdots,x^n) \in \mathbb{R} \times \mathbb{R} \times \mathbb{T}^{n-1}  \ | \ 0 \leq t < t' \rbrace.
		\label{E:MTUDEF}
\end{align}
\end{subequations}
\end{definition}
We refer to the $\Sigma_t$ and $\Sigma_t^u$ as ``constant time slices,'' 
the $\mathcal{P}_u^t$ as ``characteristics,'' 
and the $\mathcal{T}_{t,u}$ as ``tori.'' 
Note that $\mathcal{M}_{t,u}$ is ``open-at-the-top'' by construction.

\begin{center}
\begin{overpic}[scale=.15]{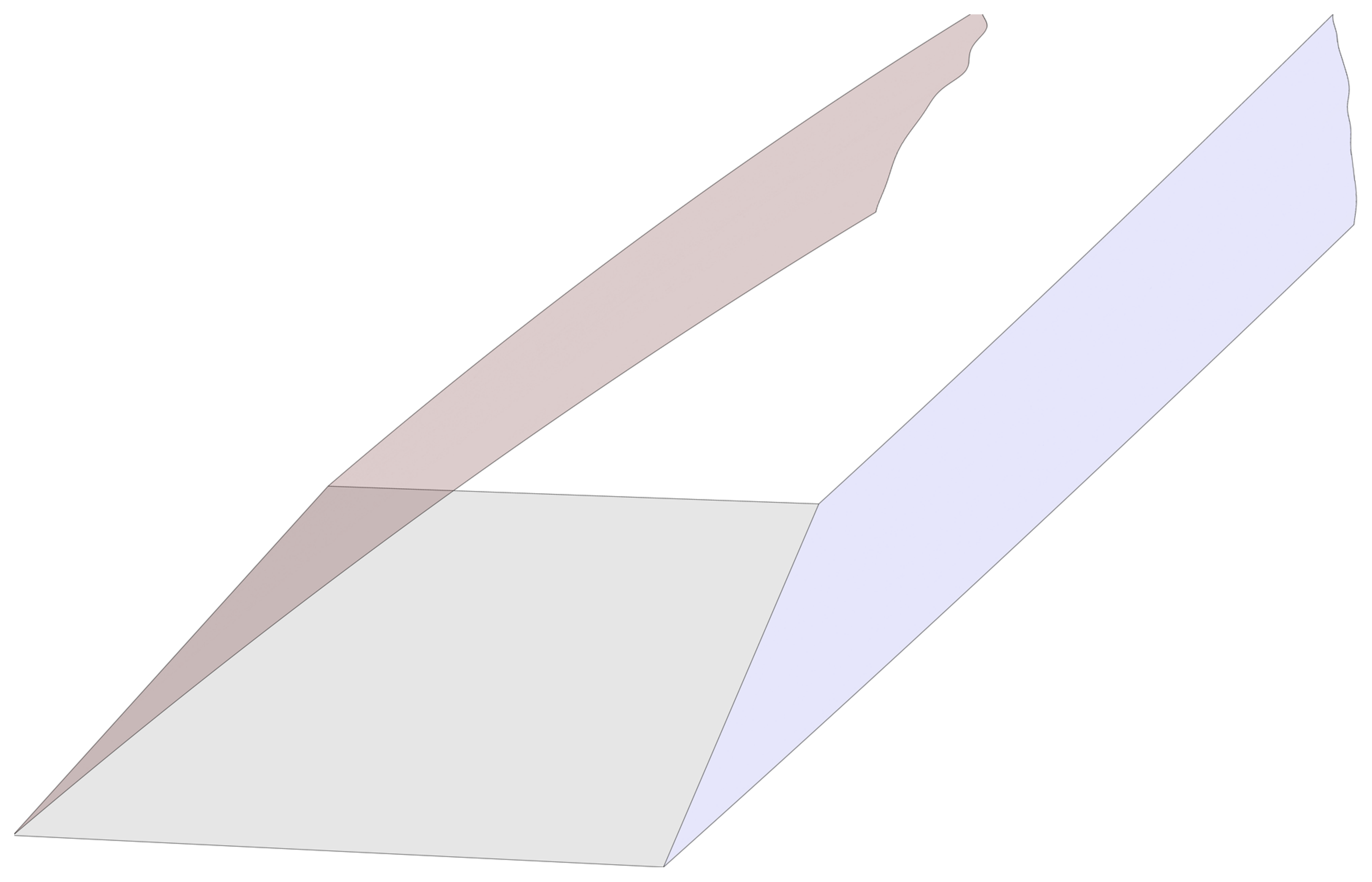} 
\put (50,35) {\large$\displaystyle \mathcal{M}_{t,u}$}
\put (37,37.3) {\large$\displaystyle \mathcal{P}_{U_0}^t$}
\put (59,18) {\large$\displaystyle \mathcal{P}_0^t$}
\put (64,54) {\large$\displaystyle \mathcal{T}_{t,u}$}
\put (97,53) {\large$\displaystyle \mathcal{T}_{t,0}$}
\put (6.5,8.5) {\large$\displaystyle \mathcal{T}_{0,u}$}
\put (50,8.5) {\large$\displaystyle \mathcal{T}_{0,0}$}
\put (8,3.5) {\large$\displaystyle \mbox{``interesting'' data}$}
\put (63,24) {\large \rotatebox{45}{$\displaystyle \Psi \mbox{ and } \noshock \mbox{ very small}$}}
\put (35,18) {\large$\Sigma_0^{U_0}$}
\put (31,9) {\large$\displaystyle U_0$}
%
\put (-4.8,16) {\large$\displaystyle x^2 \in \mathbb{T}$}
\put (24,-4.5) {\large$\displaystyle x^1 \in \mathbb{R}$}
\thicklines
\put (-1.1,3){\vector(.9,1){22}}
\put (.5,1.8){\vector(100,-4.5){48}}
\put (10.5,13.9){\line(.9,1){2}}
\put (53,11.9){\line(.43,1){1}}
\put (11.5,15){\line(100,-4.5){42}}
\end{overpic}
\captionof{figure}{The spacetime region under study in the case $n=2$.}
\label{F:REGION}
\end{center}

To complete the geometric coordinate system, 
we now construct local coordinates on the tori $\mathcal{T}_{t,u}$.

\begin{definition}[\textbf{Geometric torus coordinates}]
	We define the local geometric torus coordinates
	$(\vartheta^2,\cdots,\vartheta^n)$
	to be the solutions to the following initial value problems,
	where $\Lunit$ is the transport operator vectorfield from equation \eqref{E:SHOCKEQN}:
	\begin{align}
		\Lunit \vartheta^i
		& = 0,
		&& \vartheta^i|_{\Sigma_0} = x^i,
		&& (i=2,3,\cdots,n).
	\end{align}
	Note that we can view $(\vartheta^2,\cdots,\vartheta^n)$ as locally defined coordinates 
	on $\mathcal{T}_{t,u} \simeq \mathbb{T}^{n-1}$.
\end{definition}

\begin{definition}[\textbf{Geometric coordinates and the corresponding partial derivative vectorfields}]
	\label{D:GEOMETRCICOORDINATES}
		We refer to
		$(t,u,\vartheta^2,\cdots,\vartheta^n)$
		as the geometric coordinates,
		and we set $\vartheta := (\vartheta^2,\cdots,\vartheta^n)$.
		We denote the corresponding partial derivative vectorfields by
		\begin{align} \label{E:GEOMETRCICOORDINATEPARTIALDERIVATIVEVECTORFIELDS}
			\frac{\partial}{\partial t},
			\frac{\partial}{\partial u},
			\CoordAng{i} := \frac{\partial}{\partial \vartheta^i},
			\qquad
			(i=2,\cdots,n).
		\end{align}
\end{definition}

Note that the $\CoordAng{i}$ are $\mathcal{T}_{t,u}$-tangent by construction.
Moreover, we note even though the coordinate functions $\vartheta^i$ are 
only locally defined on $\mathcal{T}_{t,u}$, the vectorfields
$\lbrace \CoordAng{i} \rbrace_{i=2,\cdots,n}$ 
can be defined so as to form a
smooth (relative to the geometric coordinates) global positively oriented frame on $\mathcal{T}_{t,u}$.

\subsection{The inverse foliation density}
\label{SS:MUDEF}
We now define $\upmu > 0$, the inverse foliation density of the characteristics $\mathcal{P}_u$.
When $\upmu$ goes to $0$,
the characteristics intersect and, as our main theorem shows,
$\max_{\alpha=0,\cdots,n} |\partial_{\alpha} \Psi|$ blows up. That is, 
\emph{$\upmu \downarrow 0$ signifies the formation of a shock singularity}.

\begin{definition}[\textbf{Inverse foliation density}]
	\label{D:MUDEF}
	We define $\upmu > 0$ as follows:
	\begin{align} \label{E:MUDEF}
		\upmu 
		& := \frac{1}{\partial_t u}.
	\end{align}
\end{definition}

We observe that from \eqref{E:LUNIT0ISONE}-\eqref{E:LUNIT1ISONEFORBACKGROUND}
and \eqref{E:EIKONAL}, it follows that when
$|\Psi| + |\noshock|$ is sufficiently small (as will be the case in our main theorem), 
we have
\begin{align} \label{E:MUINITIALCONDITION}
	\upmu|_{\Sigma_0} = 1 + \mathcal{O}_{\star}(|\Psi|) + \mathcal{O}_{\star}(|\noshock|).
\end{align}
In particular, if $\Psi$ and $\noshock$ are initially small,
then $\upmu$ is initially close to $1$.

\subsection{Vectorfields and one-forms adapted to the characteristics and the blowup-coefficient}
\label{SS:VECTORFIELDSANDONEFORMSADAPTEDTOCHARACTERISTICS}
In this subsection, we construct various vectorfields and one-forms that are
adapted to the characteristics $\mathcal{P}_u$. We also derive some of their basic properties.
We also define the blowup-coefficient, which captures the 
genuinely nonlinear nature of the transport equation \eqref{E:SHOCKEQN}.

\begin{definition}[\textbf{The eikonal function gradient one-forms}]
	\label{D:EIKONALFUNCTIONONEFORMS}
	We define $\uplambda$ and $\upxi$ to be the one-forms with the following Cartesian components,
	$(0 \leq \alpha \leq n)$, $(1 \leq j \leq n)$:
	\begin{subequations}
	\begin{align} \label{E:EIKONALFUNCTIONONEFORMS}
		\uplambda_{\alpha}
		& := \upmu \partial_{\alpha} u,
		&& \\
		\upxi_0
		& := 0,
		&&
		\upxi_j 
		:= \upmu \partial_j u.
		\label{E:EIKONALFUNCTIONSPATIALONEFORM}
	\end{align}
	\end{subequations}
\end{definition}

\begin{remark}
	From \eqref{E:MUDEF} and \eqref{E:EIKONALFUNCTIONONEFORMS}, we deduce that
	\begin{align} \label{E:UPLAMBDA0ISONE}
		\uplambda_0 
		& = 1.
	\end{align}
\end{remark}

The following definition captures the strength of the coefficient of the
main term that drives the shock formation
(as is evidenced by the estimates \eqref{E:MUNOTCOMMUTED}-\eqref{E:LUNITMUNOTCOMMUTED}).
The definition is adapted to the $x^1$
direction since in our main theorem, we study solutions with approximate plane symmetry
(where by plane symmetric solutions, we mean ones that depend only on $t$ and $x^1$).

\begin{definition}[\textbf{The blowup-coefficient}]
	\label{D:GENIUNELYNONLINEARCONSTANT}
	Viewing $\Lunit^1 = \Lunit^1(\Psi,\noshock)$,
	we define the coefficient $\blowupcoeff$ as
	\begin{align} \label{E:BLOWUPCONSTANT}
		\blowupcoeff
		& := \frac{\partial \Lunit^1}{\partial \Psi} \upxi_1.
	\end{align}
\end{definition}

\begin{remark}[$\blowupcoeff \neq 0$]
	\label{R:BLOWUPCOEFFICIENTISNONZERO}
	The solutions that we will study 
	will be such that $\upxi_1$ is a small perturbation of $-1$;
	see definition \eqref{E:UPXIJSMALL} 
	and the estimate \eqref{E:LINFTYUPXIJITSELF}.
	Hence, by \eqref{E:GENUINELYNONLINEAR},
	it follows that
	$\blowupcoeff \neq 0$
	for the solutions under study.
\end{remark}

In the next definition, we define a pair $\mathcal{P}_u$-transversal
vectorfields that we use to study the solution.

\begin{definition}[$\mathcal{P}_u$-\textbf{transversal vectorfields}]
We define the Cartesian components of the $\Sigma_t$-tangent
vectorfields $\Radunit$ and $\Rad$ as follows, 
$(1 \leq j \leq n)$:
\begin{subequations}
\begin{align}
	\Radunit^j
	& := - \Lunit^j,
		\label{E:RADUNITJ} \\
	\Rad^j
	&:= \upmu \Radunit^j
	= - \upmu \Lunit^j.
	\label{E:RADJ}
\end{align}
\end{subequations}
\end{definition}

We now derive some basic properties of $\Lunit$ and $\Rad$.

\begin{lemma}[\textbf{Basic properties of} $\Lunit$ \textbf{and} $\Rad$]
Relative to the geometric coordinates, we have
\begin{align} \label{E:LISDDT}
	\Lunit = \frac{\partial}{\partial t}.
\end{align}

Moreover, the following identity holds:
\begin{align} \label{E:RADAPPLIEDTOUISONE}
	\Rad u & = 1.
\end{align}

Finally, there exists an $\mathcal{T}_{t,u}$-tangent vectorfield
$\Xi$ such that
\begin{align} \label{E:RADDECOMPINTOPARTIALUPLUSANGULAR}
	\Rad
	& = \frac{\partial}{\partial u}
		- \Xi.
\end{align}
\end{lemma}

\begin{proof}
	To prove \eqref{E:LISDDT}, we note that 
	$\Lunit u = \Lunit \vartheta^j = 0$ by construction.
	Also taking into account \eqref{E:LUNIT0ISONE}, we conclude \eqref{E:LISDDT}.
	
	To prove \eqref{E:RADAPPLIEDTOUISONE}, we first use
	the eikonal equation \eqref{E:EIKONAL} and the assumption \eqref{E:LUNIT0ISONE} 
	to deduce the identity $\partial_t u = - \Lunit^a \partial_a u$.
	Multiplying this identity by $\upmu$ and
	appealing to definition \eqref{E:MUDEF},
	we deduce that
	$1 = - \upmu \Lunit^a \partial_a u$
	which, in view of definition \eqref{E:RADJ},
	yields \eqref{E:RADAPPLIEDTOUISONE}.
	The existence of an $\mathcal{T}_{t,u}$-tangent vectorfield
	such that 
	\eqref{E:RADDECOMPINTOPARTIALUPLUSANGULAR}
	holds then follows as a simple consequence 
	of \eqref{E:RADAPPLIEDTOUISONE}
	and the identity $\Rad t = 0$ (that is, the fact that $\Rad$ is $\Sigma_t$-tangent).
\end{proof}

\begin{lemma}[\textbf{Basic identities for the eikonal function gradient one-forms}]
	The following identities hold
	\begin{subequations}
	\begin{align}
		\Lunit^{\alpha} \uplambda_{\alpha}
		& = 0,
		&
		\Lunit^a \upxi_a
		& = -1,
			\label{E:LUNITONEFORMCONTRACTIONS} \\
		\Radunit^{\alpha} \uplambda_{\alpha}
		& = 1,
		&
		\Radunit^a \upxi_a
		& = 1.
		\label{E:RADUNITONEFORMCONTRACTIONS}
	\end{align}
	Moreover, if $Y$ is an $\mathcal{T}_{t,u}$-tangent vectorfield, then
	\begin{align} \label{E:TORUSTANGENTVECTORFIELDONEFORMCONTRACTIONS}
		Y^{\alpha} \uplambda_{\alpha}
		& = 0,
		&
		Y^a \upxi_a
		& = 0.
	\end{align}
	\end{subequations}
\end{lemma}

\begin{proof}
	The identities in \eqref{E:LUNITONEFORMCONTRACTIONS}
	are a straightforward consequence of equation \eqref{E:EIKONAL},
	definitions \eqref{E:EIKONALFUNCTIONONEFORMS}-\eqref{E:EIKONALFUNCTIONSPATIALONEFORM},
	\eqref{E:LUNIT0ISONE}, 
	and \eqref{E:UPLAMBDA0ISONE}.
	The identities in \eqref{E:RADUNITONEFORMCONTRACTIONS}
	follow from 
	\eqref{E:RADJ},
	\eqref{E:RADAPPLIEDTOUISONE},
	definitions \eqref{E:EIKONALFUNCTIONONEFORMS}-\eqref{E:EIKONALFUNCTIONSPATIALONEFORM},
	and the fact that $\Radunit^0 = 0$.
	To obtain 
	\eqref{E:TORUSTANGENTVECTORFIELDONEFORMCONTRACTIONS}, 
	we first note that 
	for $\mathcal{T}_{t,u}$-tangent vectorfields $Y$, we have
	$Y \in \mbox{\upshape span} \lbrace \CoordAng{i} \rbrace_{i=2,\cdots,n}$ and thus
	$Y^{\alpha} \partial_{\alpha} u = 0$.
	The identities in \eqref{E:TORUSTANGENTVECTORFIELDONEFORMCONTRACTIONS} 
	follow from this fact,
	definitions \eqref{E:EIKONALFUNCTIONONEFORMS}-\eqref{E:EIKONALFUNCTIONSPATIALONEFORM},
	and the fact that $Y^0 = 0$.
	
\end{proof}

To obtain estimates for the solution's derivatives,
we will commute the equations with the
vectorfields belonging to the following sets.

\begin{definition}[\textbf{Sets of geometric commutation vectorfields}]
	\label{D:COMMUTATIONVECTORFIELDS}
	We define the following sets of commutation vectorfields:
	\begin{subequations}
	\begin{align} \label{E:COMMUTATIONVECTORFIELDS}
		\Fullset 
		& := \left\lbrace	
					\Lunit,
					\Rad,
					\CoordAng{2},
					\CoordAng{3},
					\cdots,
					\CoordAng{n}
			\right\rbrace,
				\\
			\Tanset 
		& := \left\lbrace	
					\Lunit,
					\CoordAng{2},
					\CoordAng{3},
					\cdots,
					\CoordAng{n}
			\right\rbrace.
		\label{E:TANGENTIALCOMMUTATIONVECTORFIELDS}
	\end{align}
	\end{subequations}
\end{definition}

\begin{remark}
	Note that $\Tanset$ consists of the $\mathcal{P}_u$-tangent elements of $\Fullset$.
\end{remark}

\subsection{Perturbed parts of various scalar functions}
\label{SS:PERTURBEDPARTS}
In this subsection, we define the perturbed parts of
various scalar functions that we have constructed.
The perturbed quantities, which are decorated with the subscript ``$Small$,''
vanish for the background solution $(\Psi,\noshock) = (0,0)$.

\begin{definition}[\textbf{The perturbed parts of various scalar functions}]
	\label{D:PERTURBEDPART}
	Let $\Lunit$ be the vectorfield from equation \eqref{E:SHOCKEQN},
	let $\lbrace \CoordAng{i} \rbrace_{i=2,\cdots,n}$ 
	be the geometric torus vectorfields from \eqref{E:GEOMETRCICOORDINATEPARTIALDERIVATIVEVECTORFIELDS},
	and let $\upxi$ be the one-form defined in \eqref{E:EIKONALFUNCTIONSPATIALONEFORM}.
	We define the following ``background'' quantities, which are constants, $(j=1,\cdots,n)$:
	\begin{subequations}
	\begin{align}
		\Lunitback^j
		&:= \Lunit^j|_{(\Psi,\noshock)=(0,0)},
			\label{E:LUNITBACKJ} \\
		\Radunitback^j
		&:= \Radunit^j|_{(\Psi,\noshock)=(0,0)}.
		\label{E:RADUNITBACKJ}
	\end{align}
	\end{subequations}
	In \eqref{E:LUNITBACKJ}-\eqref{E:RADUNITBACKJ}, we are viewing
	$\Lunit^j$ and $\Radunit^j$ to be functions of $(\Psi,\noshock)$
	(this is possible for $\Radunit^j$ by \eqref{E:RADUNITJ}).
	Note that by \eqref{E:LUNIT1ISONEFORBACKGROUND} and \eqref{E:RADUNITJ},
	we have
	\begin{align} \label{E:BACKGROUND1COMPONENTS}
		\Lunitback^1 
		& = 1,
		&
		\Radunitback^1
		& = -1.
	\end{align}
	
	We also define the following perturbed quantities:
	\begin{subequations}
	\begin{align}
		\Lunit_{(Small)}^j
		& := \Lunit^j - \Lunitback^j,
			\label{E:LUNITSMALLJ} \\
		\Radunit_{(Small)}^j
		& := \Radunit^j - \Radunitback^j
			= - \Lunit_{(Small)}^j,
			\label{E:RADUNITSMALLJ} \\	
		\CoordAngSmallcomp{i}{j}
		& := \CoordAngcomp{i}{j}
			- 
			\delta^{ij},
			\label{E:COORDANGJKSMALL} \\
		\upxi_j^{(Small)}
		& := \upxi_j + \delta_j^1,
			\label{E:UPXIJSMALL}
	\end{align}
	\end{subequations}
	where the second equality in \eqref{E:RADUNITSMALLJ} follows from \eqref{E:RADUNITJ}
	and $\delta^{ij}$ and $\delta_j^1$ are standard Kronecker deltas.
\end{definition}

\subsection{Arrays of unknowns and schematic notation}
\label{SS:ARRAYS}
We use the following arrays for convenient shorthand notation.

\begin{definition}[\textbf{Shorthand notation for various solution variables}]
\label{D:SHORTHANDARRAYS}
We define the following arrays $\GdVar$ and $\BadVar$ of scalar functions:
\begin{subequations}
	\begin{align}
		\GdVar
		& := (\Psi,\noshockuparg{J},\diffnoshockdoublearg{\alpha}{J},
		\upxi_i^{(Small)},\CoordAngSmallcomp{j}{k})_{0 \leq \alpha \leq n,\, 1 \leq i,\, k \leq n,\, 2\leq j \leq n,\, 1 \leq J \leq M},
			\label{E:GOODARRAY} \\
		\BadVar
		& := (\upmu,\Psi,\noshockuparg{J},\diffnoshockdoublearg{\alpha}{J},\upxi_i^{(Small)},\CoordAngSmallcomp{j}{k})_{0 \leq \alpha \leq n,
		\, 1 \leq i,\, k \leq n,\, 2 \leq j \leq n,\, 1 \leq J \leq M}.
		\label{E:BADARRAY}
	\end{align}
	\end{subequations}
\end{definition}

\begin{remark}[\textbf{Schematic functional dependence}]
\label{R:SCHEMATICTENSORFIELDPRODUCTS}
In the remainder of the article, we use the notation
$\smoothfunction(s_1,s_2,\cdots,s_m)$ to schematically depict
an expression 
that depends smoothly on the scalar functions $s_1,s_2,\cdots,s_m$.
Note that in general, $\smoothfunction(0) \neq 0$.
\end{remark}

\begin{remark}[\textbf{The meaning of the symbol} $\Singletan$]
	\label{R:MEANINGOFSINGLETAN}
	Throughout, $\Singletan$ schematically denotes a 
	differential operator that is tangent to the characteristics 
	$\mathcal{P}_u$, typically $\Lunit$ or $\CoordAng{i}$.
	We use such notation when
	the precise details of $\Singletan$ are not important.
\end{remark}

\subsection{Cartesian partial derivatives in terms of geometric vectorfields}
\label{SS:CARTESIANVECTORFIELDSINTERMSOFGEOMETRIC}
In the next lemma, we expand the vectorfields $\lbrace \partial_{\alpha} \rbrace_{\alpha = 0,\cdots,n}$
in terms of the geometric commutation vectorfields.

\begin{lemma}[\textbf{Cartesian partial derivatives in terms of geometric vectorfields}]
	\label{L:CARTESIANDERIVATIVESINTERMSOFGEOMETRICDERIVATIVES}
	There exist smooth scalar functions
	$\smoothfunction_{ij}(\GdVar)$ such that
	the vectorfields $\partial_{\alpha}$
	can be expanded as follows 
	in terms of the elements of the set
	$\Fullset$ defined in \eqref{E:COMMUTATIONVECTORFIELDS}
	whenever $|\GdVar|$ is sufficiently small,
	where $\upxi_j$ is defined in \eqref{E:EIKONALFUNCTIONSPATIALONEFORM}:
	\begin{subequations}
	\begin{align}
		\partial_t
		& = \Lunit 
				+ 
				\Radunit,
				&&
			\label{E:PARTIALTINTERMSOFGEO} \\
		\partial_j
		& = 	
				\upxi_j
				\Radunit
				+
				\sum_{i=2}^n
				\smoothfunction_{ij}(\GdVar)
				\CoordAng{i},
				\label{E:PARTIALJINTERMSOFGEO}
				&& (1 \leq j \leq n).
	\end{align}
	\end{subequations}
\end{lemma}

\begin{proof}
	\eqref{E:PARTIALTINTERMSOFGEO} follows from \eqref{E:LUNIT0ISONE} and \eqref{E:RADUNITJ}.
	
	To prove \eqref{E:PARTIALJINTERMSOFGEO}, 
	we first note that for any fixed
	$j$ with $1 \leq j \leq n$, 
	since $\partial_j$ is $\Sigma_t$-tangent
	and since
	$\lbrace \Radunit, \CoordAng{2}, \cdots, \CoordAng{n} \rbrace$
	spans the tangent space of $\Sigma_t$,
	there exist unique ($j$-dependent) scalars $\upalpha_1,\cdots,\upalpha_n$
	such that
	$\partial_j = \upalpha_1 \Radunit + \sum_{i=2}^n \upalpha_i \CoordAng{i}$.
	Using both sides of this expansion to differentiate the eikonal function $u$ 
	and using \eqref{E:RADJ} and \eqref{E:RADAPPLIEDTOUISONE},
	we obtain the identity $\partial_j u = \upalpha_1 \upmu^{-1}$. 
	In view of definition \eqref{E:EIKONALFUNCTIONSPATIALONEFORM},
	we conclude that $\upalpha_1 = \upxi_j$, as is stated on RHS~\eqref{E:PARTIALJINTERMSOFGEO}.
	Next, for $1 \leq j,k \leq n$, we allow both sides of the expansion to differentiate the Cartesian coordinate $x^k$
	to obtain the identity
	$\delta_j^k = \upalpha_1 \Radunit^k + \sum_{i=2}^n \upalpha_i \CoordAngcomp{i}{k}$.
	For fixed $j$, we can view this as an identity whose left-hand side is the $n$-dimensional vector
	with components $(\delta_j^1,\cdots,\delta_j^n)^{\top}$
	and whose right-hand side is equal to the product of a matrix $M_{n \times n}$ and the $n$-dimensional vector
	$(\upalpha_1,\cdots,\upalpha_n)^{\top}$, where $\top$ denotes transpose.
	From Def.\ \ref{D:PERTURBEDPART}, we see that 
	$
	M_{n \times n}
	=
		\left(
		\begin{array}{c|c}
			-1 & \multirow{1}{*}{$\mathbf{0}_{1 \times (n-1)}$} \\
			\hline
			\multicolumn{1}{c|}{*_{(n-1) \times 1}} & \mathbb{I}_{(n-1) \times (n-1)}
		\end{array}
		\right)
	+
	M_{n \times n}^{(Small)}
	$,
	where the entries of $*_{(n-1) \times 1}$ are of the schematic form $\smoothfunction(\GdVar)$
	and the entries of $M_{n \times n}^{(Small)}$ are of the schematic form $\GdVar \smoothfunction(\GdVar)$
	(and thus are small when $|\GdVar|$ is small). Hence, when $|\GdVar|$ is small, 
	we can invert $M_{n \times n}$
	to conclude that the $\upalpha_i$ are smooth functions of $\GdVar$,
	which completes the proof of \eqref{E:PARTIALJINTERMSOFGEO}.
\end{proof}

\subsection{Evolution equations for the Cartesian components of various geometric quantities}
\label{SS:EVOLUTIONEQUATIONS}
In this subsection, we derive evolution equations for 
the Cartesian components of various geometric quantities
that are adapted to the characteristics $\mathcal{P}_u$.
Later, we will use these transport equations 
to derive estimates for these quantities.

\begin{lemma}[\textbf{Transport equations for} $\upmu$, $\upxi_j$, \textbf{and} $\upxi_j^{(Small)}$]
	\label{L:EVOLUTIONFORUPMUANDUPXI}
	The scalar functions $\upmu$,
	$\upxi_j$,
	and $\upxi_j^{(Small)}$,
	which are defined respectively in 
	\eqref{E:MUDEF}
	\eqref{E:EIKONALFUNCTIONSPATIALONEFORM},
	and \eqref{E:UPXIJSMALL},
	verify the following transport equations,
	where the scalar functions $\smoothfunction_{ij}(\GdVar)$ are as in 
	Lemma~\ref{L:CARTESIANDERIVATIVESINTERMSOFGEOMETRICDERIVATIVES},
	$(i=2,\cdots,n)$,
	$(j=1,\cdots,n)$:
	\begin{subequations}
	\begin{align}
		\Lunit \upmu
		& = 
			(\Rad \Lunit^a) \upxi_a
			+
			\upmu (\Lunit \Lunit^a) \upxi_a,
			&&
			\label{E:LUNITUPMU} \\
		\Lunit \upxi_j
		= \Lunit \upxi_j^{(Small)}
		& = (\Lunit \Lunit^a) \upxi_a \upxi_j
				- 
				\sum_{i=2}^n
				\smoothfunction_{ij}(\GdVar)
				(\CoordAng{i} \Lunit^a) \upxi_a.
			\label{E:LUNITUPXIJ}
	\end{align}
	\end{subequations}
	
	Moreover, there exist functions that are smooth whenever
	$|\GdVar|$ is sufficiently small and that are
	schematically denoted by indexed versions of ``$\smoothfunction$'',
	such that the following initial conditions hold along $\Sigma_0$:
	\begin{subequations}
	\begin{align} \label{E:MUIC}
		\upmu|_{\Sigma_0}
		& = 1 + (\Psi,\noshock) \cdot \smoothfunction(\Psi,\noshock),
			\\
		\upxi_j|_{\Sigma_0}
		& =\left\lbrace
				- 1 + (\Psi,\noshock) \cdot \smoothfunction(\Psi,\noshock)
			\right\rbrace
			\delta_j^1,
				\label{E:XIJIC} \\
		\upxi_j^{(Small)}|_{\Sigma_0}
		& = (\Psi,\noshock) \cdot \smoothfunction(\Psi,\noshock) \delta_j^1.
			\label{E:XIJSMALLIC}
	\end{align}
	\end{subequations}
\end{lemma}

\begin{proof}
	Differentiating the eikonal equation \eqref{E:EIKONAL} with $\partial_{\alpha}$ and using
	\eqref{E:LUNIT0ISONE}, we obtain
	\begin{align} \label{E:EIKONALCARTESIANDIFFERENTIATED}
		\Lunit \partial_{\alpha} u 
		& = 
		- 
		(\partial_{\alpha} \Lunit^a) \partial_a u.
	\end{align}
	Setting $\alpha = 0$ in \eqref{E:EIKONALCARTESIANDIFFERENTIATED} and appealing to definition \eqref{E:MUDEF},
	we deduce
	\begin{align} \label{E:UPMUEVOLUTIONALMOSTDERIVED}
		\Lunit \upmu
		& = \upmu (\partial_t \Lunit^a) (\upmu \partial_a u).
	\end{align}
	From \eqref{E:UPMUEVOLUTIONALMOSTDERIVED},
	\eqref{E:PARTIALTINTERMSOFGEO},
	\eqref{E:RADJ},
	and definition \eqref{E:EIKONALFUNCTIONSPATIALONEFORM}, 
	we conclude \eqref{E:LUNITUPMU}.
	 
	Next, we set $\alpha = j$ in \eqref{E:EIKONALCARTESIANDIFFERENTIATED},
	multiply the equation by $\upmu$,
	and use definition \eqref{E:EIKONALFUNCTIONSPATIALONEFORM} 
	and \eqref{E:UPMUEVOLUTIONALMOSTDERIVED} to compute that
	\begin{align} \label{E:LUNITUPXIJALMOSTDERIVED}
		\Lunit (\upmu \partial_j u)
		& = 
		- (\partial_j \Lunit^a) (\upmu \partial_a u)
		+
		(\partial_t \Lunit^a) (\upmu \partial_a u) (\upmu \partial_j u)
			\\
		& = - (\partial_j \Lunit^a) \upxi_a
			+
				(\partial_t \Lunit^a) \upxi_a \upxi_j.
			\notag
	\end{align}
	From \eqref{E:LUNITUPXIJALMOSTDERIVED}
	and \eqref{E:PARTIALTINTERMSOFGEO}-\eqref{E:PARTIALJINTERMSOFGEO},
	we conclude \eqref{E:LUNITUPXIJ}.
	
	To prove \eqref{E:MUIC},
	we use \eqref{E:LUNIT0ISONE}-\eqref{E:LUNIT1ISONEFORBACKGROUND},
	\eqref{E:EIKONAL},
	and definition \eqref{E:MUDEF}
	to obtain
	$(1/\upmu)|_{\Sigma_0} = \partial_t u|_{\Sigma_0} = - \Lunit^a \partial_a u|_{\Sigma_0} = \Lunit^1|_{\Sigma_0}
	= 1 + (\Psi,\noshock) \cdot \smoothfunction(\Psi,\noshock)
	$,
	from which \eqref{E:MUIC} easily follows (when $|\Psi|$ and $|\noshock|$ are small).
	To prove \eqref{E:XIJIC}, we use definition \eqref{E:EIKONALFUNCTIONSPATIALONEFORM}
	and the above argument to deduce that
	$\upxi_j|_{\Sigma_0} = - (\upmu \delta_j^1)|_{\Sigma_0} 
	= 
	\left\lbrace
		- 1 + (\Psi,\noshock) \cdot \smoothfunction(\Psi,\noshock)
	\right\rbrace
	\delta_j^1$,
	as desired.
	\eqref{E:XIJSMALLIC} then follows from 
	\eqref{E:XIJIC} and definition \eqref{E:UPXIJSMALL}.
\end{proof}

In the next lemma, we derive transport equations for the Cartesian components of the
geometric torus coordinate partial derivative vectorfields.

\begin{lemma}[\textbf{Transport equations for the Cartesian components of} $\CoordAng{i}$]
	\label{L:COORDANGCOMPONENTRECTANGULAR}
	The Cartesian components 
	$\CoordAngcomp{i}{j}$ of the
	$\mathcal{T}_{t,u}$-tangent vectorfields
	from
	\eqref{E:GEOMETRCICOORDINATEPARTIALDERIVATIVEVECTORFIELDS}
	and their perturbed parts $\CoordAngSmallcomp{i}{j}$ defined in \eqref{E:COORDANGJKSMALL}
	are solutions to the following transport equation initial value problem:
	\begin{subequations}
	\begin{align} 
		\Lunit \CoordAngcomp{i}{j}
		& = \CoordAng{i} \Lunit^j,
		&& 
		\CoordAngcomp{i}{j}|_{\Sigma_0}
		= \delta^{ij},
		\label{E:COORDANGCOMPONENTRECTANGULAR}
			\\
		\Lunit \CoordAngSmallcomp{i}{j}
		& = \CoordAng{i} \Lunit^j,
		&& 
		\CoordAngcomp{i}{j}|_{\Sigma_0}
		= 0.
		\label{E:COORDANGSMALLCOMPONENTRECTANGULAR}
	\end{align}	
	\end{subequations}
	where $\delta^{ij}$ is the standard Kronecker delta.
\end{lemma}

\begin{proof}
	$\Lunit$ and $\CoordAng{i}$ are geometric coordinate partial derivative vectorfields
	and they therefore commute: $[\Lunit,\CoordAng{i}] = 0$. Relative to Cartesian coordinates,
	the vanishing commutator can be expressed as
	$
	\Lunit \CoordAngcomp{i}{j}
	= \CoordAng{i} \Lunit^j
	$,
	which is the desired evolution equation in \eqref{E:COORDANGCOMPONENTRECTANGULAR}.
	Next, we observe that along $\Sigma_0$, $\CoordAng{i} = \partial_i$ by construction.
	Hence, $\CoordAngcomp{i}{j}|_{\Sigma_0} = \CoordAng{i}|_{\Sigma_0} x^j = \partial_i x^j = \delta^{ij}$,
	which yields the initial condition \eqref{E:COORDANGCOMPONENTRECTANGULAR}.
	\eqref{E:COORDANGSMALLCOMPONENTRECTANGULAR} then follows from definition \eqref{E:COORDANGJKSMALL}
	and \eqref{E:COORDANGCOMPONENTRECTANGULAR}.
\end{proof}

\subsection{Vectorfield commutator properties}
\label{SS:COMMUTATIONIDENTITEIS}
In this subsection, we derive some basic properties of various vectorfield commutators.

\begin{lemma}[\textbf{Vectorfield commutator properties}]
	\label{L:VECTORFIELDCOMMUTATORIDENTITIES}
		The following vectorfields are $\mathcal{T}_{t,u}$-tangent, $(i=2,\cdots,n)$:
	\begin{align} \label{E:COMMUTATORARELTUTANGENT}
		[\Lunit, \Rad],
			\qquad
		[\Lunit, \CoordAng{i}],
			\qquad
		[\Rad, \CoordAng{i}],
		&&
		(i=2,\cdots,n).
	\end{align}
	
	Moreover, there exist smooth functions,
	denoted by subscripted versions of ``$\smoothfunction$'', such that 
	the following identities hold 
	whenever $|\GdVar|$ is sufficiently small
	(see Remark~\ref{R:MEANINGOFSINGLETAN} regarding the notation):
	\begin{subequations}
	\begin{align}
			[\Lunit, \CoordAng{i}]
			=
			[\CoordAng{i_2},\CoordAng{i_2}]
		& = 0,
				\label{E:LUNITANGULARCOMMUTATOR} \\
		[\Lunit, \Rad]
		& = \sum_{i=2}^n \smoothfunction_i(\BadVar,\Lunit \Psi, \Rad \Psi) \CoordAng{i},
				\label{E:LUNITRADCOMMUTATOR} \\
		[\Rad, \CoordAng{i}]
		& = \sum_{j=2}^n \smoothfunction_{ij}(\BadVar,\Rad \GdVar,\Singletan \Psi,\Singletan \upmu) \CoordAng{j}.
		\label{E:RADCOORDANGCOMMUTATOR}
\end{align}
\end{subequations}	
\end{lemma}

\begin{proof}
	Since \eqref{E:LISDDT} implies that $\Lunit$ is a geometric coordinate partial derivative vectorfield
	and since, by definition, the same is true of $\CoordAng{i}$,
	we conclude \eqref{E:LUNITANGULARCOMMUTATOR}.
	
	To prove \eqref{E:LUNITRADCOMMUTATOR}, we first use
	\eqref{E:LISDDT},
	\eqref{E:RADAPPLIEDTOUISONE},
	and the fact that $\Rad$ is $\Sigma_t$-tangent
	to deduce that
	$[\Lunit, \Rad] t = [\Lunit, \Rad] u = 0$. Hence, $[\Lunit, \Rad]$ is $\mathcal{T}_{t,u}$-tangent.
	Therefore, there exist unique scalars $\upalpha_i$ such that the following identity holds
	for $j=1,2,\cdots,n$:
	$[\Lunit, \Rad]^j = \sum_{i=2}^n \upalpha_i \CoordAngcomp{i}{j}$.
	Next, we use the fact that $\Lunit^a = \smoothfunction(\Psi,\noshock)$,
	\eqref{E:RADUNITJ}-\eqref{E:RADJ},
	and the evolution equation \eqref{E:LUNITUPMU}
	to deduce the schematic identity
	$
	[\Lunit, \Rad]^j = \Rad \Lunit^j - \Lunit(\upmu \Radunit^j)
	= \smoothfunction(\BadVar,\Lunit^{\alpha} \diffnoshockdownarg{\alpha},\Rad^a \diffnoshockdownarg{a},\Lunit \Psi,\Rad \Psi)
	= \smoothfunction(\BadVar,\Lunit \Psi,\Rad \Psi)
	$.
	Next, for $i,j=2,\cdots,n$, we view the identity
	$[\Lunit, \Rad]^j = \sum_{i=2}^n \upalpha_i \CoordAngcomp{i}{j}$
	as an identity whose left-hand side is the $n-1$ dimensional vector
	with Cartesian components equal to 
	$([\Lunit, \Rad]^2,\cdots,[\Lunit, \Rad]^n)^{\top}$
	and whose right-hand side is the product of the
	$(n-1) \times (n-1)$ matrix 
	$M_{(n-1) \times (n-1)} := (\CoordAngcomp{i}{j})_{i,j=2,\cdots,n}$
	and the $n-1$ dimensional vector $(\upalpha_2,\cdots,\upalpha_n)^{\top}$,
	where $\top$ denotes transpose.
	From definition \eqref{E:COORDANGJKSMALL},
	we see that $M_{(n-1) \times (n-1)}$ is equal to the identity 
	matrix plus an error matrix whose components are of the schematic
	form $\GdVar \smoothfunction(\GdVar)$. 
	In particular, $M_{(n-1) \times (n-1)}$ is invertible whenever
	$|\GdVar|$ is sufficiently small. 
	Hence,
	$(\upalpha_2,\cdots,\upalpha_n)^{\top}$ 
	is the product of a matrix whose components are of the form $\smoothfunction(\GdVar)$
	and the vector $([\Lunit, \Rad]^2,\cdots,[\Lunit, \Rad]^n)^{\top}$, whose components 
	are of the form $\smoothfunction(\BadVar,\Lunit \Psi,\Rad \Psi)$.
	This completes the proof of \eqref{E:LUNITRADCOMMUTATOR}.
	The identity \eqref{E:RADCOORDANGCOMMUTATOR} can be proved in a similar fashion
	and we omit the details.
	\end{proof}

\begin{corollary}[\textbf{Evolution equation for} $\Xi^j$]
	\label{C:BIGXIJEVOLUTIONEQUATION}
	There exist functions that are smooth whenever
	$|\GdVar|$ is sufficiently small and that are
	schematically denoted by indexed versions of ``$\smoothfunction$'',
	such that the Cartesian components $\Xi^j$, $(j=1,\cdots,n)$, of the $\mathcal{T}_{t,u}$-tangent
	vectorfield $\Xi$ from \eqref{E:RADDECOMPINTOPARTIALUPLUSANGULAR}
	verify the evolution equation
	\begin{align} \label{E:BIGXIJEVOLUTIONEQUATION}
		\Lunit \Xi^j
		& = 
				\sum_{i=2}^n
				\Xi^a
				\smoothfunction_{ia}(\GdVar)
				\CoordAng{i} \Lunit^j
				-
				\sum_{i=2}^n \smoothfunction_i(\BadVar,\Lunit \Psi, \Rad \Psi) \CoordAngcomp{i}{j}
	\end{align}
	and the initial condition
	\begin{align} \label{E:BIGXIJINITIALCONDITION}
		\Xi^j|_{\Sigma_0},
		& = \smoothfunction^j(\Psi,\noshock) 
	\end{align}
	where the $\smoothfunction_{ia}$ on RHS~\eqref{E:BIGXIJEVOLUTIONEQUATION}
	are as in equation \eqref{E:PARTIALJINTERMSOFGEO},
	and the second term on RHS~\eqref{E:BIGXIJEVOLUTIONEQUATION} is the negative of the
	term on RHS~\eqref{E:LUNITRADCOMMUTATOR}.
\end{corollary}

\begin{proof}
	From 
	\eqref{E:LISDDT}
	and
	\eqref{E:RADDECOMPINTOPARTIALUPLUSANGULAR}, we deduce that
	$[\Lunit,\Xi]^j = - [\Lunit,\Rad]^j$.
	Considering the 
	Cartesian components of both sides of this equation
	and using \eqref{E:LUNITRADCOMMUTATOR},
	we obtain
	$
	\Lunit \Xi^j
	= \Xi^a \partial_a \Lunit^j
	-
				\sum_{i=2}^n \smoothfunction_i(\BadVar,\Lunit \Psi, \Rad \Psi) \CoordAngcomp{i}{j}
	$.
	Finally, we use \eqref{E:PARTIALJINTERMSOFGEO}
	to substitute for $\partial_a$ in the expression
	$\Xi^a \partial_a \Lunit^j$, and 
	we use \eqref{E:TORUSTANGENTVECTORFIELDONEFORMCONTRACTIONS}
	to deduce that the component $\Xi^a \upxi_a \Radunit \Lunit^j$ vanishes.
	In total, this yields equation \eqref{E:BIGXIJEVOLUTIONEQUATION}.
	
	To prove \eqref{E:BIGXIJINITIALCONDITION}, we use
	\eqref{E:RADDECOMPINTOPARTIALUPLUSANGULAR}
	to deduce that $\Xi^j = \Xi x^j = \frac{\partial}{\partial u} x^j - \Rad^j$.
	In view of the way in which the geometric coordinates were constructed, 
	along $\Sigma_0$, we have $\frac{\partial}{\partial u} = - \partial_1$.
	Moreover, in view of \eqref{E:RADUNITJ}-\eqref{E:RADJ} and \eqref{E:MUIC},
	we deduce that
	$\Rad^j|_{\Sigma_0} 
	= \Rad x^j|_{\Sigma_0}  
	= (\upmu \Radunit^j)|_{\Sigma_0} 
	= \upmu|_{\Sigma_0} \smoothfunction(\Psi,\noshock) 
	= \smoothfunction(\Psi,\noshock)$,
	where $\smoothfunction$ depends on $j$.
	Combining the above, we conclude \eqref{E:BIGXIJINITIALCONDITION}.
\end{proof}	

\subsection{The change of variables map}
\label{SS:CHOV}
In this subsection, we define the change of variables map from geometric to 
Cartesian coordinates and derive some of its basic properties.

\begin{definition}[\textbf{Change of variables map}]
	\label{D:CHOV}
	We define 
	$\Upsilon:\mathbb{R} \times \mathbb{R} \times \mathbb{T}^{n-1} \rightarrow \mathbb{R} \times \mathbb{R} \times \mathbb{T}^{n-1}$
	to be the change of variables map from geometric to Cartesian coordinates,
	i.e., $\Upsilon^{\alpha}(t,u,\vartheta^2,\cdots,\vartheta^n) = x^{\alpha}$.
\end{definition}

\begin{lemma}[\textbf{Basic properties of the change of variables map}]
	\label{L:CHOVCALCULATIONS}
	The following identities hold,
	where $\Lunit$ is the vectorfield from \eqref{E:SHOCKEQN},
	the $\CoordAng{i}$ are the vectorfields from \eqref{E:GEOMETRCICOORDINATEPARTIALDERIVATIVEVECTORFIELDS}, 
	$\Rad$ is the vectorfield from \eqref{E:RADJ}, 
	and $\Xi$ is the vectorfield from \eqref{E:RADDECOMPINTOPARTIALUPLUSANGULAR}:
	\begin{align} \label{E:CHOVCALCULATIONS}
	\frac{\partial \Upsilon}{\partial (t,u,\vartheta^2,\cdots,\vartheta^n)}
	:=
		\frac{\partial (x^0,x^1,x^2,\cdots,x^n)}{\partial (t,u,\vartheta^2,\cdots,\vartheta^n)}
	& = \begin{pmatrix}
				1 & 0 & 0 & 0 & \cdots & 0 \\
				\Lunit^1 &  \upmu \Radunit^1 + \Xi^1  & \CoordAngcomp{2}{1} & \CoordAngcomp{3}{1} & \cdots & \CoordAngcomp{n}{1} \\
				\Lunit^2 & \upmu \Radunit^2 + \Xi^2 & \CoordAngcomp{2}{2} & \CoordAngcomp{3}{2} & \cdots & \CoordAngcomp{2}{2} \\
				\cdots & \cdots & \cdots & \cdots & \cdots & \cdots \\
				\Lunit^n & \upmu \Radunit^n + \Xi^n & \CoordAngcomp{2}{n} & \CoordAngcomp{3}{n} & \cdots & \CoordAngcomp{n}{n} 
			\end{pmatrix}.
	\end{align}
	
	Moreover, there exists a smooth function of $\GdVar$ vanishing at $\GdVar = 0$,
	schematically denoted by
	$\GdVar \smoothfunction(\GdVar)$, 
	such that
	\begin{align} \label{E:DETERMINANTOFCHOV}
	\mbox{\upshape{det}}
	\frac{\partial (x^0,x^1,x^2,\cdots,x^n)}{\partial (t,u,\vartheta^2,\cdots,\vartheta^n)}
	= \frac{\partial (x^1,x^2,\cdots,x^n)}{\partial (u,\vartheta^2,\cdots,\vartheta^n)}
	& = 
			- \upmu 
			\left\lbrace
				1 + \GdVar \smoothfunction(\GdVar)
			\right\rbrace.
	\end{align}
	
	Similarly, the following identity holds:
	\begin{align} \label{E:DETERMINANTOFANGULARPARTCHOV}
	\mbox{\upshape{det}}
	\frac{\partial (x^2,\cdots,x^n)}{\partial (\vartheta^2,\cdots,\vartheta^n)}
	& = 1 + \GdVar \smoothfunction(\GdVar).
	\end{align}
	
\end{lemma}

\begin{proof}
	The first column of \eqref{E:CHOVCALCULATIONS}
	is a simple consequence of \eqref{E:LISDDT} and the fact that 
	$\Lunit x^{\alpha} = \Lunit^{\alpha}$. 
	The second column of \eqref{E:RADJ} follows similarly from the fact that
	$\Rad$ is $\Sigma_t$-tangent (i.e., $\Rad t = 0$),
	\eqref{E:RADJ}, and \eqref{E:RADDECOMPINTOPARTIALUPLUSANGULAR}.
	The remaining $n-1$ columns of \eqref{E:RADJ} follow similarly from the fact that
	the vectorfields $\CoordAng{i}$ are $\Sigma_t$-tangent.
	
	The first equality in \eqref{E:DETERMINANTOFCHOV} is a simple consequence of \eqref{E:CHOVCALCULATIONS}.
	To derive the second equality in
	\eqref{E:DETERMINANTOFCHOV}, we first note that since 
	$\Xi \in \mbox{\upshape span} \lbrace \CoordAng{i} \rbrace_{i=2,\cdots,n}$,
	we can delete $\Xi$ from the matrix on RHS~\eqref{E:CHOVCALCULATIONS} without changing its determinant.
	It follows that
	$
	\mbox{\upshape LHS~\eqref{E:DETERMINANTOFCHOV}}
	=
	\upmu
	\mbox{\upshape det}
	\begin{pmatrix}
				\Radunit^1 & \CoordAngcomp{2}{1} & \CoordAngcomp{3}{1} & \cdots & \CoordAngcomp{n}{1} \\
				\Radunit^2 & \CoordAngcomp{2}{2} & \CoordAngcomp{3}{2} & \cdots & \CoordAngcomp{2}{2} \\
				\cdots & \cdots & \cdots & \cdots & \cdots \\
				\Radunit^n & \CoordAngcomp{2}{n} & \CoordAngcomp{3}{n} & \cdots & \CoordAngcomp{n}{n} 
			\end{pmatrix}
	$.
	In view of Def.\ \ref{D:PERTURBEDPART}
	and definition \eqref{E:GOODARRAY},
	we see that the previous expression is equal to 
	$\upmu$ times the determinant of $M_{n \times n} + M_{n \times n}^{(Small)}$,
	where $M_{n \times n}$ and $M_{n \times n}^{(Small)}$
	are the matrices from the proof of Lemma~\ref{L:CARTESIANDERIVATIVESINTERMSOFGEOMETRICDERIVATIVES}.
	Using arguments similar to the ones given in the proof of Lemma~\ref{L:CARTESIANDERIVATIVESINTERMSOFGEOMETRICDERIVATIVES},
	we conclude the identity \eqref{E:DETERMINANTOFCHOV}. 
	The identity \eqref{E:DETERMINANTOFANGULARPARTCHOV}
	can be proved via a similar argument, and we omit the details.
\end{proof}

\subsection{Integration forms and integrals}
\label{SS:INTEGRATIONFORMSANDINTEGRALS}
In this subsection, we define quantities connected to
the two kinds of integration that we use in our analysis:
integration with respect to the geometric coordinates and integration
with respect to the Cartesian coordinates. 
In Remark~\ref{R:ROLEOFCARTESIANFORMS}, we clarify
why both kinds of integration play a role in our analysis and why 
geometric integration is the most important for our analysis.
In Lemma~\ref{L:VOLFORMRELATION}, 
we quantify the relationship between the two kinds of integration.

\subsubsection{Geometric integration}
\label{SSS:GEOMETRICINTEGRATION}

\begin{definition}[\textbf{Geometric forms and related integrals}]
	\label{D:NONDEGENERATEVOLUMEFORMS}
	Relative to the geometric coordinates of Def.\ \ref{D:GEOMETRCICOORDINATES},
	we define the following forms:\footnote{Throughout the paper,
		we blur the distinction between 
		(non-negative) integration measure $d \torusvol$
		and the corresponding form $d \vartheta^2 \wedge \cdots \wedge d \vartheta^n$,
		and similarly for the other quantities appearing in \eqref{E:RESCALEDVOLUMEFORMS}.
		The precise meaning will be clear from context. \label{FN:BLURFORMS}} 
	\begin{align} \label{E:RESCALEDVOLUMEFORMS}
			d \torusvol
			& : = d \vartheta^2 \cdots d \vartheta^n,
				&&
			d \tvol
			:= d \torusvol du',
				\\
			d \conevol 
			&:= d \torusvol dt',
				&&
			d \vol 
			:= d \torusvol du' dt'.
				\notag
	\end{align}
	
	If $f$ is a scalar function, then we define
	\begin{subequations}
	\begin{align}
	\int_{\mathcal{T}_{t,u}}
			f
		\, d \torusvol
		& 
		:=
		\int_{\vartheta \in \mathbb{T}}
			f(t,u,\vartheta)
		\, d \torusvol,
			\label{E:LINEINTEGRALDEF} \\
		\int_{\Sigma_t^u}
			f
		\, d \tvol
		& 
		:=
		\int_{u'=0}^u
		\int_{\vartheta \in \mathbb{T}}
			f(t,u',\vartheta)
		\, d \torusvol du',
			\label{E:SIGMATUINTEGRALDEF} \\
		\int_{\mathcal{P}_u^t}
			f
		\, d \conevol
		& 
		:=
		\int_{t'=0}^t
		\int_{\vartheta \in \mathbb{T}}
			f(t',u,\vartheta)
		\, d \torusvol dt',
			\label{E:PUTINTEGRALDEF} \\
		\int_{\mathcal{M}_{t,u}}
			f
		\, d \vol
		& 
		:=
		\int_{t'=0}^t
		\int_{u'=0}^u
		\int_{\vartheta \in \mathbb{T}}
			f(t',u',\vartheta)
		\, d \torusvol du' dt'.
		\label{E:MTUTUINTEGRALDEF}
	\end{align}
	\end{subequations}
\end{definition}

\subsubsection{Cartesian integration}
\label{SSS:INTEGRATIONWITHRESPECTTOCARTESIAN}

\begin{definition}[\textbf{The one-form} $\covL$]
	\label{D:EUCLIDEANNORMALTONULLHYPERSURFACE}
	Let $\uplambda$ be the one-form from Def.\ \ref{D:EIKONALFUNCTIONONEFORMS}.
 	We define $\covL$ to be the one-form with the following Cartesian components:
	\begin{align} \label{E:EUCLIDEANNORMALTONULLHYPERSURFACE}
		\covL_{\nu}
		& := \frac{1}{(\delta^{\alpha \beta} \uplambda_{\alpha} \uplambda_{\beta})^{1/2}} \uplambda_{\nu},
	\end{align}	
	where $\delta^{\alpha \beta}$ is the standard inverse Euclidean metric on $\mathbb{R} \times \Sigma$
	(that is, $\delta^{\alpha \beta} = \mbox{\upshape diag} (1,1,\cdots,1)$ relative to the Cartesian coordinates). 
	Note that $\covL$ is the Euclidean-unit-length co-normal to $\mathcal{P}_u$.
\end{definition}

\begin{definition}[\textbf{Cartesian coordinate volume and area forms and related integrals}]
	\label{D:CARTESIANFORMS}
	We define
	\begin{align} \label{E:CARTESIANFORMS}
	d \mathcal{M} := dx^1 dx^2 \cdots dx^n dt,
	\qquad
	d \Sigma := dx^1 dx^2 \cdots dx^n, 
	\qquad
	d \mathcal{P}
	\end{align}
	to be, respectively, the standard volume form on
	$\mathcal{M}_{t,u}$ induced by the Euclidean metric\footnote{By definition, the Euclidean metric
	has the components $\mbox{\upshape diag}(1,1,\cdots,1)$ relative to the standard
	Cartesian coordinates $(t,x^1,x^2,\cdots,x^n)$ on $\mathbb{R} \times \Sigma$.} 
	on $\mathbb{R} \times \Sigma$,
	the standard area form induced on 
	$\Sigma_t^u$ by the Euclidean metric on $\mathbb{R} \times \Sigma$,
	and the standard area form induced on
	$\mathcal{P}_u^t$ by the Euclidean metric on $\mathbb{R} \times \Sigma$.

	We define the integrals of functions $f$ with respect to the above forms in analogy with the way
	that we defined the integrals \eqref{E:LINEINTEGRALDEF}-\eqref{E:MTUTUINTEGRALDEF}.
	For example,
	\[
	\int_{\Sigma_t^U}
			f
		\, d \Sigma
	:=
		\int_{\lbrace (x^1,\cdots,x^n) \ | \ 0 \leq u(t,x^1,\cdots,x^n) \leq U \rbrace}
			f(t,x^1,\cdots,x^n)
		\, d x^1 \cdots dx^n,
\]
where $u(t,x^1,\cdots,x^n)$ is the eikonal function.
	
\end{definition}

\begin{remark}[\textbf{The role of the Cartesian forms}]
	\label{R:ROLEOFCARTESIANFORMS}
	We never \emph{estimate} integrals involving the Cartesian forms;
	before deriving estimates,
	we will always use Lemma~\ref{L:VOLFORMRELATION} below
	order to replace the Cartesian forms with the geometric ones of
	Def.\ \ref{D:NONDEGENERATEVOLUMEFORMS};
	we use the Cartesian forms only when deriving energy \emph{identities}
	relative to the Cartesian coordinates, in which the Cartesian forms naturally appear.
\end{remark}

\subsubsection{Comparison between the Cartesian integration measures and the geometric integration measures}
\label{SSS:CARTESIANEUCLIDEANFORMCOMPARISON}
In the next lemma, we quantify the relationship between 
the Cartesian integration measures and the geometric integration measures.

\begin{lemma}[\textbf{Comparison between Euclidean and geometric integration measures}]
	\label{L:VOLFORMRELATION}
	There exist scalar functions, schematically denoted by
	$\smoothfunction(\GdVar)$, 
	that are smooth for $|\GdVar|$ sufficiently small
	and such that
	the following relationship holds
	between the geometric integration measures corresponding to Def.\ \ref{D:NONDEGENERATEVOLUMEFORMS}
	and the Euclidean integration measures corresponding to Def.\ \ref{D:CARTESIANFORMS},
	where all of the measures are non-negative (see Footnote~\ref{FN:BLURFORMS}):
	\begin{align} \label{E:VOLFORMRELATION}
		d \mathcal{M}
		& = 
				\upmu 
				\left\lbrace 
					1 + \upgamma \smoothfunction(\upgamma)
				\right\rbrace 
				d \vol,
		&
		d \Sigma 
		& = 
				\upmu
				\left\lbrace 
					1 + \upgamma \smoothfunction(\upgamma)
				\right\rbrace 
				d \tvol,
		&
		d \mathcal{P}
		& =  
			\left\lbrace 
				\sqrt{2} + \GdVar \smoothfunction(\GdVar)
			\right\rbrace 
			d \conevol. 	
	\end{align}
\end{lemma}

\begin{proof}
	We prove only the identity
	$
	d \mathcal{P}
	=  
			\left\lbrace 
				\sqrt{2} + \GdVar \smoothfunction(\GdVar)
			\right\rbrace 
			d \conevol
	$
	since the other two 
	identities in \eqref{E:VOLFORMRELATION}
	are a straightforward consequence of
	Lemma~\ref{L:CHOVCALCULATIONS} 
	(in particular, the Jacobian determinant\footnote{Note that
	the minus sign in equation \eqref{E:DETERMINANTOFCHOV}
	does not appear in equation \eqref{E:VOLFORMRELATION}
	since we are viewing \eqref{E:VOLFORMRELATION} as a relationship between integration measures.} 
	expressions in \eqref{E:DETERMINANTOFCHOV}).
	Throughout this proof, 
	we view $d \conevol$ 
	(see \eqref{E:RESCALEDVOLUMEFORMS})
	to be the $n$-form $dt \wedge d \vartheta^2 \wedge \cdots \wedge d \vartheta^n$ on $\mathcal{P}_u$,
	where $dt \wedge d \vartheta^2 = dt \otimes d \vartheta^2 - d \vartheta^2 \otimes dt$, etc.
	Similarly, we view $d \mathcal{P}$ to be the $n$-form induced on 
	$\mathcal{P}_u$ by the standard Euclidean metric on $\mathbb{R} \times \Sigma$.
	Then relative to Cartesian coordinates, 
	we have $d \mathcal{P} = (dx^0 \wedge dx^1 \wedge \cdots \wedge dx^n) \cdot W$,
	where $W$ is the future-directed Euclidean normal to 
	$\mathcal{P}_u$ and $(dx^0 \wedge dx^1 \wedge \cdots \wedge dx^n) \cdot W$ denotes contraction 
	of $W$ against the first slot of $dx^0 \wedge dx^1 \wedge \cdots \wedge dx^n$.
	Note that $W^{\alpha} = \delta^{\alpha \beta} \covL_{\beta}$,
	where $\covL_{\alpha}$ is defined in \eqref{E:EUCLIDEANNORMALTONULLHYPERSURFACE}
	and $\delta^{\alpha \beta} = \mbox{\upshape diag} (1,1,\cdots,1)$ is the standard inverse Euclidean metric
	on $\mathbb{R} \times \Sigma$.
	Since $d \conevol$ and $d \mathcal{P}$ are proportional
	and since $(dt \wedge d \vartheta^2 \wedge \cdots \wedge d \vartheta^n) 
	\cdot (\Lunit \otimes \CoordAng{2} \otimes \cdots \otimes \CoordAng{n}) = 1$,
	it suffices to show that
	$
	\left\lbrace 
		\sqrt{2} + \GdVar \smoothfunction(\GdVar)
	\right\rbrace 
	=  
	(dx^0 \wedge dx^1 \wedge \cdots \wedge dx^n) \cdot (W \otimes \Lunit \otimes \CoordAng{2} \otimes \cdots \otimes \CoordAng{n})
	$.
	To proceed, we note that
	$(dx^0 \wedge dx^1 \wedge \cdots \wedge dx^n) \cdot (W \otimes \Lunit \otimes \CoordAng{2} \otimes \cdots \otimes \CoordAng{n})$
	is equal to the determinant of the $(1+n) \times (1+n)$ matrix
	$
	N:
	=
	\left(
		\begin{array}{ccccc}
			W^0 & \Lunit^0 & 0 & \cdots & 0 \\
			W^1 & \Lunit^1 & \CoordAngcomp{2}{1} & \cdots & \CoordAngcomp{n}{1} \\
			\cdots & \cdots & \cdots & \cdots & \cdots \\
			W^n & \Lunit^n & \CoordAngcomp{2}{n} & \cdots & \CoordAngcomp{n}{n} 
		\end{array}
		\right)
	$.
	From 
	\eqref{E:LUNIT0ISONE}-\eqref{E:LUNIT1ISONEFORBACKGROUND},
	Def.\ \ref{D:EIKONALFUNCTIONONEFORMS},
	\eqref{E:UPLAMBDA0ISONE},
	Def.\ \ref{D:PERTURBEDPART},
	definition \eqref{E:GOODARRAY},
	definition \eqref{E:EUCLIDEANNORMALTONULLHYPERSURFACE},
	and the relation $W^{\alpha} = \delta^{\alpha \beta} \covL_{\beta}$,
	it follows that
	$
	N=
	\left(
		\begin{array}{cc|c}
			\frac{\sqrt{2}}{2} & 1 & \multirow{2}{*}{$\mathbf{0}_{2 \times (n-1)}$} \\
			- \frac{\sqrt{2}}{2} & 1 & \\
			\hline
			\multicolumn{2}{c|}{*_{(n-1) \times 2}} & \mathbb{I}_{(n-1) \times (n-1)}
		\end{array}
		\right)
	+ N^{(Small)}
	$,
	where the entries of the submatrix $*_{(n-1) \times 2}$
	are of the schematic form $\smoothfunction(\GdVar)$,
	$\mathbb{I}_{(n-1) \times (n-1)}$ is the identity matrix,
	and $N^{(Small)}$ is a matrix whose entries are all of the schematic form $\GdVar \smoothfunction(\GdVar)$,
	where $\smoothfunction$ is smooth. From these facts and the basic properties 
	of the determinant, we conclude that 
	$\mbox{\upshape det} N = \sqrt{2} + \GdVar \smoothfunction(\GdVar)$,
	which is the desired identity.
	
\end{proof}

\subsection{Notation for repeated differentiation}
\label{SS:NOTATIONFORREPEATED}
In this subsection, we define some notation that we use when performing repeated differentiation.

\begin{definition}[\textbf{Notation for repeated differentiation}]
\label{D:REPEATEDDIFFERENTIATIONSHORTHAND}
Recall that the commutation vectorfield sets $\Fullset$ and $\Tanset$ are defined in Def.\ \ref{D:COMMUTATIONVECTORFIELDS}.
We label the $n+1$ vectorfields in 
$\Fullset$ as follows: $Z_{(1)} = \Lunit, Z_{(2)} = \CoordAng{2}, Z_{(3)} = \CoordAng{3}, \cdots, Z_{(n)} = \CoordAng{n}, Z_{(n+1)} = \Rad$.
Note that $\Tanset = \lbrace Z_{(1)}, Z_{(2)},\cdots, Z_{(n)} \rbrace$.
We define the following vectorfield operators:
\begin{itemize}
	\item If $\vec{I} = (\iota_1, \iota_2, \cdots, \iota_N)$ is a multi-index
		of order $|\vec{I}| := N$
		with $\iota_1,\iota_2,\cdots,\iota_N \in \lbrace 1,2,\cdots,n+1 \rbrace$,
		then $\Fullset^{\vec{I}} := Z_{(\iota_1)} Z_{(\iota_2)} \cdots Z_{(\iota_N)}$ 
		denotes the corresponding $N^{th}$ order differential operator.
		We write $\Fullset^N$ rather than $\Fullset^{\vec{I}}$
		when we are not concerned with the structure of $\vec{I}$,
		and we sometimes omit the superscript when $N=1$.
	\item If $\vec{I} = (\iota_1,\iota_2,\cdots,\iota_N)$,
		then 
		$\vec{I}_1 + \vec{I}_2 = \vec{I}$ 
		means that
		$\vec{I}_1 = (\iota_{k_1},\iota_{k_2},\cdots,\iota_{k_m})$
		and
		$\vec{I}_2 = (\iota_{k_{m+1}},\iota_{k_{m+2}},\cdots,\iota_{k_N})$,
		where $1 \leq m \leq N$ and
		$k_1, k_2, \cdots, k_N$ is a permutation of 
		$1,2,\cdots,N$. 
	\item Sums such as $\vec{I}_1 + \vec{I}_2 + \cdots + \vec{I}_K = \vec{I}$
		have an analogous meaning.
	\item $\mathcal{P}_u$-tangent operators such as 
		$\Tanset^{\vec{I}}$ are defined analogously,  
		except in this case we have
		$\iota_1,\iota_2,\cdots,\iota_N \in \lbrace 1,2,\cdots,n \rbrace$.
		We write $\Tanset^N$ rather than $\Tanset^{\vec{I}}$
		when we are not concerned with the structure of $\vec{I}$,
		and we sometimes omit the superscript when $N=1$.
\end{itemize}
\end{definition}

\subsection{Notation involving multi-indices}
\label{SS:MULTIINDICES}
In defining our main $L^2$-controlling quantity
(see Def.\ \ref{D:MAINCOERCIVEQUANT}), 
we will refer to the following set of multi-indices.

\begin{definition}[\textbf{A set of $\Fullset$-multi-indices}]
\label{D:COMMUTATORMULTIINDICES}
We define
$\mathcal{I}_*^{[1,N];1}$
to be the set of $\Fullset$ multi-indices $\vec{I}$
(in the sense of Def.\ \ref{D:REPEATEDDIFFERENTIATIONSHORTHAND})
such that 
\textbf{i)} $1 \leq |\vec{I}| \leq N$,
\textbf{ii)} $\Fullset^{\vec{I}}$ contains \emph{at least one factor} belonging to 
$\Tanset = \lbrace \Lunit, \CoordAng{2},\CoordAng{3},\cdots,\CoordAng{n} \rbrace$,
and 
\textbf{iii)} $\Fullset^{\vec{I}}$ contains no more than $1$ factor of $\Rad$.
\end{definition}

\subsection{Norms}
\label{SS:NORMS}
In this subsection, we define the norms that we use in studying the solution.

\begin{definition}[\textbf{Pointwise norms}]
	\label{D:POINTWISENORM}
	We define the following pointwise norms for 
	arrays $\noshock = (\noshockuparg{J})_{1 \leq J \leq M}$
	and 
	$\diffnoshock = (\diffnoshockdoublearg{\alpha}{J})_{0 \leq \alpha \leq n, 1 \leq J \leq M}$:
	\begin{align} \label{E:POINTWISENORMOFARRAYS}
		|\noshock|
		& : = \sum_{J=1}^M
		|\noshockuparg{J}|,
		&
		|\diffnoshockdownarg{\alpha}|
		& : = \sum_{J=1}^M
		|\diffnoshockdoublearg{\alpha}{J}|,
		&
		|\diffnoshock|
		& : = \sum_{J=1}^M
					\sum_{\alpha=0}^n
		|\diffnoshockdoublearg{\alpha}{J}|.
	\end{align}
\end{definition}

We will use the following $L^2$ and $L^{\infty}$ norms in our analysis.

\begin{definition}[$L^2$ \textbf{and} $L^{\infty}$ \textbf{norms}]
In terms of the non-degenerate forms of Def.\ \ref{D:NONDEGENERATEVOLUMEFORMS},
we define the following norms for 
scalar or array-valued functions $w$:
\label{D:SOBOLEVNORMS}
	\begin{subequations}
	\begin{align}  \label{E:L2NORMS}
			\left\|
				w
			\right\|_{L^2(\mathcal{T}_{t,u})}^2
			& :=
			\int_{\mathcal{T}_{t,u}}
				|w|^2
			\, d \torusvol,
				\qquad
			\left\|
				w
			\right\|_{L^2(\Sigma_t^u)}^2
			:=
			\int_{\Sigma_t^u}
				|w|^2
			\, d \tvol,
				\\
			\left\|
				w
			\right\|_{L^2(\mathcal{P}_u^t)}^2
			& :=
			\int_{\mathcal{P}_u^t}
				|w|^2
			\, d \conevol,
			\notag
	\end{align}

	\begin{align} 
			\left\|
				w
			\right\|_{L^{\infty}(\mathcal{T}_{t,u})}
			& :=
				\mbox{ess sup}_{\vartheta \in \mathbb{T}^{n-1}}
				|w|(t,u,\vartheta),
			\qquad
			\left\|
				w
			\right\|_{L^{\infty}(\Sigma_t^u)}
			:=
			\mbox{ess sup}_{(u',\vartheta) \in [0,u] \times \mathbb{T}^{n-1}}
				|w|(t,u',\vartheta),
			\label{E:LINFTYNORMS}
				\\
			\left\|
				w
			\right\|_{L^{\infty}(\mathcal{P}_u^t)}
			& :=
			\mbox{ess sup}_{(t',\vartheta) \in [0,t] \times \mathbb{T}^{n-1}}
				|w|(t',u,\vartheta).
			\notag
	\end{align}
	\end{subequations}
\end{definition}

\subsection{Strings of commutation vectorfields and vectorfield seminorms}
\label{SS:STRINGSOFCOMMUTATIONVECTORFIELDS}
We will use
the following shorthand notation to capture the relevant structure
of our vectorfield differential operators and to schematically depict estimates.

\begin{remark}
	Some operators in Def.\ \ref{D:VECTORFIELDOPERATORS} are decorated with a $*$.
	These operators involve $\mathcal{P}_u$-tangent differentiations that often
	lead to a gain in smallness in the estimates.
	More precisely, the operators $\Tanset_*^N$ always lead to a gain in smallness while
	the operators $\Fullset_*^{N;1}$ lead to a gain in smallness except 
	perhaps when they are applied to $\upmu$
	(because $\Lunit \upmu$ and its $\Rad$ derivatives are not generally small).
\end{remark}

\begin{definition}[\textbf{Strings of commutation vectorfields and vectorfield seminorms}] \label{D:VECTORFIELDOPERATORS}
	\ \\
	\begin{itemize}
		\item $\Fullset^{N;1} f$ 
			denotes an arbitrary string of $N$ commutation
			vectorfields in $\Fullset$ (see \eqref{E:COMMUTATIONVECTORFIELDS})
			applied to $f$, where the string contains \emph{at most} $1$ factor of the $\mathcal{P}_u^t$-transversal
			vectorfield $\Rad$. 
			We sometimes write $Z f$ instead of $\Fullset^{1;1} f$.
		\item $\Tanset^N f$
			denotes an arbitrary string of $N$ commutation
			vectorfields in $\Tanset$ (see \eqref{E:TANGENTIALCOMMUTATIONVECTORFIELDS})
			applied to $f$. Consistent with Remark~\ref{R:MEANINGOFSINGLETAN}, 
			we sometimes write $\Singletan f$ instead of $\Tanset^1 f$.
		\item 
			For $N \geq 1$,
			$\Fullset_*^{N;1} f$
			denotes an arbitrary string of $N$ commutation
			vectorfields in $\Fullset$ 
			applied to $f$, where the string contains \emph{at least} one $\mathcal{P}_u$-tangent factor 
			and \emph{at most} $1$ factor of $\Rad$.
			We also set  $\Fullset_*^{0;0} f := f$.
		\item For $N \geq 1$,
					$\Tanset_*^N f$ 
					denotes an arbitrary string of $N$ commutation
					vectorfields in $\Tanset$ 
					applied to $f$, where the string contains
					\emph{at least one factor} belonging to the geometric torus coordinate partial derivative vectorfield set
					$\lbrace 
						\CoordAng{2},
						\CoordAng{3},
						\cdots,
						\CoordAng{n}
					\rbrace$
					or \emph{at least two factors} of $\Lunit$.
\end{itemize}

	\begin{remark}[\textbf{Another way to think about operators} $\Tanset_*^N$]
		\label{R:ANOTHERWAYTOTHINKABOUTPSTARDERIVATIVES}
		For exact simple plane wave solutions,
		if $N \geq 1$ and $f$ is \emph{any} of the quantities that we must estimate,
		then we have $\Tanset_*^N f \equiv 0$. 
	\end{remark}
	
	We also define seminorms constructed out of sums of the above strings of vectorfields:
	\begin{itemize}
		\item $|\Fullset^{N;1} f|$ 
		simply denotes the magnitude of one of the $\Fullset^{N;1} f$ as defined above
		(there is no summation).
	\item $|\Fullset^{\leq N;1} f|$ is the \emph{sum} over all terms of the form $|\Fullset^{N';1} f|$
			with $N' \leq N$ and $\Fullset^{N';1} f$ as defined above.
			We sometimes write $|\Fullset^{\leq 1} f|$ instead of $|\Fullset^{\leq 1;1} f|$.
		\item $|\Fullset^{[1,N];1} f|$ is the sum over all terms of the form $|\Fullset^{N';1} f|$
			with $1 \leq N' \leq N$ and $\Fullset^{N';1} f$ as defined above.
		\item Sums such as 
			$|\Tanset^{\leq N} f|$,
			$|\Tanset_*^{[1,N]} f|$,
			etc.\
			are defined analogously.
		\item Seminorms such as
			$\| \Fullset_*^{[1,N];1} f \|_{L^{\infty}(\Sigma_t^u)}$
			and
			$\| \Tanset_*^{[1,N]} f \|_{L^{\infty}(\Sigma_t^u)}$
			(see Def.\ \ref{D:SOBOLEVNORMS})
			are defined analogously.
	\end{itemize}

\end{definition}

\section{Energy identities}
\label{S:ENERGYID}
In this section, we define the building block energies and characteristic fluxes that we use to control the solution in $L^2$
and derive their basic coerciveness properties.
We then derive energy identities involving the building block energies and characteristic fluxes.
Later in the article, in Def.\ \ref{D:MAINCOERCIVEQUANT}, we will use the building blocks to define the main $L^2$-controlling quantity.
  
\subsection{Energies and characteristic flux definitions}
\label{SS:ENERGYFLUXDEFS}

\begin{definition}[\textbf{Energies and characteristics fluxes}]
\label{D:ENERGYFLUX}
In terms of the geometric forms of Def.\ \ref{D:NONDEGENERATEVOLUMEFORMS},
we define the energy $\shocken[\cdot]$, which is a functional
of scalar-valued functions $f$, as follows:
\begin{align} \label{E:SHOCKENERGY}
	\shocken[f](t,u)
		& := 
		\int_{\Sigma_t^u} 
			f^2
		\, d \tvol.
\end{align}

In terms of the Cartesian forms of Def.\ \ref{D:CARTESIANFORMS}
and the Euclidean-unit one-form $\covL_{\alpha}$
defined in \eqref{E:EUCLIDEANNORMALTONULLHYPERSURFACE},
we define the energy $\noshocken[\cdot]$ and characteristic flux 
$\noshockfl[\cdot]$, which are functionals of $\mathbb{R}^M$-valued functions
$w$, as follows:
\begin{subequations}
\begin{align} \label{E:NOSHOCKENERGY}
	\noshocken[w](t,u)
		& := 
		\int_{\Sigma_t^u} 
			\delta_{JK} A_I^{0;J}(\Psi,\noshock) w^I w^K
		\, d \Sigma,
		\\
	\noshockfl[w](t,u)
	 & := 
		\int_{\mathcal{P}_u^t} 
			\delta_{JK} A_I^{\alpha;J}(\Psi,\noshock) \covL_{\alpha} w^I w^K 
		\, d \mathcal{P},
		\label{E:NOSHOCKNULLFLUX}
\end{align}
\end{subequations}
where $\delta_{JK}$ is a standard Kronecker delta.

\end{definition}

\begin{lemma}[\textbf{Coerciveness of the energies and null fluxes for the symmetric hyperbolic variables}]
	\label{L:COERCIVENESSOFNOSHOCKENERGY}
	If $|\GdVar|$ is sufficiently small, then the energy and the characteristic flux 
	from Def.\ \ref{D:ENERGYFLUX}
	enjoy the following coerciveness:
	\begin{subequations}
	\begin{align} \label{E:COERCIVENESSOFNOSHOCKENERGY}
		\noshocken[w](t,u)
		& \approx 
		\int_{\Sigma_t^u} 
			\upmu \delta_{JK} w^J w^K
		\, d \tvol,
			\\
		\noshockfl[w](t,u)
		& \approx 
		\int_{\mathcal{P}_u^t} 
			\delta_{JK} w^J w^K
		\, d \conevol,
		\label{E:COERCIVENESSOFNOSHOCKNULLFLUX}
	\end{align}
	\end{subequations}
	where $\delta_{JK}$ is a standard Kronecker delta.
\end{lemma}

\begin{proof}
	From the arguments given in the proof of Lemma~\ref{L:VOLFORMRELATION},
	it follows that the one-form $\covL_{\alpha}$ defined in \eqref{E:EUCLIDEANNORMALTONULLHYPERSURFACE}
	can be decomposed as $\covL_{\alpha} = \delta_{\alpha}^0 - \delta_{\alpha}^1 + \covL_{\alpha}^{(Small)}$,
	where $\covL_{\alpha}^{(Small)} = \GdVar \smoothfunction(\GdVar)$.
	Hence, from \eqref{E:POSITIVEDEFMATRICES}, it follows that
	when $|\GdVar|$ is sufficiently small, we have
	$\delta_{JK} A_I^{0;J} w^I w^K \approx \delta_{JK} w^J w^K$
	and $\delta_{JK} A_I^{\alpha;J} \covL_{\alpha} w^I w^K \approx \delta_{JK} w^J w^K$.
	Appealing to definitions \eqref{E:NOSHOCKENERGY}-\eqref{E:NOSHOCKNULLFLUX}
	and using the integration measure relationships stated in \eqref{E:VOLFORMRELATION},
	we conclude \eqref{E:COERCIVENESSOFNOSHOCKENERGY}-\eqref{E:COERCIVENESSOFNOSHOCKNULLFLUX}.
\end{proof}

\subsection{Energy-characteristic flux identities}
\label{SS:ENERGYIDENTITIES}
The integral identities in the following
proposition form the starting point for our $L^2$ analysis of solutions.
A crucial point is that LHS~\eqref{E:NONSHOCKVARIABLEBASICENERGYID}
features the characteristic flux $\noshockfl[\cdot](t,u)$,
which by \eqref{E:COERCIVENESSOFNOSHOCKNULLFLUX}
can be used to control $\noshock$ and $\diffnoshock$
on the characteristic hypersurfaces $\mathcal{P}_u^t$
\emph{without any degenerate $\upmu$ weight}.

\begin{proposition}[\textbf{Energy-characteristic flux identities}]
	\label{P:ENERGYIDENTITIES}
	Let $\Lunit = \Lunit^{\alpha}(\Psi,\noshock) \partial_{\alpha}$
	be the vectorfield from equation \eqref{E:SHOCKEQN} and 
	let $f$ be a solution to the inhomogeneous transport equation
	\begin{align} \label{E:ENERGYIDCOMMUTEDMAINSHOCKEVOLUTION} 
		\Lunit f
		& = \mathfrak{F}.
	\end{align}
	Then the following integral identity holds for the energy defined in \eqref{E:SHOCKENERGY}:
	\begin{align} \label{E:SHOCKVARIABLEBASICENERGYID}
			\shocken[f](t,u)
			& = 
			\shocken[f](0,u)
			+
			2
			\int_{\mathcal{M}_{t,u}}
				f \mathfrak{F}
			\, d \vol.
	\end{align}
	
	Moreover, let
	$A_J^{\alpha;I}(\Psi,\noshock)$ be the components of the symmetric matrices from equation \eqref{E:NONSHOCKEQUATION} 
	and let $w$ be a solution to the (linear-in-$w$) inhomogeneous symmetric hyperbolic system
	\begin{align} \label{E:ENERGYIDCOMMUTEDNONSHOCKEVOLUTION} 
		\upmu A_J^{\alpha;I} \partial_{\alpha} w^J
		& = \mathfrak{F}^I.
	\end{align}
	Then there exist smooth functions,
	schematically denoted by $\smoothfunction$,
	such that the following integral identity holds
	for the energy and characteristic flux defined in \eqref{E:NOSHOCKENERGY}-\eqref{E:NOSHOCKNULLFLUX}:
	\begin{align} \label{E:NONSHOCKVARIABLEBASICENERGYID}
			\noshocken[w](t,u)
			+
			\noshockfl[w](t,u)
			& = 
			\noshocken[w](0,u)
			+
			\noshockfl[w](t,0)
				\\
		& \ \
			+
			2
			\int_{\mathcal{M}_{t,u}}
				\left\lbrace
					1 + \GdVar \smoothfunction(\GdVar)
				\right\rbrace
				\delta_{JK} \mathfrak{F}^J w^K
			\, d \vol
				\notag \\
		& \ \
			+
			\int_{\mathcal{M}_{t,u}}
				\smoothfunction_{JK}(\BadVar,\Rad \Psi, \Singletan \Psi) w^J w^k
			\, d \vol,
			\notag
	\end{align}
	where $\delta_{JK}$ is the standard Kronecker delta.
	
\end{proposition}

\begin{proof}
The identity \eqref{E:SHOCKVARIABLEBASICENERGYID}
is a simple consequence of equation \eqref{E:ENERGYIDCOMMUTEDMAINSHOCKEVOLUTION} 
since 
$
\displaystyle
\Lunit  
= \frac{\partial}{\partial t}
$
relative to the geometric coordinates $(t,u,\vartheta)$.

To prove \eqref{E:NONSHOCKVARIABLEBASICENERGYID},
we define the following vectorfield relative to the Cartesian coordinates:
$\mathscr{J}^{\alpha} := \delta_{JK} A_I^{\alpha;J} w^I w^K$.
Using equation \eqref{E:ENERGYIDCOMMUTEDNONSHOCKEVOLUTION}
and the symmetry assumption
$ A_J^{\alpha;I} = A_I^{\alpha;J}$,
we derive (relative to the Cartesian coordinates) the following divergence identity:
$
\upmu \partial_{\alpha} \mathscr{J}^{\alpha}
=  
2 \delta_{JK} \mathfrak{F}^J w^K
+
\delta_{JK} (\upmu \partial_{\alpha} A_I^{\alpha;J}) w^I w^K
$.
We now apply the divergence theorem
to the vectorfield $\mathscr{J}$ on the region $\mathcal{M}_{t,u}$,
where we use the Cartesian coordinates,
the Euclidean metric $\delta^{\alpha \beta} := \mbox{\upshape diag} (1,1,\cdots,1)$
on $\mathbb{R} \times \Sigma$, 
and the Cartesian forms of Def.\ \ref{D:CARTESIANFORMS}
in all computations.
Also using that the future-directed Euclidean co-normal to $\Sigma_t$ has Cartesian components $\delta_{\alpha}^0$
and that the future-directed Euclidean co-normal to $\mathcal{P}_u^t$ 
has Cartesian components $\covL_{\alpha}$ (see Def.\ \ref{D:EUCLIDEANNORMALTONULLHYPERSURFACE}),
we deduce
\begin{align}	\label{E:SYMMETRICHYPERBOLICENERYGIDENTIYINDIVERGENCEFORM}
	&
	\int_{\Sigma_t^u}
		\delta_{JK} A_I^{0;J} w^I w^K
	\, d \Sigma
	+
	\int_{\mathcal{P}_u^t}
		\delta_{JK} A_I^{\alpha;J} \covL_{\alpha} w^I w^K
	\, d \mathcal{P}
		\\
	& = \int_{\Sigma_0^u}
		\delta_{JK} A_I^{0;J} w^I w^K
	\, d \Sigma
	+
	\int_{\mathcal{P}_0^t}
		\delta_{JK} A_I^{\alpha;J} \covL_{\alpha} w^I w^K
	\, d \mathcal{P}
		\notag \\
	& \ \
		+
		\int_{\mathcal{M}_{t,u}}
			\left\lbrace
				2 \delta_{JK} \mathfrak{F}^J w^K
				+
				\delta_{JK} (\upmu \partial_{\alpha} A_I^{\alpha;J}) w^I w^K
			\right\rbrace
		\, \frac{d \mathcal{M}}{\upmu}.
		\notag
\end{align}
Next, using Lemma~\ref{L:CARTESIANDERIVATIVESINTERMSOFGEOMETRICDERIVATIVES} and definition \eqref{E:BADARRAY},
we can express
the integrand $\delta_{JK} (\upmu \partial_{\alpha} A_I^{\alpha;J}) w^I w^K$
on RHS~\eqref{E:SYMMETRICHYPERBOLICENERYGIDENTIYINDIVERGENCEFORM} in the following schematic form:
$\smoothfunction_{JK}(\BadVar,\Rad \Psi, \Singletan \Psi) w^J w^k$.
Also using Lemma~\ref{L:VOLFORMRELATION}
to express the integration measure
$
\displaystyle
\frac{d \mathcal{M}}{\upmu}
$
on RHS~\eqref{E:SYMMETRICHYPERBOLICENERYGIDENTIYINDIVERGENCEFORM}
as
$
\left\lbrace
	1 + \GdVar \smoothfunction(\GdVar)
\right\rbrace
\, d \vol
$
and appealing to definitions \eqref{E:NOSHOCKENERGY}-\eqref{E:NOSHOCKNULLFLUX},
we arrive at the desired identity \eqref{E:NONSHOCKVARIABLEBASICENERGYID}.

\end{proof}

\section{The number of derivatives, data-size assumptions, bootstrap assumptions, smallness assumptions, and running assumptions}
\label{S:NUMDERIVSDATASIZEANDBOOTSTRAPASSUMPTIONS}
In this section, we state the number of derivatives that we use to close the estimates,
state our assumptions on the size of the data,
formulate bootstrap assumptions that we use to derive estimates,
and describe our smallness assumptions.
In Subsect.\ \ref{SS:EXISTENCEOFDATA}, we explain why 
there exist data that verify the assumptions.

\subsection{The number of derivatives}
\label{SS:NUMBEROFDERIVATIVES}
Throughout the rest of the paper, 
$\Ntop$ and $\Nmid$ denote two fixed positive integers
verifying the following relations, where $n$ is the number of spatial dimensions:
\begin{align} \label{E:NDEFS}
	\Ntop & \geq n + 5,
	&
	\Nmid
	& := 
		\left\lceil \frac{\Ntop}{2} \right\rceil + 1.
\end{align}
The solutions that we will study are such that, roughly, 
the order $\leq \Ntop$ derivatives of $\Psi$ (with respect to suitable strings of geometric vectorfields)
are uniformly bounded in the norm $\| \cdot \|_{L^2(\Sigma_t^u)}$ 
and the order $\leq \Nmid$ derivatives of $\Psi$ are uniformly bounded in the norm $\| \cdot \|_{L^{\infty}(\Sigma_t^u)}$.
The remaining quantities that we must estimate verify similar bounds
but, in some cases, they are one degree less differentiable. 
The definitions in \eqref{E:NDEFS} are convenient 
in the sense that they will lead to the following:
when we derive $L^2$ estimates for error term products in the commuted equations, 
all factors in the product except at most one will be 
uniformly bounded in the norm $\| \cdot \|_{L^{\infty}(\Sigma_t^u)}$.

\subsection{Data-size assumptions}
\label{SS:DATASIZE}
In this subsection, we state our assumptions on the size of the data.

\subsubsection{The data-size parameter that controls the time of shock formation}
We start with the definition of a data-size parameter 
$\TranminusdatasizeWithFactor$, which is tied to the time of first shock formation.
More precisely, our main theorem shows that $\max_{\alpha=0,\cdots,n} |\partial_{\alpha} \Psi|$ 
blows up at a time approximately equal to $\TranminusdatasizeWithFactor^{-1}$.

\begin{definition}[\textbf{The crucial quantity that controls the time of shock formation}]
	\label{D:SHOCKFORMATIONQUANTITY}
	We define $\TranminusdatasizeWithFactor$ as
	\begin{align} \label{E:SHOCKFORMATIONQUANTITY}
		\TranminusdatasizeWithFactor
		:= \sup_{\Sigma_0^1} [\blowupcoeff \Rad \Psi]_-,
	\end{align}
	where $\blowupcoeff \neq 0$ 
	(see Remark~\ref{R:BLOWUPCOEFFICIENTISNONZERO})
	is the blowup-coefficient from Def.\ \ref{D:GENIUNELYNONLINEARCONSTANT}
	and $[f]_- := |\min \lbrace f,0 \rbrace|$.
\end{definition}

\begin{remark}[\textbf{Functional dependence of } $\blowupcoeff$ \textbf{along} $\Sigma_0$]
	\label{R:FUNCTIONALDEPENDENCEOFBLOWUPCOEFFICIENT}
	Note that by \eqref{E:XIJIC}, along $\Sigma_0$,
	$\blowupcoeff$ can be viewed as a function of $\Psi|_{\Sigma_0}$ and $\noshock|_{\Sigma_0}$.
\end{remark}

\subsubsection{Data-size assumptions}
\label{SSS:DATA}
For technical convenience, we assume that the solution 
is $C^{\infty}$ with respect to the Cartesian coordinates
along the ``data hypersurfaces'' $\Sigma_0^{U_0}$ and 
$\mathcal{P}_u^{2 \TranminusdatasizeWithFactor^{-1}}$.
However, to close our estimates, we only need to make assumptions
on various Sobolev and Lebesgue norms of the data, 
where the norms are in defined terms of the geometric coordinates and commutation vectorfields. 
In this subsubsection, we state the norm assumptions,
which involve additional three parameters, denoted by 
$\Psiep$,
$\mathring{\upepsilon}$, 
and $\Trandatasize$.
We note that $\Trandatasize$ does not need to be small,
and that the same is true for the parameter 
$\TranminusdatasizeWithFactor$ from Def.\ \ref{D:SHOCKFORMATIONQUANTITY}.
We will describe our 
smallness assumptions on $\Psiep$ 
and $\mathring{\upepsilon}$ in Subsect.\ \ref{SS:SMALLNESSASSUMPTIONS}.

We assume that the data verify the following size assumptions
(see Subsect.\ \ref{SS:STRINGSOFCOMMUTATIONVECTORFIELDS} regarding the vectorfield differential operator notation).

\medskip
\noindent \underline{$L^2$ \textbf{Assumptions along} $\Sigma_0^1$}.
\begin{align} 
	\left\|
		\Fullset_*^{[1,\Ntop];1} \Psi
	\right\|_{L^2(\Sigma_0^1)},
		\,
	\left\|
		\Fullset^{\leq \Ntop-1;1} \noshock
	\right\|_{L^2(\Sigma_0^1)},
		\,
	\left\|
		\Fullset^{\leq \Ntop-1;1} \diffnoshock
	\right\|_{L^2(\Sigma_0^1)}
	\leq \mathring{\upepsilon}.
	\label{E:L2SMALLDATAALONGSIGMA0}
\end{align}

\medskip

\noindent \underline{$L^{\infty}$ \textbf{Assumptions along} $\Sigma_0^1$}.
\begin{subequations}
\begin{align} 
	\left\|
		\Psi
	\right\|_{L^{\infty}(\Sigma_0^1)}
	& \leq \Psiep,
	\label{E:PSIITSELFLINFTYSMALLDATAALONGSIGMA0}
		\\
	\left\|
		\Fullset_*^{[1,\Nmid];1} \Psi
	\right\|_{L^{\infty}(\Sigma_0^1)},
		\,
	\left\|
		\Fullset^{\leq \Nmid-1;1} \noshock
	\right\|_{L^{\infty}(\Sigma_0^1)},
		\,
	\left\|
		\Fullset^{\leq \Nmid-1;1} \diffnoshock
	\right\|_{L^{\infty}(\Sigma_0^1)}
	& \leq \mathring{\upepsilon},
	\label{E:LINFTYSMALLDATAALONGSIGMA0}
		\\
	\left\|
		\Rad \Psi
	\right\|_{L^{\infty}(\Sigma_0^1)}
	& \leq \Trandatasize.
		\label{E:LINFTYLARGEDATAALONGSIGMA0} 
\end{align}
\end{subequations}

\medskip

\noindent \underline{\textbf{Assumptions along} $\mathcal{P}_0^{2 \TranminusdatasizeWithFactor^{-1}}$}.
\begin{align} 
	\left\|
		\Fullset^{\leq \Ntop-1;1} \noshock
	\right\|_{L^2(\mathcal{P}_0^{2 \TranminusdatasizeWithFactor^{-1}})},
		\,
	\left\|
		\Fullset^{\leq \Ntop-1;1} \diffnoshock
	\right\|_{L^2(\mathcal{P}_0^{2 \TranminusdatasizeWithFactor^{-1}})}
	\leq \mathring{\upepsilon}.
	\label{E:L2SMALLDATAALONGP0}
\end{align}

\medskip

\noindent \underline{\textbf{Assumptions along} $\mathcal{T}_{0,u}$}.
We assume that for $u \in [0,1]$, we have
\begin{align} 
	\left\|
		\Tanset^{\leq \Ntop-2} \noshock
	\right\|_{L^2(\mathcal{T}_{0,u})},
		\,
	\left\|
		\Tanset^{\leq \Ntop-2} \diffnoshock
	\right\|_{L^2(\mathcal{T}_{0,u})}
	\leq \mathring{\upepsilon}.
	\label{E:L2SMALLDATAALONGELL0U}
\end{align}

\begin{remark}
	Roughly, we will study solutions that are perturbations of
	non-trivial solutions with $\mathring{\upepsilon} = 0$.
	Note that $\mathring{\upepsilon} = 0$ corresponds to
	a simple plane symmetric wave, as we described in Subsect.\ \ref{SS:SIMPLEWAVS}.
	Note also that $\Psiep$, $\TranminusdatasizeWithFactor$, and $\Trandatasize$
	are generally non-zero for simple plane symmetric waves.
\end{remark}

\subsubsection{Estimates for the initial data of the remaining geometric quantities}
\label{SSS:DATAOFOTHERDATASIZE}
To close our proof, we will have to estimate the scalar functions
$\upmu$,
$\upxi_j^{(Small)}$, 
$\CoordAngSmallcomp{i}{j}$,
and
$\Xi^j$
featured in the array \eqref{E:BADARRAY} and definition \eqref{E:RADDECOMPINTOPARTIALUPLUSANGULAR}.
In this subsubsection, as a preliminary step, 
we estimate the size of their data along $\Sigma_0^1$.

\begin{lemma}[\textbf{Estimates for the data of} 
$\upmu$,
$\upxi_j^{(Small)}$, 
$\CoordAngSmallcomp{i}{j}$,
\textbf{and}
$\Xi^j$
]
\label{L:ESTIMATESFORNONLINEARGEOMETRICOPTICS}
Under the data-size assumptions of Subsubsect.\ \ref{SSS:DATA},
there exists a constant $C > 0$ depending on the parameter $\Trandatasize$ from \eqref{E:LINFTYLARGEDATAALONGSIGMA0}
and a constant $C_{\star} > 0$ that \textbf{does not depend on} $\Trandatasize$
such that the following estimates hold for the scalar functions
$\upmu$,
$\upxi_j^{(Small)}$, 
$\CoordAngSmallcomp{i}{j}$
and
$\Xi^j$
defined in 
Defs.~\ref{D:MUDEF} and \ref{D:PERTURBEDPART}
and equation \eqref{E:RADDECOMPINTOPARTIALUPLUSANGULAR},
whenever $\Psiep$ and $\mathring{\upepsilon}$ are sufficiently small
(see Subsect.\ \ref{SS:STRINGSOFCOMMUTATIONVECTORFIELDS} regarding the vectorfield notation):
\begin{subequations}
\begin{align}
	\left\|
		\Tanset_*^{[1,\Ntop-1]} \upmu
	\right\|_{L^2(\Sigma_0^1)}
	& \leq C \mathring{\upepsilon}
		\label{E:MUDATATANGENTIALDERIVATIVESL2SMALL},
			\\
	\left\|
		\upmu - 1
	\right\|_{L^{\infty}(\Sigma_0^1)}
	& \leq C_{\star} (\Psiep + \mathring{\upepsilon}),
		\label{E:MUMINUSONEDATALINFINITYSMALL}
			\\
	\left\|
		\Lunit \upmu
	\right\|_{L^{\infty}(\Sigma_0^1)}
	& \leq C,
		\label{E:LUNITMUDATALINFINITYLARGE}
			\\
	\left\|
		\Tanset_*^{[1,\Nmid-1]} \upmu
	\right\|_{L^{\infty}(\Sigma_0^1)}
	& \leq C \mathring{\upepsilon},
		\label{E:MUDATATANGENTIALLINFINITYDERIVATIVESSMALL}
\end{align}
\end{subequations}

\begin{subequations}
\begin{align}
	\left\|
		\upxi_j^{(Small)}
	\right\|_{L^{\infty}(\Sigma_0^1)}
	& \leq C_{\star} (\Psiep + \mathring{\upepsilon}) \delta_j^1,
		\label{E:XISMALLJITSELFDATATALLDERIVATIVESLINFINITYSMALL}
			\\
	\left\|
		\Fullset_*^{[1,\Ntop-1];1} \upxi_j^{(Small)}
	\right\|_{L^2(\Sigma_0^1)}
	& \leq C \mathring{\upepsilon},
		\label{E:XISMALLJDATATALLDERIVATIVESL2SMALL}
			\\
\left\|
		\Fullset_*^{[1,\Nmid-1];1} \upxi_j^{(Small)}
	\right\|_{L^{\infty}(\Sigma_0^1)}
	& \leq C \mathring{\upepsilon},
		\label{E:XISMALLJDATATALLDERIVATIVESLINFINITYSMALL}
			\\
	\left\|
		\Rad \upxi_j^{(Small)}
	\right\|_{L^{\infty}(\Sigma_0^1)}
	& \leq C,
		\label{E:XISMALLJDATATARADLINFINITYLARGE}
\end{align}
\end{subequations}

\begin{subequations}
\begin{align}
	\left\|
		\Fullset^{\leq \Ntop-1;1} \CoordAngSmallcomp{i}{j}
	\right\|_{L^2(\Sigma_0^1)}
	& \leq C \mathring{\upepsilon},
		\label{E:COORDANGJSMALLDATAALLDERIVATIVEL2SMALL}
			\\
	\left\|
		\Fullset^{\leq \Nmid-1;1} \CoordAngSmallcomp{i}{j}
	\right\|_{L^{\infty}(\Sigma_0^1)}
	& \leq C \mathring{\upepsilon},
		\label{E:COORDANGJSMALLDATAALLDERIVATIVESLINFINITYSMALL}
\end{align}
\end{subequations}

\begin{subequations}
\begin{align}
	\left\|
		\Tanset^{\leq \Ntop-1} \Xi^j
	\right\|_{L^2(\Sigma_0^1)}
	& \leq C,
		\label{E:BIGXIJDATAALLDERIVATIVESL2}
			\\
	\left\|
		\Tanset^{\leq \Nmid-1} \Xi^j
	\right\|_{L^{\infty}(\Sigma_0^1)}
	& \leq C.
			\label{E:BIGXIJDATAALLDERIVATIVESLINFINITY}
\end{align}
\end{subequations}

\end{lemma}

\begin{remark}[\textbf{The ``non-small'' quantities}]
Note that the only estimates not featuring the smallness parameters
$\Psiep$ 
or
$\mathring{\upepsilon}$
are 
\eqref{E:LUNITMUDATALINFINITYLARGE}, 
\eqref{E:XISMALLJDATATARADLINFINITYLARGE},
\eqref{E:BIGXIJDATAALLDERIVATIVESL2},
and 
\eqref{E:BIGXIJDATAALLDERIVATIVESLINFINITY}.

\end{remark}

\begin{proof}[Proof sketch]
We only sketch the proof since it is standard but has a tedious component 
that is similar to other analysis that we carry out later: 
commutator estimates of the type proved in Lemma~\ref{L:COMMUTATORESTIMATES},
based on the vectorfield commutator identities \eqref{E:LUNITANGULARCOMMUTATOR}-\eqref{E:RADCOORDANGCOMMUTATOR}.

To proceed, we use
Lemmas~\ref{L:EVOLUTIONFORUPMUANDUPXI}
and
\ref{L:COORDANGCOMPONENTRECTANGULAR},
Cor.\ \ref{C:BIGXIJEVOLUTIONEQUATION}, 
and the fact that $\Lunit^{\alpha}$ and $\Radunit^{\alpha}$ are
smooth functions of $(\Psi,\noshock)$
(the latter by \eqref{E:RADUNITJ})
to deduce the following schematic relationships,
which hold along $\Sigma_0$ (with $\smoothfunction$ smooth and with $\Singletan \in \Tanset$):
\begin{align}
	(\upmu -1)|_{\Sigma_0}
	& = (\Psi,\noshock) \cdot \smoothfunction(\Psi,\noshock),
		\label{E:MUFUNCTIONALFORMALONGSIGMA0} 
		\\
	\upxi_j^{(Small)}|_{\Sigma_0}
	& = (\Psi,\noshock) \cdot \smoothfunction(\Psi,\noshock) \delta_j^1,
		\label{E:XIJSMALLFUNCTIONALFORMALONGSIGMA0} 
		\\
	\CoordAngSmallcomp{i}{j}|_{\Sigma_0}
	& = 0,
		\label{E:CORDANGCOMPSAREINITIALLYZERO} \\
	\Xi^j|_{\Sigma_0}
	& = \smoothfunction(\Psi,\noshock),
		\label{E:BIGXIJFUNCTIONALFORMALONGSIGMA0}
\end{align}
as well as the following evolution equations,
also written in schematic form:
\begin{align}
	\Lunit \upmu
	& = 
	\smoothfunction(\GdVar) \Rad \Psi 
	+ 
	\upmu \smoothfunction(\GdVar) \Lunit \Psi
	+
	\upmu \smoothfunction(\GdVar) \diffnoshock,
		\label{E:LUNITMUDATASCHEMATIC} \\
	\Lunit \upxi_j^{(Small)}
	& = 
	\smoothfunction(\GdVar) \Singletan \Psi 
	+
	\smoothfunction(\GdVar) \diffnoshock,
		\label{E:LUNITXIJSMALLDATASCHEMATIC} \\
	\Lunit \CoordAngSmallcomp{i}{j}
	& = 
	\smoothfunction(\GdVar) \Singletan \Psi 
	+
	\smoothfunction(\GdVar) \diffnoshock,
		\label{E:LUNITCOORDANGJSMALLDATASCHEMATIC} \\
	\Lunit \Xi^j 
	& = 	(\Xi^1,\cdots,\Xi^n)
				\cdot
				\smoothfunction(\GdVar,\Singletan \Psi)
				+ 
				\smoothfunction(\BadVar,\Lunit \Psi,\Rad \Psi).
	\label{E:LUNITBIGXIJSCHEMATIC}
\end{align}
By repeatedly differentiating 
\eqref{E:LUNITMUDATASCHEMATIC}-\eqref{E:LUNITBIGXIJSCHEMATIC}
with the elements of $\Fullset$
and using the commutator identities \eqref{E:LUNITANGULARCOMMUTATOR}-\eqref{E:RADCOORDANGCOMMUTATOR},
we can algebraically express 
all quantities that we need estimate in terms 
of the derivatives of
$\upmu$,
$\upxi_j^{(Small)}$, 
$\CoordAngSmallcomp{i}{j}$,
and
$\Xi^j$
with respect to the ($\Sigma_t$-tangent) 
vectorfields in $\lbrace \Rad,\CoordAng{2},\cdots,\CoordAng{n} \rbrace$
and the $\Fullset$ derivatives of 
$\Psi$,
$\noshock$,
and 
$\diffnoshock$.
Then using
\eqref{E:MUFUNCTIONALFORMALONGSIGMA0}-\eqref{E:BIGXIJFUNCTIONALFORMALONGSIGMA0},
we can can express, along $\Sigma_0$, 
the derivatives of
$\upmu$,
$\upxi_j^{(Small)}$, 
$\CoordAngSmallcomp{i}{j}$
and
$\Xi^j$
with respect to the elements of $\lbrace \Rad,\CoordAng{2},\cdots,\CoordAng{n} \rbrace$
in terms of
the derivatives of
$\Psi$ and $\noshock$
with respect to the elements of $\lbrace \Rad,\CoordAng{2},\cdots,\CoordAng{n} \rbrace$.
The estimates \eqref{E:MUDATATANGENTIALDERIVATIVESL2SMALL}-\eqref{E:BIGXIJDATAALLDERIVATIVESLINFINITY}
then follow from these algebraic expressions,
the data-size assumptions \eqref{E:L2SMALLDATAALONGSIGMA0}-\eqref{E:LINFTYLARGEDATAALONGSIGMA0},
and the standard Sobolev calculus.
We stress that the identities \eqref{E:LUNITANGULARCOMMUTATOR}-\eqref{E:RADCOORDANGCOMMUTATOR}
show that commutator terms contain a factor involving a differentiation with respect to
one of the $\CoordAng{i}$, which, in view of our data-size assumptions from
Subsubsect.\ \ref{SSS:DATA}, 
leads to a gain in $\mathcal{O}(\mathring{\upepsilon})$ 
smallness for all commutator terms.

\end{proof}

\subsection{Bootstrap assumptions}
\label{SS:BOOTSTRAP}
In this subsection, we state the bootstrap assumptions that we use to control the solution.

\subsubsection{\texorpdfstring{$\Tboot$}{The bootstrap time}, the positivity of 
\texorpdfstring{$\upmu$}{the inverse foliation density}, and the diffeomorphism property of 
\texorpdfstring{$\Upsilon$}{the change of variables map}}
\label{SS:SIZEOFTBOOT}
We now state some basic bootstrap assumptions.
We start by fixing a real number $\Tboot$ with
\begin{align} \label{E:TBOOTBOUNDS}
	0 < \Tboot \leq 2 \TranminusdatasizeWithFactor^{-1},
\end{align}
where $\TranminusdatasizeWithFactor$ is the data-dependent parameter
from Def.\ \ref{D:SHOCKFORMATIONQUANTITY}.

We assume that on the spacetime domain $\mathcal{M}_{\Tboot,U_0}$
(see \eqref{E:MTUDEF}), we have
\begin{align} \label{E:BOOTSTRAPMUPOSITIVITY} \tag{$\mathbf{BA} \upmu > 0$}
	\upmu > 0.
\end{align}
Inequality \eqref{E:BOOTSTRAPMUPOSITIVITY} essentially means that no shocks are present in
$\mathcal{M}_{\Tboot,U_0}$.

We also assume that
\begin{align} \label{E:BOOTSTRAPCHOVISDIFFEO}
	& \mbox{The change of variables map $\Upsilon$ from Def.\ \ref{D:CHOV}
	is a diffeomorphism from} \\
	& [0,\Tboot) \times [0,U_0] \times \mathbb{T}^{n-1}
	\mbox{ onto its image}.
	\notag
\end{align}

\subsubsection{Fundamental $L^{\infty}$ bootstrap assumptions}
\label{SSS:FUNDAMENTALBOOT}
In this section, we state our fundamental $L^{\infty}$ bootstrap assumptions.
We will derive strict improvements of the fundamental
bootstrap assumptions in Cor.\ \ref{C:IMPROVEMENTOFFUNDAMANETALBOOTSTRAPASSUMPTIONS},
on the basis of a priori energy estimates and Sobolev embedding.

\medskip

\noindent \underline{\textbf{Fundamental bootstrap assumptions for } $\noshock$ \textbf{and} $\diffnoshock$}.
We assume that the following inequalities hold
for $(t,u) \in [0,\Tboot) \times [0,U_0]$,
$(\alpha = 0,\cdots,n)$,
$(J=1,\cdots,M)$:
\begin{align} \label{E:FUNDAMENTALBOOTSTRAP}
	\left\|
		\Tanset^{\leq \Nmid-1} \noshockuparg{J}
	\right\|_{L^{\infty}(\Sigma_t^u)},
		\,
	\left\|
		\Tanset^{\leq \Nmid-1} \diffnoshockdoublearg{\alpha}{J}
	\right\|_{L^{\infty}(\Sigma_t^u)}
	& \leq \varepsilon,
\end{align}
where $\varepsilon > 0$ is a small bootstrap parameter
(see Subsect.\ \ref{SS:SMALLNESSASSUMPTIONS} for discussion on the required smallness).

\subsubsection{Auxiliary bootstrap assumptions}
\label{SSS:AUXBOOT}
In addition to the fundamental bootstrap assumptions, 
we find it convenient to make auxiliary bootstrap assumptions,
which we state in this subsubsection.
We will derive strict improvements of the auxiliary bootstrap assumptions
in Prop.\ \ref{P:POINTWISEANDIMPROVEMENTOFAUX}.

\noindent \underline{\textbf{Auxiliary bootstrap assumptions for } $\Psi$}.
We assume that the following inequalities hold
for $(t,u) \in [0,\Tboot) \times [0,U_0]$:
\begin{subequations}
\begin{align} 
	\left\|
		\Psi
	\right\|_{L^{\infty}(\Sigma_t^u)}
	& \leq \Psiep + \varepsilon^{1/2},
	\label{E:PSIITSELFAUXBOOT}
		\\
	\left\|
		\Fullset_*^{[1,\Nmid;]1} \Psi
	\right\|_{L^{\infty}(\Sigma_t^u)}
	& \leq \varepsilon^{1/2},
		\label{E:PSIAUXBOOT} \\
	\left\|
		\Rad \Psi
	\right\|_{L^{\infty}(\Sigma_t^u)}
	& \leq 
	\Trandatasize
	+
	\varepsilon^{1/2}.
	\label{E:RADPSIAUXBOOT}
\end{align}
\end{subequations}

\noindent \underline{\textbf{Auxiliary bootstrap assumptions for } $\noshock$ \textbf{and} $\diffnoshock$}.
We assume that the following inequalities hold
for $(t,u) \in [0,\Tboot) \times [0,U_0]$:
\begin{subequations}
\begin{align} \label{E:NOSHOCKAUXBOOT}
	\left\|
		\Fullset^{\leq \Nmid-1;1} \noshock
	\right\|_{L^{\infty}(\Sigma_t^u)}
	& \leq \varepsilon^{1/2},
		\\
	\left\|
		\Fullset^{\leq \Nmid-1;1} \diffnoshock
	\right\|_{L^{\infty}(\Sigma_t^u)}
	& \leq 
	\varepsilon^{1/2}.
	\label{E:DIFFNOSHOCKAUXBOOT}
\end{align}
\end{subequations}

\noindent \underline{\textbf{Auxiliary bootstrap assumptions for } $\upmu$, $\upxi_j^{(Small)}$, \textbf{and} $\CoordAngSmallcomp{j}{k}$}.
We assume that the following inequalities hold
for $(t,u) \in [0,\Tboot) \times [0,U_0]$:
\begin{subequations}
\begin{align}
	\left\|
		\upmu
	\right\|_{L^{\infty}(\Sigma_t^u)}
	& 
		\leq 1
		+
		2
		\TranminusdatasizeWithFactor^{-1}
		\|\blowupcoeff \Rad \Psi \|_{L^{\infty}(\Sigma_0^u)}
		+
		\Psiep^{1/2}
		+ 
		\varepsilon^{1/2},
		\label{E:MUAUXBOOT} \\
	\left\|
		\Lunit \upmu
	\right\|_{L^{\infty}(\Sigma_t^u)}
	& \leq 
			\|\blowupcoeff \Rad \Psi \|_{L^{\infty}(\Sigma_0^u)}
			+
			\varepsilon^{1/2},
		\label{E:LUNITMUAUXBOOT} \\
	\left\|
		\Tanset_*^{[1,\Nmid-1]} \upmu
	\right\|_{L^{\infty}(\Sigma_t^u)}
	& \leq \varepsilon^{1/2},
		\label{E:MUSMALLAUXBOOT}
\end{align}
\end{subequations}
where $\blowupcoeff \neq 0$ 
(see Remark~\ref{R:BLOWUPCOEFFICIENTISNONZERO})
is the blowup-coefficient from Def.\ \ref{D:GENIUNELYNONLINEARCONSTANT}
and $\|\blowupcoeff \Rad \Psi \|_{L^{\infty}(\Sigma_0^u)} \leq C_{\star} \Trandatasize$.

Moreover, we assume that
\begin{subequations}
\begin{align}
	\left\|
		\upxi_j^{(Small)}
	\right\|_{L^{\infty}(\Sigma_t^u)}
	& \leq \Psiep^{1/2} + \mathring{\upepsilon}^{1/2},
		\label{E:CONORMALONEFORMITSELFCOMPONENTSAUXBOOT} 
		\\
	\left\|
		\Fullset_*^{[1,\Nmid-1];1} \upxi_j^{(Small)}
	\right\|_{L^{\infty}(\Sigma_t^u)}
	& \leq \varepsilon^{1/2},
		\label{E:CONORMALONEFORMCOMPONENTSAUXBOOT} 
		\\
	\left\|
		\Rad \upxi_j^{(Small)}
	\right\|_{L^{\infty}(\Sigma_t^u)}
	& \leq 
		\left\|
			\Rad \upxi_j^{(Small)}
		\right\|_{L^{\infty}(\Sigma_0^u)}
		+ \varepsilon^{1/2},
		\label{E:RADDERIVATIVECONORMALONEFORMCOMPONENTSAUXBOOT} 
		\\
	\left\|
		\Fullset^{\leq \Nmid-1;1} \CoordAngSmallcomp{j}{k}
	\right\|_{L^{\infty}(\Sigma_t^u)}
	& \leq \varepsilon^{1/2}.
	\label{E:ANGULARVECTORFIELDCOMPONENTSAUXBOOT}
\end{align}
\end{subequations}

\subsection{Smallness assumptions}
\label{SS:SMALLNESSASSUMPTIONS}
For the remainder of the article, 
when we say that ``$A$ is small relative to $B$,''
we mean that there exists a continuous increasing function 
$f :(0,\infty) \rightarrow (0,\infty)$ 
such that 
$
\displaystyle
A < f(B)
$.
To avoid lengthening the paper, we typically do not 
specify the form of $f$.

We make the following
smallness assumptions, where we will
continually adjust the required smallness
in order to close our estimates.
\begin{itemize}
	\item The bootstrap parameter $\varepsilon$ 
		is small relative to $1$.
	\item $\varepsilon$ 
		is small relative to $\Trandatasize^{-1}$,
		where $\Trandatasize$ is the data-size parameter 
		from Subsubsect.\ \ref{SSS:DATA}.
	\item $\varepsilon$ is small relative to
		the data-size parameter $\TranminusdatasizeWithFactor$ 
		from Def.\ \ref{D:SHOCKFORMATIONQUANTITY}.
	\item The data-size parameter $\Psiep$ 
		from Subsubsect.\ \ref{SSS:DATA} is small relative to $1$.
	\item 
	\begin{align} \label{E:DATAEPSILONVSBOOTSTRAPEPSILON}
		\mathring{\upepsilon} 
		& \leq \varepsilon
		< \Psiep,
	\end{align}
	where $\mathring{\upepsilon}$ is the data smallness parameter from 
	from Subsubsect.\ \ref{SSS:DATA}.
\end{itemize}
The first two assumptions will allow us to treat error terms of size 
$\varepsilon$
and 
$\varepsilon \Trandatasize$ as small quantities. The second assumption 
is relevant because the expected blowup-time is 
approximately $\TranminusdatasizeWithFactor^{-1}$,
and the assumption will allow us to show that various
error products featuring a small factor $\varepsilon$
remain small for $t < 2 \TranminusdatasizeWithFactor^{-1}$, 
which is plenty of time for us to show that a shock forms.
The smallness assumption on $\Psiep$ ensures that the solution remains
within the regime of hyperbolicity of the equations
and that $\blowupcoeff \neq 0$,
where $\blowupcoeff$ is the blowup-coefficient from
Def.\ \ref{D:GENIUNELYNONLINEARCONSTANT}.

\subsection{Existence of data verifying the size assumptions}
\label{SS:EXISTENCEOFDATA}
We now outline a proof that there exists an open set 
of data verifying the size assumptions of Subsect.\ \ref{SS:DATASIZE}
and the smallness assumptions of Subsect.\ \ref{SS:SMALLNESSASSUMPTIONS}. 
Since the assumptions are stable under Sobolev perturbations, 
it is enough to exhibit data corresponding to plane symmetric
solutions, that is, solutions that depend only on $t$ and $x^1$.
This means that along $\Sigma_0$, it is enough to exhibit appropriate 
data that depend only on $x^1$. 
To exhibit data for $\Psi$, we simply let $f(x^1)$ be any smooth non-trivial function 
that is compactly supported in $\Sigma_0^1$,
and we set $\Psi(0,x^1,\cdots,x^n) := \upkappa f(x^1)$,
where $\upkappa$ is a real parameter.
We then take vanishing data for $\noshock$, 
so that, as a consequence of the evolution
equation \eqref{E:INTRONONSHOCKEQUATION}, we have $\noshock \equiv 0$ and $\diffnoshock \equiv 0$.
With the help of these facts, it is straightforward to check that by choosing 
$\upkappa$ to be sufficiently small in magnitude,
we can satisfy all of the desired assumptions. More precisely, 
by construction, we
have $\mathring{\upepsilon} = 0$,
and by choosing $|\upkappa|$ to be small,
we can ensure that the quantity 
$\Psiep > 0$ on RHS~\eqref{E:PSIITSELFLINFTYSMALLDATAALONGSIGMA0} 
is as small as we want.

\subsection{Basic assumptions, facts, and estimates that we use silently}
\label{SS:OFTENUSEDESTIMATES}
In this subsection, we state some basic assumptions and conventions
that we silently use throughout the rest of the paper when deriving
estimates.

\begin{enumerate}
	\item All of the estimates that we derive
		hold on the bootstrap region $\mathcal{M}_{\Tboot,U_0}$.
		Moreover, in deriving estimates,
		we rely on the data-size assumptions and bootstrap assumptions 
		from Subsects.\ \ref{SS:DATASIZE}-\ref{SS:BOOTSTRAP},
		and the smallness assumptions of Subsect.\ \ref{SS:SMALLNESSASSUMPTIONS}.
	\item All quantities that we estimate can be controlled in terms of the quantities
		featured in the array $\BadVar$ from definition \eqref{E:BADARRAY}
		and their $\Fullset$-derivatives.
	\item We typically use the Leibniz rule for vectorfield differentiation
		when deriving pointwise estimates for the $\Fullset$-derivatives
		derivatives of products of the schematic form 
		$\prod_{i=1}^m p_i$. Our derivative counts are such that
		all $p_i$ except at most one are uniformly bounded in $L^{\infty}$
		on $\mathcal{M}_{\Tboot,U_0}$.
	\item The constants $C > 0$ in all of our estimates are allowed to depend on 
		the data-size parameters
		$\Trandatasize$
		and 
		$\TranminusdatasizeWithFactor^{-1}$,
		as we described in Subsect.\ \ref{SS:NOTATIONANDINDEXCONVENTIONS}.
	\item The constants $C_{\star} > 0$ do not depend on
	$\Trandatasize$
	or
	$\TranminusdatasizeWithFactor$,
	as we described in Subsect.\ \ref{SS:NOTATIONANDINDEXCONVENTIONS}.
\end{enumerate}

\subsection{Omission of the independent variables in some expressions}
\label{SS:OMISSION}
We use the following notational conventions in the rest of the article.
\begin{itemize}
	\item
	Many of our pointwise estimates are stated in the form
	\[
		|f_1| \lesssim F(t)|f_2|
	\]
	for some function $F$.
	Unless we otherwise indicate, it is understood that both $f_1$ and $f_2$
	are evaluated at the point with geometric coordinates
	$(t,u,\vartheta)$.
	\item Unless we otherwise indicate,
		in integrals $\int_{\mathcal{T}_{t,u}} f \, d \torusvol$,
		we view the integrand $f$ as a function of $(t,u,\vartheta)$, 
		and $\vartheta$ is the integration variable.
	\item Unless we otherwise indicate,
		in integrals $\int_{\Sigma_t^u} f \, d \tvol$,
		we view the integrand $f$ as a function of $(t,u',\vartheta)$, 
		and $(u',\vartheta)$ are the integration variables.
	\item Unless we otherwise indicate,
		in integrals $\int_{\mathcal{P}_u^t} f \, d \conevol$,
		we view the integrand $f$ as a function of 
		$(t',u,\vartheta)$, and $(t',\vartheta)$ are the integration variables.
	\item Unless we otherwise indicate,
		in integrals $\int_{\mathcal{M}_{t,u}} f \, d \vol$,
		we view the integrand $f$ as a function of $(t',u',\vartheta)$,
		and $(t',u',\vartheta)$ are the integration variables.
\end{itemize}

\section{Pointwise estimates and improvements of the auxiliary bootstrap assumptions}
\label{S:POINTWISEESTIMATESANDIMPROVEMENTOFAUX}
In this section, we use the data-size assumptions and bootstrap assumptions of Sect.\ \ref{S:NUMDERIVSDATASIZEANDBOOTSTRAPASSUMPTIONS}
to derive pointwise and $L^{\infty}$ estimates for various quantities.
The main result is Prop.\ \ref{P:POINTWISEANDIMPROVEMENTOFAUX}.
In particular, the results of this section yield strict improvements of the 
auxiliary bootstrap assumptions of Subsubsect.\ \ref{SSS:AUXBOOT}.

\begin{remark}
	Throughout this section, 
	we silently use the conventions described in Subsect.\ \ref{SS:OFTENUSEDESTIMATES}.
	Moreover, $\Ntop$ and $\Nmid$ denote the integers from Subsect.\ \ref{SS:NUMBEROFDERIVATIVES}.
\end{remark}

\subsection{Commutator estimates}
\label{SS:COMMUTATORESTIMATES}
We start by providing some commutator estimates that we will use throughout the analysis.

\begin{lemma}[\textbf{Commutator estimates}]
	\label{L:COMMUTATORESTIMATES}
	Let $1 \leq N \leq \Ntop$ be an integer,
	let $\vec{I}$ be a multi-index for the set $\Tanset$ of $\mathcal{P}_u$-tangent 
	commutation vectorfields such that $|\vec{I}| = N$,
	and let $\vec{J}$ be any permutation of $\vec{I}$
	(in particular, $|\vec{I}| = |\vec{J}| = N \leq \Ntop$).
	Then the following identity for scalar functions $f$ holds:
	\begin{align} \label{E:TANGENTVECTORFIELDSCOMMUTEWITHEACHOTHER} 
		\Tanset^{\vec{I}} f
		-
		\Tanset^{\vec{J}} f
		& = 0.
		\end{align}
	
	Let $1 \leq N \leq \Ntop$ be an integer.
	Then the following commutator estimate for scalar functions $f$ holds:
	\begin{align} \label{E:COMMUTATORESTIMATEFORLUNITWITHONETRANSVERSAL}
		\left|
			[\Lunit,\Fullset^{N;1}]f
		\right|
		& \lesssim 
				\left|
					\Tanset_*^{[1,N]} f
				\right|
				+
				\underbrace{
				\left|
					\Tanset_*^{[1,\lfloor N/2 \rfloor]} f
				\right|
				\left|
					\Fullset_*^{[1,N];1} \Psi
				\right|
				}_{\mbox{\upshape Absent if $N=1$}}
				+
				\underbrace{
				\left|
					\Tanset_*^{[1,\lfloor N/2 \rfloor]} f
				\right|
				\left|
					\Tanset_*^{[1,N-1]} \BadVar
				\right|}_{\mbox{\upshape Absent if $N=1$}}.
		\end{align}

	Let $2 \leq N \leq \Ntop$ be an integer,
	let $\vec{I} \in \mathcal{I}_*^{[1,N];1}$ (see Def.\ \ref{D:COMMUTATORMULTIINDICES}), 
	and let $\vec{J}$ be any permutation of $\vec{I}$.
	Then the following commutator estimate for scalar functions $f$ holds:
	\begin{align} \label{E:COMMUTATORESTIMATEFORONETRANSVERSAL}
		\left|
			\Fullset^{\vec{I}} f
			-
			\Fullset^{\vec{J}} f
		\right|
		& \lesssim 
				\left|
					\Tanset_*^{[1,N-1]} f
				\right|
				+
				\left|
					\Tanset_*^{[1,\lfloor N/2 \rfloor]} f
				\right|
				\left|
					\Fullset_*^{[1,N-1];1} \GdVar
				\right|
				+
				\left|
					\Tanset_*^{[1,\lfloor N/2 \rfloor]} f
				\right|
				\left|
					\Tanset_*^{[1,N-1]} \BadVar
				\right|.
		\end{align}
\end{lemma}

\begin{proof}
	\eqref{E:TANGENTVECTORFIELDSCOMMUTEWITHEACHOTHER} is a trivial consequence
	of the commutation identity \eqref{E:LUNITANGULARCOMMUTATOR}.
	
	The estimate
	\eqref{E:COMMUTATORESTIMATEFORLUNITWITHONETRANSVERSAL}
	is a straightforward consequence of the commutation
	identities \eqref{E:LUNITANGULARCOMMUTATOR}-\eqref{E:LUNITRADCOMMUTATOR}
	and the bootstrap assumptions.
	
	Similarly, the estimate \eqref{E:COMMUTATORESTIMATEFORONETRANSVERSAL}
	is a straightforward consequence of the commutation
	identities \eqref{E:LUNITANGULARCOMMUTATOR}-\eqref{E:RADCOORDANGCOMMUTATOR}
	and the bootstrap assumptions.
\end{proof}

\subsection{Transversal derivatives in terms of tangential derivatives}
\label{SS:ALGEBRAICEXPRESSIONSFORTRANSVERSALDERIVATIVES}
The next lemma, which is algebraic in nature, plays a crucial role in controlling $\noshock$ and $\diffnoshock$.
Roughly, the lemma shows that the $\Rad$ derivative of these quantities can be expressed in terms
of their $\mathcal{P}_u$-tangential derivatives plus error terms. 
In particular, this means 
that we do not have to commute the evolution equations for
$\noshock$ and $\diffnoshock$ with $\Rad$ in order to control $\Rad \noshock$ and $\Rad \diffnoshockdownarg{\alpha}$;
we can instead use the equations to algebraically solve for the $\Rad$ derivative.
This is important because commuting these equations (which must be weighted with $\upmu$ to avoid singular terms)
with $\Rad$
would generate the error term $\Rad \upmu$, 
which is uncontrollable based on the degree of $\Rad$-differentiability 
that we have imposed on $\Psi$.

\begin{lemma}[\textbf{Algebraic expressions for transversal derivatives in terms of tangential derivatives}]
	\label{L:ALGEBRAICTRANSVERSALINTERMSOFTANGENTIAL}
	There exist smooth functions of $\GdVar$, schematically denoted by
	$\smoothfunction$, such that the following 
	algebraic identities hold whenever $|\GdVar|$ is sufficiently small
	(where $\Singletan \in \Tanset$ and $Z \in \Fullset$):
	\begin{align} \label{E:VTRANSVERSALINTERMSOFTANGENTIAL}
		\Rad \noshock
		& = \upmu \smoothfunction(\GdVar) \diffnoshock,
			\\
		\Rad \diffnoshockdownarg{\alpha}
		& = \smoothfunction(\BadVar) \Singletan \diffnoshock
				+
				\smoothfunction(\BadVar,Z \Psi) \diffnoshock.
				\label{E:VCARTESIANDIFFERENTITEDTRANSVERSALINTERMSOFTANGENTIAL}
	\end{align}
\end{lemma}

\begin{proof}
	To prove \eqref{E:VTRANSVERSALINTERMSOFTANGENTIAL},
	we first multiply equation \eqref{E:NONSHOCKEQUATION} by $\upmu$ 
	and use Lemma~\ref{L:CARTESIANDERIVATIVESINTERMSOFGEOMETRICDERIVATIVES}
	to obtain the following identity, whose RHS is written in schematic form:
	$\upmu (A^0 + A^a \upxi_a) \Radunit v 
	= \upmu \smoothfunction(\GdVar) \Singletan v 
	= \upmu \smoothfunction(\GdVar) \Singletan^{\alpha} \diffnoshockdownarg{\alpha} 
	= \upmu \smoothfunction(\GdVar) \diffnoshock$.
	Next, using Def.\ \ref{D:PERTURBEDPART}, we see that
	$\upmu (A^0 + A^a \upxi_a) \Radunit v = (A^0 - A^1 + A_{(Small)}) \Rad v$,
	where 
	$A^0 - A^1$ is a matrix whose entries are of the schematic form $\smoothfunction(\GdVar)$ 
	and
	$A_{(Small)}$ is a matrix whose entries are of the schematic form
	$\GdVar \smoothfunction(\GdVar)$. From these facts and the assumption
	\eqref{E:POSITIVEDEFMATRICES}, it follows that whenever $|\GdVar|$ is sufficiently small,
	the matrix $A^0 - A^1 + A_{(Small)}$ is invertible. From this fact, the desired identity \eqref{E:VTRANSVERSALINTERMSOFTANGENTIAL}
	easily follows.
	
	The proof of \eqref{E:VCARTESIANDIFFERENTITEDTRANSVERSALINTERMSOFTANGENTIAL} is 
	based on equation \eqref{E:CARTESIANDIFFERENTIATEDNONSHOCKEQUATION}
	and is similar but requires one new ingredient:
	we use Lemma~\ref{L:CARTESIANDERIVATIVESINTERMSOFGEOMETRICDERIVATIVES} to (schematically) express
	RHS~\eqref{E:CARTESIANDIFFERENTIATEDNONSHOCKEQUATION} 
	as $\smoothfunction(\BadVar,Z \Psi) \diffnoshock$.
	
\end{proof}

With the help of Lemmas~\ref{L:COMMUTATORESTIMATES} and
\ref{L:ALGEBRAICTRANSVERSALINTERMSOFTANGENTIAL}, we now 
derive pointwise estimates showing that
the derivatives of
$\noshock$ and $\diffnoshock$
involving up to one $\Rad$ differentiation
can be controlled in terms of quantities that
do not depend on the $\Rad$ derivatives of $\noshock$ and $\diffnoshock$.

\begin{lemma}[\textbf{Pointwise estimates for transversal derivatives in terms of tangential derivatives}]
	\label{L:CONTROLOFTRANSVERSALINTERMSOFTANGENTIAL}
	The following estimates hold for $1 \leq N \leq \Ntop$:
	\begin{align} \label{E:POINTWISEESTIMATESTRANSVERSALDERIVATIVESOFNOSHOCKCONTROLLEDBYTANGENTAIL} 
		\left|
			\Fullset^{N;1} \noshock
		\right|
		& \lesssim 
		\left|
			\Fullset_*^{[1,N];1} \Psi
		\right|
		+
		\left|
			\Tanset^{\leq N-1} \noshock
		\right|
		+
		\left|
			\Tanset^{\leq N-1} \diffnoshock
		\right|
			\\
	& \ \
		+
		\sum_{j=1}^n
		\left|
			\Fullset_*^{[1,N-1];1} \upxi_j^{(Small)}
		\right|
		+
		\sum_{i=2}^n
		\sum_{j=1}^n
		\left|
			\Fullset_*^{[1,N-1];1} \CoordAngSmallcomp{i}{j}
		\right|
		+
		\left|
			\Tanset_*^{[1,N-1]} \upmu
		\right|.
		\notag
	\end{align}
	
	Moreover, the following estimates hold for $1 \leq N \leq \Ntop-1$:
	\begin{align}
		\left|
			\Fullset^{N;1} \diffnoshock
		\right|
		& \lesssim 
		\left|
			\Fullset_*^{[1,N];1} \Psi
		\right|
		+
		\left|
			\Tanset^{\leq N} \noshock
		\right|
		+
		\left|
			\Tanset^{\leq N} \diffnoshock
		\right|
			\label{E:POINTWISEESTIMATESTRANSVERSALDERIVATIVESOFCARTESIANDIFFERENTIATEDNOSHOCKCONTROLLEDBYTANGENTAIL} \\
	& \ \
		+
		\sum_{j=1}^n
		\left|
			\Fullset_*^{[1,N];1} \upxi_j^{(Small)}
		\right|
		+
		\sum_{i=2}^n
		\sum_{j=1}^n
		\left|
			\Fullset_*^{[1,N];1} \CoordAngSmallcomp{i}{j}
		\right|
		+
		\left|
			\Tanset_*^{[1,N-1]} \upmu
		\right|.
			\notag 
	\end{align}
	\end{lemma}

\begin{proof}
	We will prove \eqref{E:POINTWISEESTIMATESTRANSVERSALDERIVATIVESOFNOSHOCKCONTROLLEDBYTANGENTAIL}-\eqref{E:POINTWISEESTIMATESTRANSVERSALDERIVATIVESOFCARTESIANDIFFERENTIATEDNOSHOCKCONTROLLEDBYTANGENTAIL} simultaneously
	by using induction in $N$.
	The base case $N=1$ can be handled using
	the same arguments given below and we omit these details.
	We therefore assume the induction hypothesis that 
	\eqref{E:POINTWISEESTIMATESTRANSVERSALDERIVATIVESOFNOSHOCKCONTROLLEDBYTANGENTAIL}-\eqref{E:POINTWISEESTIMATESTRANSVERSALDERIVATIVESOFCARTESIANDIFFERENTIATEDNOSHOCKCONTROLLEDBYTANGENTAIL}
	have been proved with $N-1$ in the role of $N$; 
	To prove \eqref{E:POINTWISEESTIMATESTRANSVERSALDERIVATIVESOFNOSHOCKCONTROLLEDBYTANGENTAIL}-\eqref{E:POINTWISEESTIMATESTRANSVERSALDERIVATIVESOFCARTESIANDIFFERENTIATEDNOSHOCKCONTROLLEDBYTANGENTAIL}
	in the case $N$, we first consider an order $N$ operator of the form $\Tanset^{N-1} \Rad$.
	Using \eqref{E:VCARTESIANDIFFERENTITEDTRANSVERSALINTERMSOFTANGENTIAL},
	we deduce that
	$\Tanset^{N-1} \Rad V_{\alpha}
		=
		\Tanset^{N-1}
		\left\lbrace
			\smoothfunction(\BadVar) \Singletan \diffnoshock
				+
				\smoothfunction(\BadVar,Z \Psi) \diffnoshock
		\right\rbrace
	$.			
	From this expression and the bootstrap assumptions,
	we deduce that
	$
	\left|
		\Tanset^{N-1} \Rad V_{\alpha}
	\right|
	\lesssim \mbox{RHS~\eqref{E:POINTWISEESTIMATESTRANSVERSALDERIVATIVESOFCARTESIANDIFFERENTIATEDNOSHOCKCONTROLLEDBYTANGENTAIL}}
	$
	as desired.
	Then using the commutator estimate \eqref{E:COMMUTATORESTIMATEFORONETRANSVERSAL} and the bootstrap assumptions,
	we can arbitrarily permute the vectorfield factors in
	$\Tanset^{N-1} \Rad V_{\alpha}$
	up to error terms that are pointwise bounded in magnitude by
	$
	\lesssim \mbox{RHS~\eqref{E:POINTWISEESTIMATESTRANSVERSALDERIVATIVESOFCARTESIANDIFFERENTIATEDNOSHOCKCONTROLLEDBYTANGENTAIL}}
	$
	plus error terms of the form
	$\left|
			\Fullset_*^{\leq N-1;1} \noshock
		\right|
		+
		\left|
			\Fullset_*^{\leq N-1;1} \diffnoshock
		\right|
$,
which (by the induction hypothesis) have already been shown to be bounded by
$
\lesssim \mbox{RHS~\eqref{E:POINTWISEESTIMATESTRANSVERSALDERIVATIVESOFCARTESIANDIFFERENTIATEDNOSHOCKCONTROLLEDBYTANGENTAIL}}
$. We have therefore obtained the desired bounds for $\diffnoshock$ in the case that $\Fullset^{N;1}$
contains a factor of $\Rad$. 
	In the case that the operator
	$\Fullset^{N;1}$ contains a factor of $\Rad$,
	the estimate \eqref{E:POINTWISEESTIMATESTRANSVERSALDERIVATIVESOFNOSHOCKCONTROLLEDBYTANGENTAIL} for $\noshock$
	follows similarly with the help of equation \eqref{E:VTRANSVERSALINTERMSOFTANGENTIAL}.
	To prove \eqref{E:POINTWISEESTIMATESTRANSVERSALDERIVATIVESOFNOSHOCKCONTROLLEDBYTANGENTAIL}
	in the case that the operator
	$\Fullset^{N;1}$ does not contain a factor of $\Rad$, that is, that $\Fullset^{N;1} = \Tanset^N$,
	we first write 
	$\Tanset^N \noshock 
	= \Tanset^{N-1}(\Singletan^{\alpha} \partial_{\alpha} \noshock)
	= \Tanset^{N-1}(\Singletan^{\alpha} \diffnoshockdownarg{\alpha})
	= \Tanset^{N-1}(\smoothfunction(\GdVar) \diffnoshockdownarg{\alpha})
	$.
	From this expression and the bootstrap assumptions, we
	bound the magnitude of the RHS of this equation by
	$\lesssim \mbox{RHS~\eqref{E:POINTWISEESTIMATESTRANSVERSALDERIVATIVESOFNOSHOCKCONTROLLEDBYTANGENTAIL}}$
	as desired. 
	In the case that
$\Fullset^{N;1}$ does not contain a factor of $\Rad$, that is, that $\Fullset^{N;1} = \Tanset^N$,
the estimate \eqref{E:POINTWISEESTIMATESTRANSVERSALDERIVATIVESOFCARTESIANDIFFERENTIATEDNOSHOCKCONTROLLEDBYTANGENTAIL} 
is trivial.
	We have therefore closed the induction. We clarify that in the final step,
	we allow $N = \Ntop$ in
	\eqref{E:POINTWISEESTIMATESTRANSVERSALDERIVATIVESOFNOSHOCKCONTROLLEDBYTANGENTAIL},
	but not in \eqref{E:POINTWISEESTIMATESTRANSVERSALDERIVATIVESOFCARTESIANDIFFERENTIATEDNOSHOCKCONTROLLEDBYTANGENTAIL}.
\end{proof}

\subsection{Pointwise estimates and improvements of the auxiliary bootstrap assumptions}
\label{SS:POINTWISEANDIMPROVEMENTOFAUX}
We now state and prove the main result of this section.

\begin{proposition}[\textbf{Pointwise estimates and improvements of the auxiliary bootstrap assumptions}]	
	\label{P:POINTWISEANDIMPROVEMENTOFAUX}
	Let $\Ntop$ and $\Nmid$ be the integers fixed in Subsect.\ \ref{SS:NUMBEROFDERIVATIVES}.
	If $N \leq \Ntop$, then the following estimates hold
	(see Subsect.\ \ref{SS:STRINGSOFCOMMUTATIONVECTORFIELDS} regarding the vectorfield differential operator notation).
	
	\noindent \underline{\textbf{Pointwise estimates for the commuted evolution equations of} $\Psi$, $\noshock$ \textbf{and} $\diffnoshock$}.
	\begin{align}
		\left|
			\Lunit \Fullset^{N;1} \Psi
		\right|
		& \lesssim 
		\left|
			\Fullset_*^{[1,N];1} \Psi
		\right|
		+
		\left|
			\Fullset_*^{[1,N-1];1} \GdVar
		\right|
		+
		\left|
			\Tanset_*^{[1,N-1]} \BadVar
		\right|.
			\label{E:SHOCKVARIABLEPOINTWISEINHOM}
	\end{align}
	
	Similarly, if $1 \leq N \leq \Ntop$, then the following pointwise estimates hold:
	\begin{subequations}
	\begin{align}
		\left|
			\upmu A^{\alpha} \partial_{\alpha} \Tanset^{N-1} v
		\right|
		& \lesssim 
		\left|
			\Fullset_*^{[1,N];1} \Psi
		\right|
		+
		\left|
			\Fullset_*^{[1,N-1];1} \GdVar
		\right|
		+
		\left|
			\Tanset_*^{[1,N-1]} \BadVar
		\right|,
			\label{E:NONSHOCKVARIABLEPOINTWISEINHOM}
				\\
		\left|
			\upmu A^{\alpha} \partial_{\alpha} \Tanset^{N-1} V_{\alpha}
		\right|
		& \lesssim 
		\left|
			\Fullset_*^{[1,N];1} \Psi
		\right|
		+
		\left|
			\Fullset_*^{[1,N-1];1} \GdVar
		\right|
		+
		\left|
			\Tanset_*^{[1,N-1]} \BadVar
		\right|
		+
		|\diffnoshock|.
			\label{E:NONSHOCKCARTESIANDIFFERENTIATEDVARIABLEPOINTWISEINHOM}	
\end{align}
\end{subequations}

\medskip
\noindent \underline{\textbf{Pointwise estimates for the commuted evolution equations of} 
$\upxi_j^{(Small)}$, $\CoordAngSmallcomp{i}{j}$, \textbf{and} $\upmu$}.
If $1 \leq N \leq \Ntop$,
then the following estimates hold:
\begin{subequations}
	\begin{align}
		\left|
			\Lunit \Fullset^{N-1;1} \upxi_j^{(Small)}
		\right|
		& \lesssim 
			\left|
				\Fullset_*^{[1,N];1} \Psi
			\right|
			+
			\left|
				\Fullset_*^{[1,N-1];1} \GdVar
			\right|
			+
			\left|
				\Tanset_*^{[1,N-1]} \BadVar
			\right|
			+
			|\diffnoshock|,
			\label{E:COMMUTEDLUNITUPXIJ}
				\\
		\left|
			\Lunit \Fullset^{N-1;1} \CoordAngSmallcomp{i}{j}
		\right|
		& \lesssim 
			\left|
				\Fullset_*^{[1,N];1} \Psi
			\right|
			+
			\left|
				\Fullset_*^{[1,N-1];1} \GdVar
			\right|
			+
			\left|
				\Tanset_*^{[1,N-1]} \BadVar
			\right|
			+
			|\diffnoshock|.
		\label{E:COMMUTEDLUNITCOORDANGIJ}
	\end{align}
	\end{subequations}
	
	Furthermore, if $2 \leq N \leq \Ntop$, then the following estimates hold:
	\begin{align}
		\left|
			\Lunit \Tanset^{N-1} \upmu
		\right|
		& \lesssim 
		\left|
			\Fullset_*^{[1,N];1} \Psi
		\right|
		+
		\left|
			\Tanset^{[1,N-1]} \GdVar
		\right|
		+
		\left|
			\Tanset_*^{[1,N-1]} \BadVar
		\right|.
			\label{E:COMMUTEDLUNITUPMU} 
	\end{align}
	
	\medskip
	
\noindent \underline{$L^{\infty}$ \textbf{estimates for} $\Psi$}.
	In addition, the following estimates hold:
	\begin{subequations}
	\begin{align}  
		\left\|
			\Psi
		\right\|_{L^{\infty}(\Sigma_t^u)}
		& \leq
		\Psiep
		+
		C \varepsilon,
		\label{E:PSIITSELFLINFTY}
			\\
		\left\|
			\Fullset_*^{[1,\Nmid];1} \Psi
		\right\|_{L^{\infty}(\Sigma_t^u)}
		& \leq C \varepsilon,
			\label{E:LINFTYPSICOMMUTED} \\
		\left\|
			\Rad \Psi
		\right\|_{L^{\infty}(\Sigma_t^u)}
		& \leq
		\Trandatasize
		+
		C \varepsilon.
		\label{E:RADPSILINFTY}
	\end{align}
\end{subequations}

\medskip

\noindent \underline{$L^{\infty}$ \textbf{estimates for} $\noshock$ \textbf{and} $\diffnoshock$}.
Moreover, the following estimates hold:
\begin{subequations}
\begin{align} \label{E:NOSHOCKLINFTY}
	\left\|
		\Fullset^{\leq \Nmid-1;1} \noshock
	\right\|_{L^{\infty}(\Sigma_t^u)}
	& \leq C \varepsilon,
		\\
	\left\|
		\Fullset^{\leq \Nmid-1;1} \diffnoshock
	\right\|_{L^{\infty}(\Sigma_t^u)}
	& \leq C \varepsilon.
	\label{E:DIFFNOSHOCKLINFTY}
\end{align}
\end{subequations}

\medskip

\noindent \underline{$L^{\infty}$ \textbf{estimates for} $\upxi_j^{(Small)}$, $\CoordAngSmallcomp{i}{j}$, \textbf{and} $\upmu$}.
The following estimates hold:
\begin{subequations}
	\begin{align}
		\left\|
			\upxi_j^{(Small)}
		\right\|_{L^{\infty}(\Sigma_t^u)}
		& \leq C_{\star} \Psiep \delta_j^1 + C \varepsilon,
			\label{E:LINFTYUPXIJITSELF}
				\\
		\left\|
			\Fullset_*^{[1,\Nmid-1];1} \upxi_j^{(Small)}
		\right\|_{L^{\infty}(\Sigma_t^u)}
		& \leq C \varepsilon,
			\label{E:LINFTYUPXIJCOMMUTED}
				\\
		\left\|
			\Rad \upxi_j^{(Small)}
		\right\|_{L^{\infty}(\Sigma_t^u)}
		& \leq 
			\left\|
				\Rad \upxi_j^{(Small)}
			\right\|_{L^{\infty}(\Sigma_0^u)}
			+ 
			C \varepsilon,
			\label{E:LINFTYUPXIJRADCOMMUTED}
				\\
		\left\|
			\Fullset^{\leq \Nmid-1;1} \CoordAngSmallcomp{i}{j}
		\right\|_{L^{\infty}(\Sigma_t^u)}
		& \leq C \varepsilon,
		\label{E:LINFTYCOORDANGIJCOMMUTED}
			\\
		\left\|
			\Tanset_*^{[1,\Nmid-1]} \upmu
		\right\|_{L^{\infty}(\Sigma_t^u)}
		& \leq C \varepsilon.
			\label{E:LINFTYUPMUCOMMUTED}
	\end{align}
	\end{subequations}
	
\medskip
		
\noindent \underline{\textbf{Sharp estimates for} $\upmu$ \textbf{and} $\Lunit \upmu$}.
In addition, the following pointwise estimates hold:
\begin{subequations}
\begin{align} 
	\upmu(t,u,\vartheta)
	& = 1
		+
		t [\blowupcoeff \Rad \Psi](0,u,\vartheta)
		+
		\mathcal{O}_{\star}(\Psiep) 
		+ 
		\mathcal{O}(\varepsilon)
			\label{E:MUNOTCOMMUTED} \\
	& = 1
		+
		t \blowupcoeff(0,u,\vartheta) \Rad \Psi(t,u,\vartheta)
		+
		\mathcal{O}_{\star}(\Psiep) 
		+ 
		\mathcal{O}(\varepsilon),
						\notag
			\\
	\Lunit \upmu(t,u,\vartheta)
	& = 
		[\blowupcoeff \Rad \Psi](0,u,\vartheta)
		+
		\mathcal{O}(\varepsilon)
		\label{E:LUNITMUNOTCOMMUTED}
			\\
	& = 
		\left\lbrace
			\blowupcoeff|_{(\Psi,\noshock) = (0,0)}
			+
			\mathcal{O}_{\star}(\Psiep)
		\right\rbrace
		\Rad \Psi(t,u,\vartheta)
		+
		\mathcal{O}(\varepsilon),
		\notag
\end{align}
\end{subequations}
where the blowup-coefficient $\blowupcoeff$ is defined in Def.\ \ref{D:GENIUNELYNONLINEARCONSTANT}.

Moreover,
\begin{subequations}
\begin{align} 
	\left\|
		\upmu
	\right\|_{L^{\infty}(\Sigma_t^u)}
	& \leq
		1
		+
		2
		\TranminusdatasizeWithFactor^{-1}
		\|\blowupcoeff \Rad \Psi \|_{L^{\infty}(\Sigma_0^u)}
		+
		C_{\star} \Psiep
		+ 
		C \varepsilon,
		\label{E:UPMULINFINTYIMPROVED}
			\\
	\left\|
		\Lunit \upmu
	\right\|_{L^{\infty}(\Sigma_t^u)}
	& \leq
		\|\blowupcoeff \Rad \Psi \|_{L^{\infty}(\Sigma_0^u)}
		+ 
		C \varepsilon.
		\label{E:LUPMULINFINTYIMPROVED}
\end{align}	
\end{subequations}

	\medskip
	\noindent \underline{\textbf{Estimates for} $\Xi^j$}.
	Finally, 
	if $1 \leq N \leq \Ntop$,
	then the following estimates hold
	for the Cartesian components $\Xi^j$ of the $\mathcal{T}_{t,u}$-tangent
	vectorfield $\Xi$ from \eqref{E:RADDECOMPINTOPARTIALUPLUSANGULAR}:
	\begin{subequations}
	\begin{align} \label{E:LUNITBIGXIJPOINTWISE}
		\left|
			\Lunit \Tanset^{\leq N-1} \Xi^j
		\right|
		& \lesssim 
			\left|
			\Tanset^{\leq N-1} \Xi^j
			\right|
			+ 
			\left|
				\Fullset_*^{\leq N;1} \Psi
			\right|
			+
			\left|
				\Tanset^{[1,N-1]} \BadVar
			\right|
			+
			1,
			\\
		\left\|
			\Tanset^{\leq \Nmid-1} \Xi^j
		\right\|_{L^{\infty}(\Sigma_t^u)}
		& \lesssim 1.
		\label{E:BIGXIJLINFITY}
	\end{align}
	\end{subequations}
\end{proposition}

\begin{remark}[\textbf{Strict improvements of the auxiliary bootstrap assumptions}]
	\label{R:STRICTIMPROVEMENTOFAUX}
	The $L^{\infty}$ estimates of Prop.\ \ref{P:POINTWISEANDIMPROVEMENTOFAUX} provide,
	in particular, strict improvements of the auxiliary bootstrap assumptions
	of Subsubsect.\ \ref{SSS:AUXBOOT} whenever $\Psiep$ and $\varepsilon$ are sufficiently small.
\end{remark}

\begin{proof}[Proof of Prop.\ \ref{P:POINTWISEANDIMPROVEMENTOFAUX}]
	See Subsect.\ \ref{SS:OFTENUSEDESTIMATES} for some comments on the analysis.
	We start by noting that the order in which we prove estimates is important.
	Throughout the proof, we use the phrase
	``conditions on the data''
	to mean the assumptions from Subsubsect.\ \ref{SSS:DATA} for the data of $\Psi$, $\noshock$, and $\diffnoshock$,
	as well as the estimates from Subsubsect.\ \ref{SSS:DATAOFOTHERDATASIZE}
	for the data of 
	$\upmu$,
	$\upxi_j^{(Small)}$,
	$\CoordAngSmallcomp{i}{j}$,
	and
	$\Xi^j$.
	We also silently use \eqref{E:DATAEPSILONVSBOOTSTRAPEPSILON}.

	\medskip	
	
\noindent \underline{\textbf{Proof of \eqref{E:SHOCKVARIABLEPOINTWISEINHOM}}}:
	The estimate 
	follows from the evolution equation \eqref{E:SHOCKEQN},
	the commutator estimate \eqref{E:COMMUTATORESTIMATEFORLUNITWITHONETRANSVERSAL},
	and the bootstrap assumptions.
	
	\medskip
	
	\noindent \underline{\textbf{Proof of \eqref{E:COMMUTEDLUNITUPMU}}}:
	We first schematically write \eqref{E:LUNITUPMU} as
	$\Lunit \upmu 
	= 
	\smoothfunction(\GdVar) \Rad \Psi
	+
	\smoothfunction(\BadVar) \Lunit \Psi
	+
	\smoothfunction(\BadVar) \diffnoshock
	$.
	Hence, using \eqref{E:TANGENTVECTORFIELDSCOMMUTEWITHEACHOTHER},
	we deduce
	$\Lunit \Tanset^{N-1} \upmu 
	= 
	\Tanset^{N-1}
	\left\lbrace
			\smoothfunction(\GdVar) \Rad \Psi
			+
			\smoothfunction(\BadVar) \Lunit \Psi
			+
			\smoothfunction(\BadVar) \diffnoshock
	\right\rbrace
	$.
	The desired bound \eqref{E:COMMUTEDLUNITUPMU} then follows from this equation and the bootstrap assumptions
	(we stress that the assumption $N \geq 2$ is needed for this estimate).
	
	\medskip
	
	\noindent \underline{\textbf{Proof of \eqref{E:COMMUTEDLUNITUPXIJ} and \eqref{E:COMMUTEDLUNITCOORDANGIJ}}}:
	We first schematically write \eqref{E:LUNITUPXIJ}
	as
	$\Lunit \upxi_j^{(Small)}
	= 
	\smoothfunction(\GdVar) \Singletan \Psi 
	+
	\smoothfunction(\GdVar) \diffnoshock
	$.
	Hence,
	$\Lunit \Fullset^{N-1;1} \upxi_j^{(Small)}
	= 
	[\Lunit,\Fullset^{N-1;1}] \upxi_j^{(Small)}
	+
	\Fullset^{N-1;1}
	\left\lbrace
		\smoothfunction(\GdVar) \Singletan \Psi 
		+
		\smoothfunction(\GdVar) \diffnoshock
	\right\rbrace
	$.
	To bound the magnitude of the term
	$\Fullset^{N-1;1}
	\left\lbrace
		\cdots
	\right\rbrace$
	by
	$\lesssim \mbox{RHS~\eqref{E:COMMUTEDLUNITUPXIJ}}$,
	we use the bootstrap assumptions. To bound the commutator term
	$[\Lunit,\Fullset^{N-1;1}] \upxi_j^{(Small)}$, we also use
	\eqref{E:COMMUTATORESTIMATEFORLUNITWITHONETRANSVERSAL}.
	The estimate \eqref{E:COMMUTEDLUNITCOORDANGIJ} can be proved in the same way 
	as the estimate \eqref{E:COMMUTEDLUNITUPXIJ},
	since by \eqref{E:COORDANGSMALLCOMPONENTRECTANGULAR},
	$\CoordAngSmallcomp{i}{j}$ obeys a schematically identical evolution equation:
	$\Lunit \CoordAngSmallcomp{i}{j}
	=
	\smoothfunction(\GdVar) \Singletan \Psi 
	+
	\smoothfunction(\GdVar) \diffnoshock
	$.
	
	\medskip
	
	\noindent \underline{\textbf{Proof of \eqref{E:LINFTYPSICOMMUTED}, 
	\eqref{E:LINFTYUPXIJCOMMUTED},
	\eqref{E:LINFTYCOORDANGIJCOMMUTED},
	and \eqref{E:LINFTYUPMUCOMMUTED}}}:
	We set
	\begin{align}
		q = q(t,u,\vartheta) 
		& := 
		\left|
			\Fullset_*^{[1,\Nmid];1} \Psi
		\right|
			\\
	& \ \
		+
		\sum_{j=1}^n
		\left|
			\Fullset_*^{[1,\Nmid-1];1} \upxi_j^{(Small)}
		\right|
		+
		\sum_{i=2}^n
		\sum_{j=1}^n
		\left|
			\Fullset^{\leq \Nmid-1;1} \CoordAngSmallcomp{i}{j}
		\right|
		+
		\left|
			\Tanset_*^{[1,\Nmid-1]} \upmu
		\right|.
		\notag
	\end{align}
	From 
	\eqref{E:SHOCKVARIABLEPOINTWISEINHOM},
	\eqref{E:COMMUTEDLUNITUPXIJ}-\eqref{E:COMMUTEDLUNITUPMU},
	the pointwise estimates of Lemma~\ref{E:POINTWISEESTIMATESTRANSVERSALDERIVATIVESOFNOSHOCKCONTROLLEDBYTANGENTAIL},
	the fundamental bootstrap assumptions \eqref{E:FUNDAMENTALBOOTSTRAP},
	and the fundamental theorem of calculus,
	we deduce,
	in view of the fact that 
	$\Lunit = \frac{\partial}{\partial t}$,
	that
	$
	|q(t,u,\vartheta)|
		\leq
		|q(0,u,\vartheta)|
		+
		c
		\int_{s=0}^t 
			|q(s,u,\vartheta)|
		\, ds
		+
		C \varepsilon
	$.
	Moreover, the conditions on the data imply that 
	$|q(0,u,\vartheta)| \leq C \varepsilon$.
	Hence, from Gronwall's inequality, we deduce that
	$|q(t,u,\vartheta)| \lesssim \varepsilon \exp(ct) \lesssim \varepsilon$,
	which implies all four of the desired bounds.
	
	\medskip
	\noindent \underline{\textbf{Proof of \eqref{E:PSIITSELFLINFTY}, \eqref{E:RADPSILINFTY}, \eqref{E:LINFTYUPXIJITSELF},
	and \eqref{E:LINFTYUPXIJRADCOMMUTED}}}:
	To prove \eqref{E:PSIITSELFLINFTY},
	we first use the
	fundamental theorem of calculus to obtain
	$|\Psi|(t,u,\vartheta)
	\leq |\Psi|(0,u,\vartheta)
	+ 
	\int_{s=0}^t
		|\Lunit \Psi |(s,u,\vartheta)
	\, ds
	$.
	The estimate \eqref{E:LINFTYPSICOMMUTED} implies that
	the time integral in the previous inequality is $\lesssim \varepsilon$.
	In view of the conditions on the data, we conclude
	\eqref{E:PSIITSELFLINFTY}.
	The remaining three estimates can be proved similarly
	with the help of the estimates \eqref{E:LINFTYPSICOMMUTED} and \eqref{E:LINFTYUPXIJCOMMUTED}.
	
	\medskip
	
	\noindent \underline{\textbf{Proof of \eqref{E:NOSHOCKLINFTY}-\eqref{E:DIFFNOSHOCKLINFTY}}}:
	These estimates follow from the pointwise estimates 
	\eqref{E:POINTWISEESTIMATESTRANSVERSALDERIVATIVESOFNOSHOCKCONTROLLEDBYTANGENTAIL}-\eqref{E:POINTWISEESTIMATESTRANSVERSALDERIVATIVESOFCARTESIANDIFFERENTIATEDNOSHOCKCONTROLLEDBYTANGENTAIL},
	the fundamental bootstrap assumptions \eqref{E:FUNDAMENTALBOOTSTRAP},
	and the estimates
	\eqref{E:LINFTYPSICOMMUTED},
	\eqref{E:LINFTYUPXIJCOMMUTED},
	\eqref{E:LINFTYCOORDANGIJCOMMUTED},
	and
	\eqref{E:LINFTYUPMUCOMMUTED}.
	
	\medskip
	
	\noindent \underline{\textbf{Proof of 
	\eqref{E:NONSHOCKVARIABLEPOINTWISEINHOM}-\eqref{E:NONSHOCKCARTESIANDIFFERENTIATEDVARIABLEPOINTWISEINHOM}}}:
	We first use Lemma~\ref{L:CARTESIANDERIVATIVESINTERMSOFGEOMETRICDERIVATIVES}
	to deduce the schematic relation
	\begin{align} \label{E:SCHEMATICUPMUTIMESCARTESIANDERIVATIVEINTERMSOFGEOMETRICDERIVATIVES}
		\upmu \partial_{\alpha}
		& = \smoothfunction(\GdVar) \Rad
			+
			\upmu \smoothfunction(\GdVar) \Singletan
			= \smoothfunction(\GdVar) \Rad
			+
			\smoothfunction(\BadVar) \Singletan.
	\end{align}
	Next, using \eqref{E:SCHEMATICUPMUTIMESCARTESIANDERIVATIVEINTERMSOFGEOMETRICDERIVATIVES},
	the definition $\partial_{\alpha} \noshock = \diffnoshockdownarg{\alpha}$,
	and the fact that 
	for $Z \in \Fullset$ we have $Z^{\alpha} = \smoothfunction(\BadVar)$,
	we deduce that 
	$\upmu \times \mbox{RHS~\eqref{E:CARTESIANDIFFERENTIATEDNONSHOCKEQUATION}} 
	= 
	\smoothfunction(\BadVar,Z \Psi) \diffnoshock
	$.
	Therefore, commuting $\upmu \times \mbox{\eqref{E:CARTESIANDIFFERENTIATEDNONSHOCKEQUATION}}$ with
	$\Tanset^{N-1}$, we obtain
	\begin{align} \label{E:TANGENTCOMMUTEDIDCARTESIANDIFFERENTATIEDNONSHOCKVARIABLEEQUATION}
		\upmu A^{\alpha} \partial_{\alpha} \Tanset^{N-1} V_{\alpha}
		& =
			[\smoothfunction(\GdVar) \Rad, \Tanset^{N-1}] V_{\alpha}
			+ 
			[\smoothfunction(\BadVar) \Singletan, \Tanset^{N-1}] V_{\alpha}
			+
			\Tanset^{N-1} 
			\left\lbrace
				\smoothfunction(\BadVar,Z \Psi) \diffnoshock
			\right\rbrace.
	\end{align}
	Using the bootstrap assumptions,
	we deduce that 
	$
	\left|\Tanset^{N-1} 
			\left\lbrace
				\smoothfunction(\BadVar,Z \Psi) \diffnoshock
			\right\rbrace
	\right|
	\lesssim \mbox{RHS~\eqref{E:NONSHOCKCARTESIANDIFFERENTIATEDVARIABLEPOINTWISEINHOM}}
	$
	as desired. To bound the commutator term
	$\left|
		[\smoothfunction(\BadVar) \Singletan, \Tanset^{N-1}] V_{\alpha}
	\right|$,
	we use the bootstrap assumptions and the commutator identity \eqref{E:TANGENTVECTORFIELDSCOMMUTEWITHEACHOTHER}.
	To bound the commutator term
	$\left|
		[\smoothfunction(\GdVar) \Rad, \Tanset^{N-1}] V_{\alpha}
	\right|$,
	we use the bootstrap assumptions,
	the commutator estimate \eqref{E:COMMUTATORESTIMATEFORONETRANSVERSAL},
	and the pointwise estimate \eqref{E:POINTWISEESTIMATESTRANSVERSALDERIVATIVESOFCARTESIANDIFFERENTIATEDNOSHOCKCONTROLLEDBYTANGENTAIL}.
	We have therefore proved \eqref{E:NONSHOCKCARTESIANDIFFERENTIATEDVARIABLEPOINTWISEINHOM}.
	The estimate \eqref{E:NONSHOCKVARIABLEPOINTWISEINHOM} can be proved in a similar fashion
	starting from equation \eqref{E:NONSHOCKEQUATION} and with the help of 
	\eqref{E:POINTWISEESTIMATESTRANSVERSALDERIVATIVESOFNOSHOCKCONTROLLEDBYTANGENTAIL}; we omit the details.
	
	\medskip
	
	\noindent \underline{\textbf{Proof of \eqref{E:LUNITMUNOTCOMMUTED}}}:
	A special case of \eqref{E:LINFTYUPMUCOMMUTED} is the estimate 
	$\Lunit \Lunit \upmu(t,u,\vartheta) = \mathcal{O}(\varepsilon)$.
	From this bound and the fundamental theorem of calculus, we deduce
	$
	\Lunit \upmu(t,u,\vartheta) 
	= \Lunit \upmu(0,u,\vartheta)
	+ \mathcal{O}(\varepsilon)
	$.
	Next, we use the identity
	$
	(\Rad \Lunit^a) \upxi_a
	=
	- (\Rad \Lunit^1) \upxi_1
	+
	\sum_{a=2}^n (\Rad \Lunit^a) \upxi_a^{(Small)}
	$,
	definition \eqref{E:BLOWUPCONSTANT},
	and the conditions on the data
	to decompose \eqref{E:LUNITUPMU} at time $0$ as
	\begin{align} \label{E:FIRSTSTEPLUNITMUNOTCOMMUTED}
		\Lunit \upmu(0,u,\vartheta)
		& = - [\blowupcoeff \Rad \Psi](0,u,\vartheta)
				+
				\mathcal{O}(\varepsilon).
	\end{align}
	We next note that fundamental theorem of calculus yields 
	\begin{align} \label{E:FTCID}
		[\blowupcoeff \Rad \Psi](t,u,\vartheta)
		& = [\blowupcoeff \Rad \Psi](0,u,\vartheta)
			+
			\int_{s=0}^t
				\Lunit [\blowupcoeff \Rad \Psi](s,u,\vartheta)
			\, ds.
	\end{align}
	Since the estimates \eqref{E:LINFTYPSICOMMUTED} and \eqref{E:LINFTYUPXIJCOMMUTED} 
	and the bootstrap assumptions
	imply that
	$\Lunit [\blowupcoeff \Rad \Psi] = \mathcal{O}(\varepsilon)$,
	we find from \eqref{E:FTCID} that
	$
	[\blowupcoeff \Rad \Psi](t,u,\vartheta) 
	= 
	[\blowupcoeff \Rad \Psi](0,u,\vartheta) 
	+ 
	\mathcal{O}(\varepsilon)
	$.
	Combining these estimates,
	we arrive at both of the bounds stated in \eqref{E:LUNITMUNOTCOMMUTED}.
	
	\medskip
	
	\noindent \underline{\textbf{Proof of \eqref{E:MUNOTCOMMUTED}}}:
	Using the fundamental theorem of calculus (as in \eqref{E:FTCID})
	and the initial condition $\upmu|_{\Sigma_0} = 1 + \mathcal{O}_{\star}(\Psiep) + \mathcal{O}(\varepsilon)$, 
	which follows from \eqref{E:MUINITIALCONDITION} and the conditions on the data,
	we obtain 
	$\upmu(t,u,\vartheta) 
	= 
	1 
	+ \mathcal{O}_{\star}(\Psiep)
	+
	\mathcal{O}(\varepsilon)
	+ 
	\int_{s=0}^t
		\Lunit \upmu(s,u,\vartheta)
	\, ds
	$.
	Substituting RHS~\eqref{E:LUNITMUNOTCOMMUTED} (evaluated at $(s,u,\vartheta)$)
	for the integrand $\Lunit \upmu(s,u,\vartheta)$,
	we arrive at the first estimate stated 
	in \eqref{E:MUNOTCOMMUTED}. To obtain the second estimate stated in \eqref{E:MUNOTCOMMUTED},
	we use the first estimate and the bound
	$
	[\blowupcoeff \Rad \Psi](t,u,\vartheta) 
	= 
	[\blowupcoeff \Rad \Psi](0,u,\vartheta) 
	+ 
	\mathcal{O}(\varepsilon)
	$
	noted in the previous paragraph.
	
	\medskip
	
	\noindent \underline{\textbf{Proof of \eqref{E:UPMULINFINTYIMPROVED} and \eqref{E:LUPMULINFINTYIMPROVED}}}:
	\eqref{E:UPMULINFINTYIMPROVED} follows easily from \eqref{E:MUNOTCOMMUTED}
	and the fact that $0 < t < 2 \TranminusdatasizeWithFactor^{-1}$.
	Similarly, \eqref{E:LUPMULINFINTYIMPROVED} follows easily from \eqref{E:LUNITMUNOTCOMMUTED}.
	
	\medskip
	
	\noindent \underline{\textbf{Proof of \eqref{E:LUNITBIGXIJPOINTWISE}-\eqref{E:BIGXIJLINFITY}}}:
	Using \eqref{E:BIGXIJEVOLUTIONEQUATION} and \eqref{E:TANGENTVECTORFIELDSCOMMUTEWITHEACHOTHER},
	we deduce the following schematic identity:
	$\Lunit \Tanset^{N-1} \Xi^j 
	=
	\Tanset^{N-1}
	\left\lbrace
		\Xi^a 
		\left[
			\smoothfunction(\GdVar) \diffnoshock
			+
			\smoothfunction(\GdVar) \Singletan \Psi
		\right]
	\right\rbrace
	+
	\Tanset^{N-1}
	\left\lbrace
		\smoothfunction(\BadVar,Z \Psi)
	\right\rbrace
	$.
	From this identity and the bootstrap assumptions, we deduce
	\begin{align} \label{E:ALMOSTLUNITBIGXIJPOINTWISE}
		\max_{1 \leq j \leq n}
		\left|
			\Lunit \Tanset^{\leq N-1} \Xi^j
		\right|
		& \lesssim
			\max_{1 \leq j \leq n}
			\left|
				\Tanset^{\leq N-1} \Xi^j
			\right|
				\\
	&  \ \
			+
			\max_{1 \leq j \leq n}
			\left|
				\Tanset^{\leq \lfloor (N-1)/2 \rfloor} \Xi^j
			\right|
			\left\lbrace
				\left|
					\Fullset_*^{\leq N;1} \Psi
				\right|
				+
				\left|
					\Tanset_*^{[1,N-1]} \BadVar
				\right|
				+
				1
			\right\rbrace
			\notag	\\
		& \ \
			+
			\left|
					\Fullset_*^{\leq N;1} \Psi
				\right|
				+
				\left|
					\Tanset_*^{[1,N-1]} \BadVar
				\right|
				+
				1.
			\notag
	\end{align}
	In particular, from \eqref{E:ALMOSTLUNITBIGXIJPOINTWISE}
	and the bootstrap assumptions, we deduce
	\begin{align} \label{E:LOWDERIVATIVESALMOSTLUNITBIGXIJPOINTWISE}
		\max_{1 \leq j \leq n}
		\left|
			\Lunit \Tanset^{\leq \Nmid-1} \Xi^j
		\right|
		& \lesssim
			\max_{1 \leq j \leq n}
			\left|
				\Tanset^{\leq \Nmid-1} \Xi^j
			\right|
			+
			1.
	\end{align}
	Moreover, from the conditions on the data, we deduce that
	$
	\max_{1 \leq j \leq n}
	\left|
		\Tanset^{\leq \Nmid-1} \Xi^j
	\right|(0,u,\vartheta)
	\lesssim 1
	$.
	Recalling that $\Lunit = \frac{\partial}{\partial t}$,
	we now use this data bound, 
	\eqref{E:LOWDERIVATIVESALMOSTLUNITBIGXIJPOINTWISE},
	and Gronwall's inequality in 
	$
	\max_{1 \leq j \leq n}
		\left|
			\Tanset^{\leq \Nmid-1} \Xi^j
		\right|
	$ to deduce
	that
	$
		\max_{1 \leq j \leq n}
		\left\|
			\Tanset^{\leq \Nmid-1} \Xi^j
		\right\|_{L^{\infty}(\Sigma_t^u)}
		\lesssim 1
	$,
	which is the desired bound \eqref{E:BIGXIJLINFITY}.
	Finally, from \eqref{E:ALMOSTLUNITBIGXIJPOINTWISE} and 
	\eqref{E:BIGXIJLINFITY},
	we conclude \eqref{E:LUNITBIGXIJPOINTWISE}. 
	
\end{proof}

\subsection{Estimates closely tied to the formation of the shock}
\label{SS:ESTIMATESTIEDTOSHOCKFORMATION}
In this subsection, we prove a lemma that lies at the heart of
showing that $\upmu$ vanishes in finite time and that 
its vanishing coincides with the blowup of
$\max_{\alpha=0,\cdots,n} |\partial_{\alpha} \Psi|$.
Roughly, the lemma shows that when $\upmu$ is small,
$\Rad \Psi$ must be quantitatively large in magnitude
and that $\Rad \Psi$ has a sign that forces $\upmu$
to continue shrinking
(the latter fact is important in that 
$\Rad \Psi$ is the dominant term in the evolution equation \eqref{E:LUNITUPMU} for $\upmu$).

We start by defining a quantity that captures the 
``worst-case'' behavior of $\upmu$ along $\Sigma_t^u$.

\begin{definition} \label{D:MUSTARDEF}
	We define the following quantity, where $\upmu$
	is the inverse foliation density from Def.\ \ref{D:MUDEF}:
	\begin{align} \label{E:MUSTARDEF}
		\upmu_{\star}(t,u)
		& := \min_{\Sigma_t^u} \upmu.
	\end{align}
\end{definition}

We now prove the main result of this subsection.

\begin{lemma}[$|\Rad \Psi|$ \textbf{ is large when } $\upmu$ \textbf{is small}]
	\label{L:CRUCIALESTIMATESFORUPMUANDRADPSI}
	The following implication holds:
	\begin{align}	\label{E:CRUCIALESTIMATESFORUPMUANDRADPSI}
		\upmu(t,u,\vartheta)
		< \frac{1}{4}
		\implies
		[\blowupcoeff \Rad \Psi](t,u,\vartheta)
		<
		-
		\frac{1}{4}
		\TranminusdatasizeWithFactor,
	\end{align}
	where the blowup-coefficient $\blowupcoeff \neq 0$ 
	(see Remark~\ref{R:BLOWUPCOEFFICIENTISNONZERO})
	is defined in Def.\ \ref{D:GENIUNELYNONLINEARCONSTANT}
	and the data-size parameter $\TranminusdatasizeWithFactor$
	is defined in Def.\ \ref{D:SHOCKFORMATIONQUANTITY}.
	
	In addition,
	\begin{align}	\label{E:ABSOLUTEVALUECRUCIALESTIMATESFORUPMUANDRADPSI}
		\upmu(t,u,\vartheta)
		< \frac{1}{4}
		\implies
		|\Radunit \Psi|(t,u,\vartheta)
		>
		\frac{1}{8 |\widetilde{\blowupcoeff}|}
		\frac{1}{\upmu(t,u,\vartheta)}
		\TranminusdatasizeWithFactor,
	\end{align}
	where the constant $\widetilde{\blowupcoeff} := \blowupcoeff|_{(\Psi,\noshock) = (0,0)}$
	is the blowup-coefficient evaluated at the background value of $(\Psi,\noshock) = (0,0)$
	(this makes sense in view of Remark~\ref{R:FUNCTIONALDEPENDENCEOFBLOWUPCOEFFICIENT}).
	
	Finally, when $U_0 = 1$,
	the quantity $\upmu_{\star}$ defined in \eqref{E:MUSTARDEF} verifies the following estimate:
	\begin{align} \label{E:MUSTARKEYESTIMATE}
			\upmu_{\star}(t,1)
			& = 1 - t \TranminusdatasizeWithFactor + \mathcal{O}_{\star}(\Psiep) + \mathcal{O}(\varepsilon).
	\end{align}
\end{lemma}

\begin{proof}
	From the second estimate stated in \eqref{E:MUNOTCOMMUTED},
	we deduce that if $\upmu < 1/4$, then
	$
	\displaystyle
	t [\blowupcoeff \Rad \Psi](t,u,\vartheta) 
	= - \frac{3}{4}
	+ \mathcal{O}_{\star}(\Psiep)
	+ 
	\mathcal{O}(\varepsilon)
	$.
	From this bound and the fact that
	$0 \leq t < \Tboot < 2 \TranminusdatasizeWithFactor^{-1}$,
	we conclude \eqref{E:CRUCIALESTIMATESFORUPMUANDRADPSI}.
	
	To prove \eqref{E:ABSOLUTEVALUECRUCIALESTIMATESFORUPMUANDRADPSI}, we
	first use the fundamental theorem of calculus to deduce
	\begin{align} \label{E:ANOTHERFTCID}
		\blowupcoeff(t,u,\vartheta)
		& = \blowupcoeff(0,u,\vartheta)
			+
			\int_{s=0}^t
				\Lunit \blowupcoeff (s,u,\vartheta)
			\, ds.
	\end{align}
	Since the estimates \eqref{E:LINFTYPSICOMMUTED} and \eqref{E:LINFTYUPXIJCOMMUTED} 
	and the bootstrap assumptions
	imply that
	$\Lunit \blowupcoeff = \mathcal{O}(\varepsilon)$,
	we find from \eqref{E:ANOTHERFTCID} that
	$
	\blowupcoeff(t,u,\vartheta) 
	= 
	\blowupcoeff(0,u,\vartheta) 
	+ 
	\mathcal{O}(\varepsilon)
	$.
	Moreover, in view of Remark~\ref{R:FUNCTIONALDEPENDENCEOFBLOWUPCOEFFICIENT}
	and our data assumptions \eqref{E:PSIITSELFLINFTYSMALLDATAALONGSIGMA0}-\eqref{E:LINFTYSMALLDATAALONGSIGMA0},
	we have, by Taylor expanding, the following estimate:
	$\blowupcoeff(0,u,\vartheta) 
	:= 
	\blowupcoeff|_{(\Psi(0,u,\vartheta),\noshock(0,u,\vartheta))}
	=
	\widetilde{\blowupcoeff}
	+ \mathcal{O}_{\star}(\Psiep)
	+ \mathcal{O}(\varepsilon)
	$.	
	It follows that
	$
	\blowupcoeff(t,u,\vartheta) 
	= 
	\widetilde{\blowupcoeff}
	+ \mathcal{O}_{\star}(\Psiep)
	+ \mathcal{O}(\varepsilon)
	$.
	Using this estimate to substitute for the factor $\blowupcoeff(t,u,\vartheta)$
	in the second inequality in \eqref{E:CRUCIALESTIMATESFORUPMUANDRADPSI}
	and then taking the absolute value of the resulting inequality,
	we deduce that if $\upmu(t,u,\vartheta) < \frac{1}{4}$, then
	$
	|\Rad \Psi|(t,u,\vartheta)
	>
	\frac{1}{4 \left\lbrace |\widetilde{\blowupcoeff}| 
	+ \mathcal{O}_{\star}(\Psiep)
	+ \mathcal{O}(\varepsilon) \right\rbrace}
	\TranminusdatasizeWithFactor
	$.
	Dividing both sides of this inequality by $\upmu(t,u,\vartheta)$
	and appealing to \eqref{E:RADJ},
	we arrive at \eqref{E:ABSOLUTEVALUECRUCIALESTIMATESFORUPMUANDRADPSI}.
	
	To prove \eqref{E:MUSTARKEYESTIMATE}, we use the first line of \eqref{E:MUNOTCOMMUTED} to deduce 
	$
	\upmu(t,u,\vartheta)
	= 1 
		+ t [\blowupcoeff \Rad \Psi](0,u,\vartheta)
		+ \mathcal{O}_{\star}(\Psiep)
		+ \mathcal{O}(\varepsilon) 
	$.
	Taking the min of both sides of this estimate over $(u,\vartheta) \in [0,1] \times \mathbb{T}^{n-1}$
	and appealing to 
	Def.\ \ref{D:SHOCKFORMATIONQUANTITY} and
	Def.\ \ref{D:MUSTARDEF}, we conclude \eqref{E:MUSTARKEYESTIMATE}.
\end{proof}

\section{Estimates for the change of variables map}
\label{S:ESTIMATESFORCHOV}
In this section, we derive estimates for
the change of variables map $\Upsilon$ from Def.\ \ref{D:CHOV}.
The main result is Prop.\ \ref{P:CHOVREMAINSADIFFEOMORPHISM}, 
which will serve as a technical ingredient in our proof
that the solution exists up until the first shock. Roughly, the proposition
shows that if $\upmu$ remains bounded from below strictly away from $0$,
then $\Upsilon$ can be extended to a
diffeomorphism on the closure of the bootstrap domain.

\subsection{Control of the components of the change of variables map}
\label{SS:CHOVTECHNICALPRELIMINARY}
In this subsection, we provide two preliminary lemmas
that yield estimates for the components of $\Upsilon$.

\begin{lemma}[\textbf{Bounds for geometric coordinate partial derivatives of functions in terms of geometric vectorfield derivatives}]
	\label{L:POINTWISEGEOMETRICPARTIALDERIVATIVESINTERMSOFGEOMETRICVECTORFIELDS}
	For $K \in \lbrace 0, 1 \rbrace$, the following estimate holds for scalar functions $f$:
	\begin{align} \label{E:POINTWISEGEOMETRICPARTIALDERIVATIVESINTERMSOFGEOMETRICVECTORFIELDS}
		\sum_{i_0 + \cdots i_n \leq 1}
		\left\|
			\left(\frac{\partial}{\partial t}\right)^{i_0 + K}
			\left(\frac{\partial}{\partial u}\right)^{i_1}
			\left(\frac{\partial}{\partial \vartheta^2}\right)^{i_2}
			\cdots
			\left(\frac{\partial}{\partial \vartheta^n}\right)^{i_n}
			f
		\right\|_{L^{\infty}(\Sigma_t^u)}
		& \lesssim 
			\left\|
				\Fullset^{\leq 1+K;1} f
			\right\|_{L^{\infty}(\Sigma_t^u)}.
	\end{align}
\end{lemma}

\begin{proof}
	From \eqref{E:RADDECOMPINTOPARTIALUPLUSANGULAR} 
	and \eqref{E:PARTIALJINTERMSOFGEO},
	the fact that $\Xi$ is $\mathcal{T}_{t,u}$-tangent,
	and \eqref{E:TORUSTANGENTVECTORFIELDONEFORMCONTRACTIONS},
	we deduce the identity
	$
	\displaystyle
	\frac{\partial}{\partial u}
	= \Rad
		+ \Xi^a \partial_a
	= \Rad
		+
		\sum_{i=2}^n
		\Xi^a \smoothfunction_{ia}(\GdVar)
		\CoordAng{i}
	$.
	From this identity
	and the $L^{\infty}$ estimates of Prop.\ \ref{P:POINTWISEANDIMPROVEMENTOFAUX}
	(in particular the estimate \eqref{E:BIGXIJLINFITY}),
	it follows that
	$
	\displaystyle
	\frac{\partial}{\partial u}
	$
	is a linear combination of the elements of
	$\Fullset$ with coefficients that are bounded 
	in the norm $\| \cdot \|_{L^{\infty}(\Sigma_t^u)}$ by $\lesssim 1$.
	The estimate \eqref{E:POINTWISEGEOMETRICPARTIALDERIVATIVESINTERMSOFGEOMETRICVECTORFIELDS}
	is a straightforward consequence of this fact
	and the facts that
	$
	\displaystyle
	\Lunit =
	\frac{\partial}{\partial t}
	\in \Fullset
	$
	and
	$
	\displaystyle
	\CoordAng{i} =
	\frac{\partial}{\partial \vartheta^i}
	\in \Fullset
	$.
\end{proof}

We now show that
$\Upsilon$ can be extended to a function defined
on the closure of the bootstrap domain
that belongs to several function spaces.

\begin{lemma}[\textbf{A preliminary extension result for the change of variables map}]
	\label{L:CHOVMAPREGULARITY}
	The components $\Upsilon^{\alpha}(t,u,\vartheta)$ of 
	the change of variables map from Def.\ \ref{D:CHOV}
	extend to the compact domain 
	$[0,\Tboot] \times [0,U_0] \times \mathbb{T}^{n-1}$
	with the following regularity, $(i=2,\cdots,n)$, $(\alpha = 0,\cdots,n)$:
	\[
	\Upsilon^{\alpha},
		\,
	\frac{\partial}{\partial \vartheta^i} \Upsilon^{\alpha}
	\in
	C\left([0,\Tboot],W^{1,\infty}([0,U_0] \times \mathbb{T}^{n-1}) \right)
	\cap
	C^1\left([0,\Tboot],L^{\infty}([0,U_0] \times \mathbb{T}^{n-1}) \right).
	\]
	
	Moreover, the following estimates\footnote{The $L^{\infty}$ estimate
	for the torus coordinates $x^i \in \mathbb{T}$ (where $i=2,\cdots,n$) stated in 
	\eqref{E:CHOVREGULARITYMAINLINFINITYESTIMATE} should be interpreted as the 
	statement that for each fixed $i \in \lbrace 2,\cdots,n \rbrace$
	and $(u,\vartheta) \in [0,U_0] \times \mathbb{T}^{n-1}$,
	the Euclidean distance traveled by the curves 
	$t \rightarrow x^i(t,u,\vartheta)$,
	$t \in [0,\Tboot]$,
	in the universal covering space
	$\mathbb{R}$ of $\mathbb{T}$ is uniformly bounded. \label{FN:LINFINITYNORMOFCOORDINATES}}
	hold for $(t,u) \in [0,\Tboot] \times [0,U_0]$,
	where $C = C(\Trandatasize)$:
	\begin{subequations}
	\begin{align} \label{E:CHOVREGULARITYMAINLINFINITYESTIMATE}
		\sum_{i_0 + \cdots i_n \leq 1}
		\left\|
			\left(\frac{\partial}{\partial t}\right)^{i_0}
			\left(\frac{\partial}{\partial u}\right)^{i_1}
			\left(\frac{\partial}{\partial \vartheta^2}\right)^{i_2}
			\cdots
			\left(\frac{\partial}{\partial \vartheta^n}\right)^{i_n}
			\Upsilon^{\alpha}
		\right\|_{L^{\infty}(\Sigma_t^u)}
		\leq C,
			\\
		\mathop{\sum_{i_0 + \cdots i_n \leq 2}}_{1 \leq i_2 + \cdots i_n}
		\left\|
			\left(\frac{\partial}{\partial t}\right)^{i_0}
			\left(\frac{\partial}{\partial u}\right)^{i_1}
			\left(\frac{\partial}{\partial \vartheta^2}\right)^{i_2}
			\cdots
			\left(\frac{\partial}{\partial \vartheta^n}\right)^{i_n}
			\Upsilon^{\alpha}
		\right\|_{L^{\infty}(\Sigma_t^u)}
		\leq C \varepsilon.
		\label{E:CHOVREGULARITYTORUSDERIVATIVESLINFINITYESTIMATE}
	\end{align}
	\end{subequations}
\end{lemma}

\begin{proof}
	We will show that the following estimates hold for $(t,u) \in [0,\Tboot) \times [0,U_0]$:
	\begin{align} \label{E:CHOVREGULARITYPRELIMINARYLINFINITYESTIMATE}
		\sum_{K=0}^1
		\sum_{i_0 + \cdots i_n \leq 1}
		\left\|
			\left(\frac{\partial}{\partial t}\right)^{i_0+K}
			\left(\frac{\partial}{\partial u}\right)^{i_1}
			\left(\frac{\partial}{\partial \vartheta^2}\right)^{i_2}
			\cdots
			\left(\frac{\partial}{\partial \vartheta^n}\right)^{i_n}
			\Upsilon^{\alpha}
		\right\|_{L^{\infty}(\Sigma_t^u)}
		\lesssim 1,
			\\
		\sum_{K=0}^1
		\mathop{\sum_{i_0 + \cdots i_n \leq 2}}_{1 \leq i_2 + \cdots i_n}
		\left\|
			\left(\frac{\partial}{\partial t}\right)^{i_0+K}
			\left(\frac{\partial}{\partial u}\right)^{i_1}
			\left(\frac{\partial}{\partial \vartheta^2}\right)^{i_2}
			\cdots
			\left(\frac{\partial}{\partial \vartheta^n}\right)^{i_n}
			\Upsilon^{\alpha}
		\right\|_{L^{\infty}(\Sigma_t^u)}
		\lesssim \varepsilon.
		\label{E:CHOVREGULARITYPRELIMINARYTORUSDERIVATIVESLINFINITYESTIMATE}
	\end{align}
	Since 
	$
	\Lunit = \frac{\partial}{\partial t}
	$
	relative to geometric coordinates,
	all results of the lemma then follow as straightforward consequences
	of \eqref{E:CHOVREGULARITYPRELIMINARYLINFINITYESTIMATE}-\eqref{E:CHOVREGULARITYPRELIMINARYTORUSDERIVATIVESLINFINITYESTIMATE},
	the fundamental theorem of calculus,
	and the completeness of the spaces 
	$W^{1,\infty}([0,U_0] \times \mathbb{T}^{n-1})$
	and
	$L^{\infty}([0,U_0] \times \mathbb{T}^{n-1})$.
	
	Using \eqref{E:POINTWISEGEOMETRICPARTIALDERIVATIVESINTERMSOFGEOMETRICVECTORFIELDS}, 
	we see that to establish
	\eqref{E:CHOVREGULARITYPRELIMINARYLINFINITYESTIMATE}, it suffices to show that
	\begin{align} \label{E:CHOVREGULARITYSECONDPRELIMINARYLINFINITYESTIMATE}
		\left\|
			\Fullset^{\leq 2;1} \Upsilon^{\alpha}
		\right\|_{L^{\infty}(\Sigma_t^u)}
		\lesssim 1.
	\end{align}
	To derive \eqref{E:CHOVREGULARITYSECONDPRELIMINARYLINFINITYESTIMATE}, we
	first clarify that $\Upsilon^{\alpha}$ can be identified with
	the Cartesian coordinate $x^{\alpha}$, viewed as a function of $(t,u,\vartheta^2,\cdots,\vartheta^n)$.
	To bound $x^{\alpha}$, we note that $\Lunit x^{\alpha} = \Lunit^{\alpha} = \smoothfunction(\Psi,\noshock)$.
	Hence, the bootstrap assumptions imply that
	$\left\|
		\Lunit x^{\alpha}
	\right\|_{L^{\infty}(\Sigma_t^u)}
	\lesssim 1
	$.
	From this estimate and the fundamental theorem of calculus (as in \eqref{E:FTCID}),
	we conclude (see Footnote~\ref{FN:LINFINITYNORMOFCOORDINATES}) that
	$\left\|
		x^{\alpha}
	\right\|_{L^{\infty}(\Sigma_t^u)}
	\lesssim 1
	$
	as desired.
	Next, we note that for $\Singletan \in \Tanset$,
	we have $\Singletan x^{\alpha} = \Singletan^{\alpha} = \smoothfunction(\GdVar)$
	and $\Rad x^{\alpha} = \Rad^{\alpha} = \smoothfunction(\BadVar)$.
	Hence, to complete the proof of \eqref{E:CHOVREGULARITYSECONDPRELIMINARYLINFINITYESTIMATE},
	we need only to show that
	$
	\left\|
		\Tanset^{\leq 1} \smoothfunction(\BadVar)
	\right\|_{L^{\infty}(\Sigma_t^u)}
	\lesssim 1
	$
	and
	$
	\left\|
		\Fullset^{\leq 1;1} \smoothfunction(\GdVar)
	\right\|_{L^{\infty}(\Sigma_t^u)}
	\lesssim 1
	$.
	These bounds are simple consequences of the bootstrap assumptions.
	We have therefore proved \eqref{E:CHOVREGULARITYPRELIMINARYLINFINITYESTIMATE}.
	The estimate \eqref{E:CHOVREGULARITYPRELIMINARYTORUSDERIVATIVESLINFINITYESTIMATE}
	can be proved using a similar argument and we omit the details.
\end{proof}

\subsection{The diffeomorphism properties of the change of variables map}
\label{SS:CHOVDIFFEOMORPHISM}
We now derive the main result of Sect.\ \ref{S:ESTIMATESFORCHOV}.

\begin{proposition}[\textbf{Sufficient conditions for $\Upsilon$ to be a global diffeomorphism}]
\label{P:CHOVREMAINSADIFFEOMORPHISM}
	If
	\begin{align} \label{E:MUUNIFORMLYBOUNDEDFROMBELOW}
		\inf_{(t,u) \in [0,\Tboot) \times [0,U_0]} \upmu_{\star}(t,u) > 0,
	\end{align}
	then the change of variables map $\Upsilon$ extends to a global diffeomorphism 
	from $[0,\Tboot] \times [0,U_0] \times \mathbb{T}^{n-1}$ onto its image
	with the following regularity, $(i=2,\cdots,n)$, $(\alpha = 0,\cdots,n)$:
	\begin{align} \label{E:CHOVGLOBALDIFFEOMORPHISMREGULARITY}
	\Upsilon^{\alpha},
		\,
	\CoordAng{i} \Upsilon^{\alpha}
	\in
	C\left([0,\Tboot],W^{1,\infty}([0,U_0] \times \mathbb{T}^{n-1}) \right)
	\cap
	C^1\left([0,\Tboot],L^{\infty}([0,U_0] \times \mathbb{T}^{n-1}) \right).
	\end{align}
\end{proposition}

\begin{proof}
	By the bootstrap assumption \eqref{E:BOOTSTRAPCHOVISDIFFEO}, 
	$\Upsilon$ is a diffeomorphism from 
	$[0,\Tboot) \times [0,U_0] \times \mathbb{T}^{n-1}$
	onto its image $\mathcal{M}_{\Tboot,U_0}$.
	In addition, Lemma~\ref{L:CHOVMAPREGULARITY} implies
	that each component $\Upsilon^{\alpha}$ extends to a function of 
	the geometric coordinates satisfying \eqref{E:CHOVGLOBALDIFFEOMORPHISMREGULARITY}.
	Next, we use \eqref{E:DETERMINANTOFCHOV},
	the $L^{\infty}$ estimates of Prop.\ \ref{P:POINTWISEANDIMPROVEMENTOFAUX},
	and the assumption \eqref{E:MUUNIFORMLYBOUNDEDFROMBELOW}
	to deduce that the Jacobian determinant of $\Upsilon$
	is uniformly bounded in magnitude from above and below away from
	$0$ on $[0,\Tboot] \times [0,U_0] \times \mathbb{T}^{n-1}$.
	Hence, from the inverse function theorem, we deduce that
	$\Upsilon$ extends as a local diffeomorphism from
	$[0,\Tboot] \times [0,U_0] \times \mathbb{T}^{n-1}$
	onto its image. 
	Therefore, to complete the proof of the lemma, 
	we need only to show that $\Upsilon$ is injective on the domain 
	$[0,\Tboot] \times [0,U_0] \times \mathbb{T}^{n-1}$.
	Since $\Upsilon$ is
	a diffeomorphism on the domain
	$[0,\Tboot) \times [0,U_0] \times \mathbb{T}^{n-1}$,
	it suffices to show that
	$\Upsilon(\Tboot,u_1,\vartheta_1) \neq \Upsilon(\Tboot,u_2,\vartheta_2)$
	whenever
	$(u_i,\vartheta_i) \in [0,U_0] \times \mathbb{T}^{n-1}$
	and
	$(u_1,\vartheta_1) \neq (u_2,\vartheta_2)$.
	
	We first show that if $u_1 \neq u_2$, then
	$\Upsilon(\Tboot,u_1,\vartheta_1) \neq \Upsilon(\Tboot,u_2,\vartheta_2)$.
	To this end, we observe that from definitions \eqref{E:EIKONALFUNCTIONSPATIALONEFORM} and \eqref{E:UPXIJSMALL},
	the estimates \eqref{E:UPMULINFINTYIMPROVED} and \eqref{E:LINFTYUPXIJITSELF},
	and the assumption \eqref{E:MUUNIFORMLYBOUNDEDFROMBELOW},
	it follows that $\sum_{a=1}^n|\partial_a u|$ is uniformly bounded from above and from below,
	strictly away from $0$. It follows that no two distinct (closed) characteristic hypersurface portions
	$\mathcal{P}_{u_1}^{\Tboot}$ and $\mathcal{P}_{u_2}^{\Tboot}$ can intersect,
	which yields the desired result.
	
	To finish the proof of the lemma, we must show that 
	$\Upsilon(\Tboot,u,\vartheta_1) \neq \Upsilon(\Tboot,u,\vartheta_2)$
	whenever 
	$u \in [0,U_0]$
	and
	$\vartheta_1 \neq \vartheta_2$.
	That is, we must show that for each fixed $u \in [0,U_0]$,
	the map $\upsilon$
	defined by $\upsilon(\vartheta) := \Upsilon(\Tboot,u,\vartheta)$ is an injection from $\mathbb{T}^{n-1}$ onto its image.
	To this end, for each fixed $u \in [0,U_0]$,
	we consider the family of $t$-parameterized maps $\widetilde{\upsilon}(t;\cdot)$ 
	(where $t \in [0,\Tboot]$)
	defined to be the last $n-1$ components of $\Upsilon(t,u,\cdot)$, 
	that is,
	$\widetilde{\upsilon}(t;\vartheta)
	:=
	\left(\Upsilon^2(t,u,\cdot),\Upsilon^3(t,u,\cdot),\cdots,\Upsilon^n(t,u,\cdot)\right)
	$
	(recall that $\Upsilon^i$ can be identified with the local Cartesian coordinate $x^i$).
	Note that $\widetilde{\upsilon}(t;\cdot)$ can be viewed as a map with the domain
	$\mathbb{T}^{n-1}$ (equipped with the geometric coordinates $(\vartheta^2,\cdots,\vartheta^n)$) 
	and the target $\mathbb{T}^{n-1}$ (equipped with the Cartesian coordinates $(x^2,\cdots,x^n)$).
	Since $\Upsilon$ is continuous on $[0,\Tboot] \times [0,U_0] \times \mathbb{T}^{n-1}$,
	it follows that $\upsilon$
	is homotopic to the degree-one\footnote{$\widetilde{\upsilon}(0,\cdot)$ 
	is degree-one because $x^i(0,u,\vartheta^2,\cdots,\vartheta^n) = \vartheta^i$ for $i=2,\cdots,n$ by construction.} 
	map 
	$\widetilde{\upsilon}(0,\cdot)$
	by the homotopy 
	$\widetilde{\upsilon}(t;\vartheta)$.
	Hence, it is a basic result of degree theory that
	$\widetilde{\upsilon}(t,\cdot)$ is also a degree-one map.
	In particular, $\widetilde{\upsilon}(\Tboot,\cdot)$ is degree-one.
	Next, we note that Lemma~\ref{L:CHOVMAPREGULARITY}
	implies that 	$\Upsilon^j(\Tboot,u,\cdot)$ can be viewed as a $C^1$ function
	of $(\vartheta^2,\cdots,\vartheta^n) \in \mathbb{T}^{n-1}$
	and that by
	\eqref{E:COORDANGJKSMALL} 
	and \eqref{E:CHOVREGULARITYTORUSDERIVATIVESLINFINITYESTIMATE},
	for $i,j=2,\cdots,n$, we have
	$\CoordAng{i} \Upsilon^j(\Tboot,u,\vartheta^2,\cdots,\vartheta^n) 
	= \delta^{ij} + \CoordAngSmallcomp{i}{j}(\Tboot,u,\vartheta^2,\cdots,\vartheta^n)
	= \delta^{ij} + \mathcal{O}(\varepsilon)
	$,
	where $\delta^{ij}$ is the standard Kronecker delta.
	From this estimate and the degree-one property of $\upsilon(\cdot) = \widetilde{\upsilon}(\Tboot,\cdot)$, 
	we deduce\footnote{Recall that if $f: \mathbb{T}^{n-1} \rightarrow \mathbb{T}^{n-1}$ 
	is a $C^1$ surjective map without critical points, then $f$ is degree-one if for $p,q \in \mathbb{T}^{n-1}$,
	$1 = \sum_{p \in f^{-1}(q)} \mbox{\upshape sign det} (d_p f)$,
	where $d_p f$ denotes the differential of $f$ at $p$
	and the $d_p f$ are computed relative to an atlas corresponding to the 
	smooth orientation on $\mathbb{T}^{n-1}$ chosen at the beginning of the article.
	It is a basic fact of degree theory that the sum is independent of $q$.
	Note that in the context of the present argument, 
	the components of the $(n-1) \times (n-1)$ matrix $d f(\cdot)$
	are $\CoordAng{i} \Upsilon^j(\Tboot,u,\cdot)$, $(i,j=2,3,\cdots,n)$.}
	that for sufficiently small $\varepsilon$,
	$\upsilon(\cdot)$ is a bijection\footnote{The surjective property of this map is easy to deduce.} 
	from $\mathbb{T}^{n-1}$ to $\mathbb{T}^{n-1}$.
	In particular, $\upsilon$ is injective, 
	which is the desired result.
	
\end{proof}

\section{Energy estimates and strict improvements of the fundamental bootstrap assumptions}
\label{S:ENERGYESTIMATES}
In this section, we derive the main estimates of the paper: a priori energy estimates that hold up to top order
on the bootstrap region. The main ingredients in the proofs 
are the energy identities of Sect.\ \ref{S:ENERGYID}
and the pointwise estimates of Prop.\ \ref{P:POINTWISEANDIMPROVEMENTOFAUX}. 
As a corollary, we also derive strict improvements of the fundamental $L^{\infty}$ bootstrap assumptions
of Subsubsect.\ \ref{SSS:FUNDAMENTALBOOT}.

\subsection{Definition of the fundamental \texorpdfstring{$L^2$-}{square integral-}controlling quantity}
\label{SS:DEFSOFL2CONTROLLINGQUANTITIES}
We start by defining the coercive quantity that we used to control the
solution in $L^2$ up to top order.

\begin{definition}[\textbf{The main coercive} $L^2$-\textbf{controlling quantity}]
\label{D:MAINCOERCIVEQUANT}
In terms of the energy-characteristic flux quantities of Def.\ \ref{D:ENERGYFLUX}
and the multi-index set $\mathcal{I}_*^{[1,\Ntop];1}$
of Def.\ \ref{D:COMMUTATORMULTIINDICES}, 
we define
\begin{align}
	&
	\totmax(t,u)
		:= 
		\max
		\Big\lbrace
		\max_{\vec{I} \in \mathcal{I}_*^{[1,\Ntop];1}}
		\sup_{(t',u') \in [0,t] \times [0,u]} 
		\shocken[\Fullset^{\vec{I}} \Psi](t',u'),
			\\
	& \ \ 
		\mathop{\max_{|\vec{I}| \leq \Ntop-1}}_{f \in 
		\lbrace \noshockuparg{J} \rbrace_{1 \leq J \leq M} \cup 
			\lbrace \diffnoshockdoublearg{\alpha}{J} \rbrace_{0 \leq \alpha \leq n; 1 \leq J \leq M}}
		\sup_{(t',u') \in [0,t] \times [0,u]} 
		\left\lbrace
			\noshocken[\Tanset^{\vec{I}} f](t',u')
			+ 
			\noshockfl[\Tanset^{\vec{I}} f](t',u')
		\right\rbrace
		\Big\rbrace.
		\notag
	\end{align}
\end{definition}

\subsection{Coerciveness of the fundamental \texorpdfstring{$L^2$-}{square integral-}controlling quantity}
\label{SS:COERCIVENESSOFTHEFUNDAMENTALCONTROLLING}
In the next lemma, we exhibit the coerciveness properties of $\totmax(t,u)$.

\begin{lemma}[\textbf{Coerciveness of} $\totmax(t,u)$]
\label{L:COERCIVENESSOFFUNDAMENTALCONTROLLING}
The following estimates hold:
\begin{align}
	\sup_{(t',u') \in [0,t] \times [0,u]}
	\left\|
		\Fullset_*^{[1,\Ntop];1} \Psi
	\right\|_{L^2(\Sigma_{t'}^{u'})}
	& \leq \totmax^{1/2}(t,u),
		\label{E:PSIINTERMSOFFUNDAMENTALCONTROLLING}
\end{align}

\begin{subequations}
\begin{align}
\sup_{(t',u') \in [0,t] \times [0,u]}
\left\|
	\sqrt{\upmu} \Tanset^{\leq \Ntop-1} \noshock
\right\|_{L^2(\Sigma_{t'}^{u'})}
	& \leq C \totmax^{1/2}(t,u),
		\label{E:SIGMATNOSHOCKINTERMSOFFUNDAMENTALCONTROLLING}
		\\
	\sup_{(t',u') \in [0,t] \times [0,u]}
	\left\|
		\sqrt{\upmu} \Tanset^{\leq \Ntop-1} \diffnoshock
	\right\|_{L^2(\Sigma_{t'}^{u'})}
	& \leq C \totmax^{1/2}(t,u),
		\label{E:SIGMATDIFFNOSHOCKINTERMSOFFUNDAMENTALCONTROLLING}
\end{align}
\end{subequations}

\begin{subequations}
\begin{align}
\sup_{(t',u') \in [0,t] \times [0,u]}
\left\|
	\Tanset^{\leq \Ntop-1} \noshock
\right\|_{L^2(\mathcal{P}_{u'}^{t'})}
	& \leq C \totmax^{1/2}(t,u),
		\label{E:CHARACTERISTICSNOSHOCKINTERMSOFFUNDAMENTALCONTROLLING}
		\\
	\sup_{(t',u') \in [0,t] \times [0,u]}
	\left\|
		\Tanset^{\leq \Ntop-1} \diffnoshock
	\right\|_{L^2(\mathcal{P}_{u'}^{t'})}
	& \leq C \totmax^{1/2}(t,u).
		\label{E:CHARACTERISTICSDIFFNOSHOCKINTERMSOFFUNDAMENTALCONTROLLING}
\end{align}
\end{subequations}
\end{lemma}

\begin{proof}
	Lemma~\ref{L:COERCIVENESSOFFUNDAMENTALCONTROLLING} follows from
	Def.\ \ref{D:MAINCOERCIVEQUANT},
	Def.\ \ref{E:SHOCKENERGY},
	Lemma~\ref{L:COERCIVENESSOFNOSHOCKENERGY},
	and the $L^{\infty}$ estimates of Prop.\ \ref{P:POINTWISEANDIMPROVEMENTOFAUX}
	(which provide the smallness of $\GdVar$ that assumed, for example, 
	in the hypotheses of Lemma~\ref{L:COERCIVENESSOFNOSHOCKENERGY}).
\end{proof}

\subsection{Sobolev embedding}
\label{SS:SOBELEVEMBEDDING}
The main result of this subsection is Lemma~\ref{L:BOOTSTRAPNORMSINTERMSOFCONTROLLINGQUANTITY}, 
a Sobolev embedding result which shows that the norm $\| \cdot \|_{L^{\infty}(\Sigma_t^u)}$ of $\noshock$ and $\diffnoshock$
and their $\mathcal{P}_u$-tangential 
derivatives up to mid-order is controlled by $\totmax$. 
In Cor.\ \ref{C:IMPROVEMENTOFFUNDAMANETALBOOTSTRAPASSUMPTIONS}, 
we will use the lemma as an ingredient 
in our derivation of strict improvements of the fundamental $L^{\infty}$ bootstrap assumptions.
As a preliminary step, we provide the following lemma, 
in which we we derive some $L^2$ 
estimates for $\noshock$, $\diffnoshock$, and their derivatives along 
the co-dimension two tori $\mathcal{T}_{t,u}$.

\begin{lemma}[$L^2$ \textbf{control of the non-shock-forming variables on} $\mathcal{T}_{t,u}$]
	\label{L:CONTROLOFTORUSINTEGRALSINTERMSOFFUNDAMENTALCONTROLLINGQUANTITY}
	The following estimates hold for $0 \leq \alpha \leq n$ and $1 \leq J \leq M$:
	\begin{align} \label{E:CONTROLOFTORUSINTEGRALSINTERMSOFFUNDAMENTALCONTROLLINGQUANTITY}
			\left\|
				\Tanset^{\leq \Ntop-2} \noshockuparg{J}
			\right\|_{L^2(\mathcal{T}_{t,u})},
				\,
			\left\|
				\Tanset^{\leq \Ntop-2} \diffnoshockdoublearg{\alpha}{J}
			\right\|_{L^2(\mathcal{T}_{t,u})}
			& \leq C \mathring{\upepsilon}
				+
				C \totmax^{1/2}(t,u).
		\end{align}
	
\end{lemma}

\begin{proof}
	We first note the following estimate for scalar functions $f$,
	which follows from differentiating under the integral
	and using Young's inequality:
	\begin{align} \label{E:TIMEDERIVATIVEOFTORUSINTEGRAL}
		\frac{\partial}{\partial t}
		\| f \|_{L^2(\mathcal{T}_{t,u})}^2
		& = 
		2
		\int_{\mathcal{T}_{t,u}}
			f \Lunit f
		\, d \torusvol 
		\leq
		\| f \|_{L^2(\mathcal{T}_{t,u})}^2
		+
		\| \Lunit f \|_{L^2(\mathcal{T}_{t,u})}^2.
	\end{align}
	Integrating \eqref{E:TIMEDERIVATIVEOFTORUSINTEGRAL} from time $0$ to time $t$, 
	we find that
	\begin{align} \label{E:TORUSINTEGRALGRONWALLREADY}
		\| f \|_{L^2(\mathcal{T}_{t,u})}^2
		& \leq 
		\| f \|_{L^2(\mathcal{T}_{0,u})}^2
		+
		\int_{s=0}^t
			\| f \|_{L^2(\mathcal{T}_{s,u})}^2
		\, ds
		+ 
		\| \Lunit f \|_{L^2(\mathcal{P}_u^t)}^2.
	\end{align}
	From \eqref{E:TORUSINTEGRALGRONWALLREADY} and Gronwall's inequality,
	we deduce that
	\begin{align} \label{E:TORUSINTEGRALGRONWALLED}
		\| f \|_{L^2(\mathcal{T}_{t,u})}^2
		& \leq 
		C
		\| f \|_{L^2(\mathcal{T}_{0,u})}^2
		+
		C
		\| \Lunit f \|_{L^2(\mathcal{P}_u^t)}^2.
	\end{align}
	We now apply \eqref{E:TORUSINTEGRALGRONWALLED}
	with the role of $f$ played by
	$\Tanset^{\leq \Ntop-2} \noshockuparg{J}$
	and $\Tanset^{\leq \Ntop-2}  \diffnoshockdoublearg{\alpha}{J}$.
	In view of the data-size assumptions \eqref{E:L2SMALLDATAALONGELL0U}
	and the bounds 
	$
	\| \Lunit \Tanset^{\leq \Ntop-2} \noshock \|_{L^2(\mathcal{P}_u^t)}^2 
	\lesssim \totmax(t,u)
	$
	and
	$
	\| \Lunit \Tanset^{\leq \Ntop-2} \diffnoshock \|_{L^2(\mathcal{P}_u^t)}^2 
	\lesssim \totmax(t,u)
	$,
	which follow from 
	\eqref{E:CHARACTERISTICSNOSHOCKINTERMSOFFUNDAMENTALCONTROLLING}-\eqref{E:CHARACTERISTICSDIFFNOSHOCKINTERMSOFFUNDAMENTALCONTROLLING},
	we arrive at the desired estimate \eqref{E:CONTROLOFTORUSINTEGRALSINTERMSOFFUNDAMENTALCONTROLLINGQUANTITY}.
\end{proof}

We now prove the main result of this subsection.

\begin{lemma}[$L^{\infty}$ \textbf{control of the non-shock-forming variables up to mid-order in terms of} $\totmax$]
	\label{L:BOOTSTRAPNORMSINTERMSOFCONTROLLINGQUANTITY}
	The following estimates hold:
	\begin{align} \label{E:BOOTSTRAPNORMSINTERMSOFCONTROLLINGQUANTITY}
		\left\|
			\Tanset^{\leq \Nmid-1} v
		\right\|_{L^{\infty}(\Sigma_t^u)},
		\,
		\left\|
			\Tanset^{\leq \Nmid-1} V_{\alpha}
		\right\|_{L^{\infty}(\Sigma_t^u)}
		& \leq C \mathring{\upepsilon} 
			+ 
			C \totmax^{1/2}(t,u).
	\end{align}
\end{lemma}

\begin{proof}
	Standard Sobolev embedding on $\mathbb{T}^{n-1}$
	yields the following estimate for scalar functions $f$:
	\begin{align} \label{E:BASICSOBOLEV}
		\left\|
			f
		\right\|_{L^{\infty}(\mathcal{T}_{t,u})}
		& \lesssim
		\left\|
			f
		\right\|_{L^2(\mathcal{T}_{t,u})}
		+
		\sum_{K=1}^{\lfloor \frac{n+1}{2} \rfloor}
		\sum_{Y_{(1)},\cdots, Y_{(K)} \in \lbrace \CoordAng{i} \rbrace_{i=2,3,\cdots,n}}
		\left\|
			Y_{(1)} \cdots Y_{(K)} f
		\right\|_{L^2(\mathcal{T}_{t,u})}.
	\end{align}
	The desired estimate \eqref{E:BOOTSTRAPNORMSINTERMSOFCONTROLLINGQUANTITY}
	now follows from
	\eqref{E:BASICSOBOLEV},
	\eqref{E:CONTROLOFTORUSINTEGRALSINTERMSOFFUNDAMENTALCONTROLLINGQUANTITY},
	and \eqref{E:NDEFS},
	where the last of these equations in particular implies that
	$
	\displaystyle
	\Nmid-1
	+
	\left\lfloor \frac{n+1}{2} \right\rfloor
	\leq
	\Ntop-2
	$.
\end{proof}

\subsection{Preliminary $L^2$ estimates for 
$\upmu$,
$\upxi_j^{(Small)}$, 
and
$\CoordAngSmallcomp{i}{j}$}
In the next lemma, we bound the $L^2$ norms of the derivatives of the quantities
$\upmu$, $\xi_j^{(Small)}$, and $\CoordAngSmallcomp{i}{j}$
in terms of $\totmax$. This serves as a preliminary step for our forthcoming derivation of 
$L^2$ estimates for $\Psi$, $\noshock$, and $\diffnoshock$,
since $\upmu$, $\xi_j^{(Small)}$, and $\CoordAngSmallcomp{i}{j}$
appear as source terms in their commuted evolution equations
(as is shown by RHSs~\eqref{E:SHOCKVARIABLEPOINTWISEINHOM}-\eqref{E:NONSHOCKCARTESIANDIFFERENTIATEDVARIABLEPOINTWISEINHOM}).

\begin{lemma}[$L^2$ \textbf{estimates for} 
$\upmu$,
$\upxi_j^{(Small)}$, 
\textbf{and}
$\CoordAngSmallcomp{i}{j}$
\textbf{in terms of} $\totmax$]
	\label{L:L2ESTIMATESFOREIKONALFUNCTIONQUANTITIES}
	The following estimates hold for 
	$2 \leq i \leq n$,
	$1 \leq j \leq n$,
	and
	$(t,u) \in [0,\Tboot) \times [0,U_0]$,
	where $\totmax$ is defined in Def.\ \ref{D:MAINCOERCIVEQUANT}:
	\begin{subequations}
		\begin{align}
			\left\|
				\Tanset_*^{[1,\Ntop-1]} \upmu
			\right\|_{L^2(\Sigma_t^u)}
			& \leq C \mathring{\upepsilon}
				+
				C \totmax^{1/2}(t,u),
					\label{E:L2UPMUDERIVATIVESBOUNDEDINTERMSOFENERGIES} \\
			\left\|
				\Fullset_*^{[1,\Ntop-1];1} \xi_j^{(Small)}
			\right\|_{L^2(\Sigma_t^u)}
			& \leq C \mathring{\upepsilon}
				+
				C \totmax^{1/2}(t,u),
					\label{E:L2XIJDERIVATIVESBOUNDEDINTERMSOFENERGIES} \\
			\left\|
				\Fullset^{[1,\Ntop-1];1} \CoordAngSmallcomp{i}{j}
			\right\|_{L^2(\Sigma_t^u)}
			& \leq C \mathring{\upepsilon}
				+
				C \totmax^{1/2}(t,u).
				\label{E:L2ANGUALARVECTORFIELDJCOMPONENTDERIVATIVESBOUNDEDINTERMSOFENERGIES}
		\end{align}
	\end{subequations}
\end{lemma}

\begin{proof}
	See Subsect.\ \ref{SS:OFTENUSEDESTIMATES} for some comments on the analysis.
	We set 
	\begin{align} \label{E:EIKONALFUNCTIONQUANTITIESGRONWALLABLEQUANTITY}
	q = q(t,u) 
	& := 
	\left\|
		\Tanset_*^{[1,\Ntop-1]} \upmu
	\right\|_{L^2(\Sigma_t^u)}^2
		\\
	& \ \
	+
	\sum_{j=1}^n
	\left\|
		\Fullset_*^{[1,\Ntop-1];1} \xi_j^{(Small)}
	\right\|_{L^2(\Sigma_t^u)}^2
	+
	\sum_{i=2}^n
	\sum_{j=1}^n
	\left\|
		\Fullset^{[1,\Ntop-1];1} \CoordAngSmallcomp{i}{j}
	\right\|_{L^2(\Sigma_t^u)}^2.
		\notag
	\end{align}
	The estimates from Subsubsect.\ \ref{SSS:DATAOFOTHERDATASIZE}
	for the data of 
	$\upmu$,
	$\upxi_j^{(Small)}$,
	and $\CoordAngSmallcomp{i}{j}$
	imply that 
	$q(0,u) \leq C \mathring{\upepsilon}^2$.
	Hence, from the pointwise estimates 
	\eqref{E:COMMUTEDLUNITUPXIJ}-\eqref{E:COMMUTEDLUNITCOORDANGIJ} and \eqref{E:COMMUTEDLUNITUPMU},
	the pointwise estimates
	\eqref{E:POINTWISEESTIMATESTRANSVERSALDERIVATIVESOFNOSHOCKCONTROLLEDBYTANGENTAIL}-\eqref{E:POINTWISEESTIMATESTRANSVERSALDERIVATIVESOFCARTESIANDIFFERENTIATEDNOSHOCKCONTROLLEDBYTANGENTAIL},
	Def.\ \ref{D:SHORTHANDARRAYS},
	Young's inequality,
	the energy identity \eqref{E:SHOCKVARIABLEBASICENERGYID},
	and Lemma~\ref{L:COERCIVENESSOFFUNDAMENTALCONTROLLING},
	we deduce that
	\begin{align} \label{E:GRONWALLREADYEIKONALFUNCTIONQUANTITIESL2}
		q(t,u)
		& \leq C \mathring{\upepsilon}^2
			+
			C
			\sum_{j=1}^n
			\int_{\mathcal{M}_{t,u}}
				\left|
					\Fullset_*^{[1,\Ntop-1];1} \xi_j^{(Small)}
				\right|^2
			\, d \vol
			\\
		& \ \
			+
			C
			\sum_{i=2}^n
			\sum_{j=1}^n
			\int_{\mathcal{M}_{t,u}}
				\left|
					\Fullset_*^{[1,\Ntop-1];1} \CoordAngSmallcomp{i}{j}
				\right|^2
			\, d \vol
			+
			C
			\int_{\mathcal{M}_{t,u}}
				\left|
					\Tanset_*^{[1,\Ntop-1]} \upmu
				\right|^2
			\, d \vol
			\notag \\
		& \ \
			+
			C
			\int_{\mathcal{M}_{t,u}}
				\left|
					\Fullset_*^{[1,\Ntop];1} \Psi
				\right|^2
			\, d \vol
			+
			C
			\int_{\mathcal{M}_{t,u}}
				\left|
					\Tanset^{\leq \Ntop-1} \noshock
				\right|^2
			\, d \vol
			+
			C
			\int_{\mathcal{M}_{t,u}}
				\left|
					\Tanset^{\leq \Ntop-1} \diffnoshock
				\right|^2
			\, d \vol
			\notag
			 \\
		& \leq C \mathring{\upepsilon}^2
			+ 
			C \int_{s=0}^t
					q(s,u)
				\, ds
			+ C 
				\int_{s=0}^t
					\totmax(s,u)
				\, ds	
			+	C 
				\int_{u'=0}^u
					\totmax(t,u')
				\, du'
				\notag 
				\\
		& \leq C \mathring{\upepsilon}^2
			+ 
			C \int_{s=0}^t
				q(s,u)
			\, ds
			+ 
			C \totmax(t,u).
				\notag
	\end{align}
	From \eqref{E:GRONWALLREADYEIKONALFUNCTIONQUANTITIESL2} and Gronwall's inequality,
	we conclude the bound $q(t,u) \leq C \mathring{\upepsilon}^2 + C \totmax(t,u)$,
	from which the estimates 
	\eqref{E:L2UPMUDERIVATIVESBOUNDEDINTERMSOFENERGIES}-\eqref{E:L2ANGUALARVECTORFIELDJCOMPONENTDERIVATIVESBOUNDEDINTERMSOFENERGIES} easily follow.
\end{proof}

\subsection{The main a priori estimates}
\label{SS:MAINAPRIORIESIMATES}
In the next proposition, we derive our main a priori energy estimates.

\begin{proposition}[\textbf{The main a priori estimates}]
	\label{P:MAINAPRIORI}
	There exists a constant $C > 0$ such 
	that under the data-size assumptions of Subsubsect.\ \ref{SSS:DATA},
	the bootstrap assumptions of Subsubsect.\ \ref{SSS:FUNDAMENTALBOOT},
	and the smallness assumptions of Subsect.\ \ref{SS:SMALLNESSASSUMPTIONS},
	the following estimates hold for $(t,u) \in [0,\Tboot) \times [0,U_0]$:
	\begin{align} \label{E:ENERGYGRONWALLREADY}
		\totmax(t,u)
		& \leq C \mathring{\upepsilon}^2
			+
			C
			\int_{s=0}^t
				\totmax(s,u)
			\, ds
			+
			C
			\int_{u'=0}^u
				\totmax(t,u')
			\, du'.
	\end{align}
	
	Moreover, as a consequence of
	\eqref{E:ENERGYGRONWALLREADY},
	the following estimate holds for $(t,u) \in [0,\Tboot) \times [0,1]$:
	\begin{align} \label{E:MAINAPRIORIENERGYESTIMATE}
		\totmax(t,u)
		& \leq C \mathring{\upepsilon}^2.
	\end{align}
\end{proposition}

\begin{remark}[\textbf{A top-order} $L^2$ \textbf{estimate for} $\noshock$]
	\label{R:TOPORDERNOSHOCKESTIMATE}
	From the pointwise estimate \eqref{E:POINTWISEESTIMATESTRANSVERSALDERIVATIVESOFNOSHOCKCONTROLLEDBYTANGENTAIL},
	the bootstrap assumptions,
	Lemma~\ref{L:L2ESTIMATESFOREIKONALFUNCTIONQUANTITIES},
	and \eqref{E:MAINAPRIORIENERGYESTIMATE},
	one can easily obtain
	the bound
	$\| \Fullset^{\leq \Ntop;1} \noshock \|_{L^2(\Sigma_t^u)} \leq C \mathring{\upepsilon}$,
	which is a gain of one derivative for $\noshock$ compared to what is directly 
	implied by \eqref{E:MAINAPRIORIENERGYESTIMATE}. 
	Similarly, we could gain a derivative for $\noshock$ in the $L^{\infty}$ estimate \eqref{E:IMPROVEMENTOFFUNDAMANETALBOOTSTRAPASSUMPTIONS}
	below. However, we have no need for these gains of a derivative, so we will ignore them for the remainder of the paper.
	\end{remark}

\begin{proof}[Proof of Prop.\ \ref{P:MAINAPRIORI}]
	\
	
	\medskip
	
	\noindent \underline{\textbf{Proof of \eqref{E:ENERGYGRONWALLREADY}}}:
	We first derive energy inequalities for $\Psi$ and its derivatives.
	Let $\vec{I} \in \mathcal{I}_*^{[1,\Ntop];1}$
	(see Def.\ \ref{D:COMMUTATORMULTIINDICES}).
	From the energy identity \eqref{E:SHOCKVARIABLEBASICENERGYID},
	the data-size assumption \eqref{E:L2SMALLDATAALONGSIGMA0},
	the pointwise estimate \eqref{E:SHOCKVARIABLEPOINTWISEINHOM},
	the estimates 
	\eqref{E:POINTWISEESTIMATESTRANSVERSALDERIVATIVESOFNOSHOCKCONTROLLEDBYTANGENTAIL}-\eqref{E:POINTWISEESTIMATESTRANSVERSALDERIVATIVESOFCARTESIANDIFFERENTIATEDNOSHOCKCONTROLLEDBYTANGENTAIL},
	and Young's inequality,
	we deduce
	\begin{align} \label{E:ENERGYESTIMATEFORPSIFIRSTSTEP}
			\shocken[\Fullset^{\vec{I}} \Psi](t,u)
			& \leq
			C \mathring{\upepsilon}^2
			+
			C
			\int_{\mathcal{M}_{t,u}}
			\left|
				\Fullset_*^{[1,\Ntop];1} \Psi
			\right|^2
			\, d \vol
				\\
	& \ \
		+ C
			\int_{\mathcal{M}_{t,u}}
			\left|
				\Tanset^{\leq \Ntop-1} \noshock
			\right|^2
			\, d \vol
			+
			\int_{\mathcal{M}_{t,u}}
			\left|
				\Tanset^{\leq \Ntop-1} \diffnoshock
			\right|^2
			\, d \vol
			\notag 
				\\
	& \ \
			+
			C
			\int_{\mathcal{M}_{t,u}}
				\left|
					\Tanset_*^{[1,\Ntop-1]} \upmu
				\right|^2
			\, d \vol
			+
			C
			\sum_{j=1}^n
			\int_{\mathcal{M}_{t,u}}
				\left|
					\Fullset_*^{[1,\Ntop-1];1} \xi_j^{(Small)}
				\right|^2
			\, d \vol
				\notag \\
		& \ \
			+
			C
			\sum_{i=2}^n
			\sum_{j=1}^n
			\int_{\mathcal{M}_{t,u}}
				\left|
					\Fullset_*^{[1,\Ntop-1];1} \CoordAngSmallcomp{i}{j}
				\right|^2
			\, d \vol.
			\notag
	\end{align}
	From Lemma~\ref{L:COERCIVENESSOFFUNDAMENTALCONTROLLING},
	Lemma~\ref{L:L2ESTIMATESFOREIKONALFUNCTIONQUANTITIES},
	and \eqref{E:ENERGYESTIMATEFORPSIFIRSTSTEP}, we deduce
	\begin{align} \label{E:ENERGYESTIMATEFORPSISECONDSTEP}
			\shocken[\Fullset^{\vec{I}} \Psi](t,u)
			& \leq
			C \mathring{\upepsilon}^2
			+ C 
				\int_{s=0}^t
					\totmax(s,u)
				\, ds	
			+	C 
				\int_{u'=0}^u
					\totmax(t,u')
				\, du'.
	\end{align}
	
	We now derive a similar energy inequality for $\noshock$, $\diffnoshock$, and their derivatives.
	Specifically, using the energy-characteristic flux identity \eqref{E:NONSHOCKVARIABLEBASICENERGYID},
	the data-size assumptions \eqref{E:L2SMALLDATAALONGSIGMA0} and \eqref{E:L2SMALLDATAALONGP0},
	the pointwise estimates \eqref{E:NONSHOCKVARIABLEPOINTWISEINHOM} 
	and \eqref{E:NONSHOCKCARTESIANDIFFERENTIATEDVARIABLEPOINTWISEINHOM},
	the estimates 
	\eqref{E:POINTWISEESTIMATESTRANSVERSALDERIVATIVESOFNOSHOCKCONTROLLEDBYTANGENTAIL}-\eqref{E:POINTWISEESTIMATESTRANSVERSALDERIVATIVESOFCARTESIANDIFFERENTIATEDNOSHOCKCONTROLLEDBYTANGENTAIL},
	Lemma~\ref{L:COERCIVENESSOFFUNDAMENTALCONTROLLING}, 
	Lemma~\ref{L:L2ESTIMATESFOREIKONALFUNCTIONQUANTITIES},
	and the $L^{\infty}$ estimates of Prop.\ \ref{P:POINTWISEANDIMPROVEMENTOFAUX},
	we deduce that for $|\vec{I}| \leq \Ntop - 1$, we have,
	for any 
	$f \in 
		\lbrace \noshockuparg{J} \rbrace_{1 \leq J \leq M} 
		\cup 
		\lbrace \diffnoshockdoublearg{\alpha}{J} \rbrace_{0 \leq \alpha \leq n; 1 \leq J \leq M}$, 
		the following estimate:
	\begin{align} \label{E:ENERGYESTIMATEFORREGULARVARIABLESFIRSTSTEP}
			&
			\noshocken[\Tanset^{\vec{I}} f](t,u)
			+
			\noshockfl[\Tanset^{\vec{I}} f](t,u)
				\\
			& \leq
				C \mathring{\upepsilon}^2
			+ C 
				\int_{s=0}^t
					\totmax(s,u)
				\, ds	
			+	C 
				\int_{u'=0}^u
					\totmax(t,u')
				\, du'.
				\notag
	\end{align}
	From \eqref{E:ENERGYESTIMATEFORPSISECONDSTEP}, 
	\eqref{E:ENERGYESTIMATEFORREGULARVARIABLESFIRSTSTEP},
	and Def.\ \ref{D:MAINCOERCIVEQUANT},
	we conclude the desired bound \eqref{E:ENERGYGRONWALLREADY}.

	\medskip
	\noindent \underline{\textbf{Proof of \eqref{E:MAINAPRIORIENERGYESTIMATE}}}:
	With $c > 0$ a real parameter, we define
	\begin{align} \label{E:INTEGRATINGFACTORMULIPLIEDCONTROLLINGQUANTITY}
		\totmax_c(t,u)
		& :=  \sup_{(\hat{t},\hat{u}) \in [0,t] \times [0,u]}
					\left\lbrace
						\exp(-c \hat{t}) \exp(-c \hat{u}) \totmax(\hat{t},\hat{u})
					\right\rbrace.
	\end{align}
	Using \eqref{E:ENERGYGRONWALLREADY} and the simple inequality
	$\int_{y' = 0}^y \exp(c y') \, dy' \leq \frac{1}{c} \exp(c y)$,
	we deduce that for $(\hat{t},\hat{u}) \in [0,t] \times [0,u] \subset [0,\Tboot) \times [0,U_0]$, 
	the following estimate holds:
	\begin{align} \label{E:ENERGYMIDDLEOFGRONWALLARGUMENT}
		& \exp(-c \hat{t}) \exp(-c \hat{u}) \totmax(\hat{t},\hat{u})
			\\
	& \leq C \exp(-c \hat{t}) \exp(-c \hat{u}) \mathring{\upepsilon}^2
				\notag \\
		& \ \
			+
			C \exp(-c \hat{t}) \exp(-c \hat{u})
			\times
			\left\lbrace
				\sup_{t' \in [0,\hat{t}]}
				\exp(-c t') \totmax(t',\hat{u})
			\right\rbrace
			\times
			\int_{t'=0}^{\hat{t}}
				\exp(c t')
			\, dt'
			\notag \\
	& \ \
			+
			C \exp(-c \hat{t}) \exp(-c \hat{u})
			\times
			\left\lbrace
				\sup_{u' \in [0,\hat{u}]} \exp(-c u') \totmax(\hat{t},u')
			\right\rbrace
			\times
			\int_{u'=0}^{\hat{u}}
				\exp(c u')
			\, du'
			\notag	\\
		& \leq C \mathring{\upepsilon}^2
			+
			\frac{2 C}{c} 
			\sup_{(t',u') \in [0,\hat{t}] \times [0,\hat{u}]} 
			\left\lbrace
				\exp(-c t') \exp(-c u') \totmax(t',u')
			\right\rbrace,
			\notag
	\end{align}
	where the constant $C$ on RHS~\eqref{E:ENERGYMIDDLEOFGRONWALLARGUMENT} can be chosen to be independent of $c > 0$.
	From \eqref{E:ENERGYMIDDLEOFGRONWALLARGUMENT}
	and definition \eqref{E:INTEGRATINGFACTORMULIPLIEDCONTROLLINGQUANTITY},
	we deduce that
	$\totmax_c(t,u) \leq C \mathring{\upepsilon}^2 
	+
	\frac{2 C}{c} \totmax_c(t,u)
	$.
	Hence, fixing $c := c' > 2C$, we deduce that
	$\totmax_{c'}(t,u) \leq C' \mathring{\upepsilon}^2$.
	From this bound and the definition of 
	$\totmax_{c'}$,
	it follows that for $(t,u) \in [0,\Tboot) \times [0,U_0]$, we have
	$\totmax(t,u) \leq C' \exp(c' t) \exp(c' u) \mathring{\upepsilon}^2
	\leq C'' \mathring{\upepsilon}^2
	$, where $C''$ depends on $C'$, $c'$, and $\TranminusdatasizeWithFactor^{-1}$
	(in view of the bootstrap assumption \eqref{E:TBOOTBOUNDS}).
	This is precisely the desired bound \eqref{E:MAINAPRIORIENERGYESTIMATE}.
\end{proof}

\begin{corollary}[\textbf{Improvement of the fundamental $L^{\infty}$ bootstrap assumptions}]
		\label{C:IMPROVEMENTOFFUNDAMANETALBOOTSTRAPASSUMPTIONS}
		For $0 \leq \alpha \leq n$ and $1 \leq J \leq M$,
		the following estimates hold:
		\begin{align} \label{E:IMPROVEMENTOFFUNDAMANETALBOOTSTRAPASSUMPTIONS}
	\left\|
		\Tanset^{\leq \Nmid-1} v^J
	\right\|_{L^{\infty}(\Sigma_t^u)},
		\,
	\left\|
		\Tanset^{\leq \Nmid-1} V_{\alpha}^J
	\right\|_{L^{\infty}(\Sigma_t^u)}
	& \leq C \mathring{\upepsilon}.
\end{align}
In particular, if $C \mathring{\upepsilon} < \varepsilon$, then 
the estimate \eqref{E:IMPROVEMENTOFFUNDAMANETALBOOTSTRAPASSUMPTIONS}
is a strict improvement 
of the fundamental bootstrap assumption \eqref{E:FUNDAMENTALBOOTSTRAP}.
		
\end{corollary}

\begin{proof}	
	\eqref{E:IMPROVEMENTOFFUNDAMANETALBOOTSTRAPASSUMPTIONS}
	follows from the energy estimate \eqref{E:MAINAPRIORIENERGYESTIMATE}
	and the Sobolev embedding result \eqref{E:BOOTSTRAPNORMSINTERMSOFCONTROLLINGQUANTITY}.
\end{proof}

\section{Continuation criteria}
\label{S:CONTINUATION}
In this section, we provide a proposition that yields continuation criteria.
We will use the proposition during the proof of the main theorem (Theorem~\ref{T:MAINTHM}), 
specifically as an ingredient in showing that the solution 
survives until the shock.

\begin{proposition}[\textbf{Continuation criteria}]
\label{P:CONTINUATIONCRITERIA}
Let $(\Psi,\noshockuparg{1},\cdots,\noshockuparg{M})$
be a smooth solution to the system \eqref{E:SHOCKEQN}-\eqref{E:NONSHOCKEQUATION}
verifying the size assumptions\footnote{Recall that 
even though we make size assumptions only for certain Sobolev norms,
for technical convenience, 
we have assumed that the data on $\Sigma_0^1$ and $\mathcal{P}_0^{2 \TranminusdatasizeWithFactor^{-1}}$
are $C^{\infty}$.}
on $\Sigma_0^1$ and $\mathcal{P}_0^{2 \TranminusdatasizeWithFactor^{-1}}$
stated in Subsect.\ \ref{SS:DATASIZE}
as well as the smallness assumptions stated in Subsect.\ \ref{SS:SMALLNESSASSUMPTIONS}.
Let $T_{(Local)} \in (0,2 \TranminusdatasizeWithFactor^{-1})$ and $U_0 \in (0,1]$, and assume that
the solution exists classically on the
(``open at the top'')
spacetime region $\mathcal{M}_{T_{(Local)},U_0}$
that is completely determined by the data on
$\Sigma_0^{U_0} \cup \mathcal{P}_0^{2 \TranminusdatasizeWithFactor^{-1}}$
(see Fig.\ \ref{F:REGION} on pg.~\pageref{F:REGION}).
Let $u$ be the eikonal function that verifies the eikonal equation initial value problem \eqref{E:EIKONAL},
let $\upmu$ be the inverse foliation density of the characteristics $\mathcal{P}_u$ defined in \eqref{E:MUDEF}, 
and let $\uplambda_{\alpha} = \upmu \partial_{\alpha} u$ (as in \eqref{E:EIKONALFUNCTIONONEFORMS}).
Assume that $\upmu > 0$ on $\mathcal{M}_{T_{(Local)},U_0}$ and that
the change of variables map $\Upsilon$ from geometric to 
Cartesian coordinates (see Def.\ \ref{D:CHOV}) is 
a diffeomorphism from $[0,T_{(Local)}) \times [0,U_0] \times \mathbb{T}^{n-1}$ onto $\mathcal{M}_{T_{(Local)},U_0}$
(where $\mathcal{M}_{T_{(Local)},U_0}$ is defined in \eqref{E:MTUDEF})
such that for $i=2,\cdots,n$ and $\alpha = 0,\cdots,n$, we have
\begin{align} \label{E:CONTINUATIONCHOVREGULARITY}
	\Upsilon^{\alpha},
		\,
	\CoordAng{i} \Upsilon^{\alpha}
	\in
	C\left([0,T_{(Local)}),W^{1,\infty}([0,U_0] \times \mathbb{T}^{n-1}) \right)
	\cap
	C^1\left([0,T_{(Local)}),L^{\infty}([0,U_0] \times \mathbb{T}^{n-1}) \right).
\end{align}
Let $\mathcal{H} \subset \mathbb{R} \times \mathbb{R}^M \times \mathbb{R}^{1+n}$ 
be the set of arrays $(\widetilde{\Psi},\widetilde{\noshock},\widetilde{\uplambda})$
such that the following two conditions hold:
\begin{itemize}
	\item The Cartesian components $\Lunit^i(\Psi,\noshock)$, $(i=1,\cdots,n)$, 
		and the $M \times M$ matrices 
		$A^{\alpha}(\Psi,\noshock)$,
		$(\alpha = 0,\cdots,n)$,
		are smooth
		functions for $(\Psi,\noshock)$ belonging to a neighborhood of $(\widetilde{\Psi},\widetilde{\noshock})$.
	\item $A^0(\Psi,\noshock)$ and 
		$A^{\alpha}(\Psi,\noshock) \uplambda_{\alpha}$
		are positive definite matrices
		for $(\Psi,\noshock,\uplambda)$ 
		belonging to a neighborhood of $(\widetilde{\Psi},\widetilde{\noshock},\widetilde{\uplambda})$.
\end{itemize}

Assume that none of the following four breakdown scenarios occur:
\begin{enumerate}
	\item $\inf_{\mathcal{M}_{T_{(Local)},U_0}} \upmu = 0$. 
	\item $\sup_{\mathcal{M}_{T_{(Local)},U_0}} \upmu = \infty$.
	\item There exists a sequence $p_n \in \mathcal{M}_{T_{(Local)},U_0}$
		such that $(\Psi(p_n),\noshock(p_n),\uplambda(p_n))$ escapes every compact subset of $\mathcal{H}$ as $n \to \infty$.
	\item $\sup_{\mathcal{M}_{T_{(Local)},U_0}} 
				\max_{\alpha=0,1,\cdots,n}
				\left\lbrace
				\left|
					\partial_{\alpha} \Psi
				\right|
				+
				\left|
					\diffnoshockdownarg{\alpha}
				\right|
			\right\rbrace
		= \infty
		$,
	where $\diffnoshockdoublearg{\alpha}{J} = \partial_{\alpha} \noshockuparg{J}$.
	\end{enumerate}

In addition, assume that the following condition is verified:
\begin{enumerate}
	\setcounter{enumi}{4}
	\item The change of variables map $\Upsilon$
		extends to the compact set $[0,T_{(Local)}] \times [0,U_0] \times \mathbb{T}^{n-1}$
		as a diffeomorphism onto its image
		that enjoys the regularity properties 
		\eqref{E:CONTINUATIONCHOVREGULARITY}
		with $[0,T_{(Local)})$ replaced by $[0,T_{(Local)}]$.
\end{enumerate}

Then there exists a $\Delta > 0$ such that 
$\Psi$, 
$\noshock$,
$\diffnoshock$,
$u$, 
$\upmu$,
$\uplambda$,
and all of the
other geometric quantities defined
throughout the article can be uniquely extended 
(where $\Psi$, $\noshock$, $u$, and $\upmu$ are smooth solutions to their evolutions equations)
to a strictly larger region of the form
$\mathcal{M}_{T_{(Local)} + \Delta,U_0}$
into which their Sobolev regularity 
along $\Sigma_0^{U_0}$ and $\mathcal{P}_0^{2 \TranminusdatasizeWithFactor^{-1}}$
(described in Subsect.\ \ref{SS:DATASIZE})
is propagated.\footnote{Put differently, 
the same norms that are finite
along $\Sigma_0^{U_0}$ and $\mathcal{P}_0^{2 \TranminusdatasizeWithFactor^{-1}}$
(as stated in Subsect.\ \ref{SS:DATASIZE})
are also finite along $\Sigma_t^u$ and $\mathcal{P}_u^t$
for $(t,u) \in [0,T_{(Local)} + \Delta] \times [0,U_0]$.}
Moreover, if $\Delta$ is sufficiently small,
then none of the four breakdown scenarios
occur in the larger region, and
$\Upsilon$ extends to $[0,T_{(Local)} + \Delta] \times [0,U_0] \times \mathbb{T}^{n-1}$
as a diffeomorphism onto its image
that enjoys the regularity properties 
\eqref{E:CONTINUATIONCHOVREGULARITY}
with $[0,T_{(Local)})$ replaced by $[0,T_{(Local)} + \Delta]$.
\end{proposition}

\begin{proof}[Discussion of proof]
	The proof of Prop.\ \ref{P:CONTINUATIONCRITERIA} is mostly standard.
	A sketch of a similar result was provided in
	\cite{jS2016b}*{Proposition 21.1.1}, so here, 
	we only mention the main ideas.
	Criterion $(3)$ is connected to avoiding a breakdown
	in hyperbolicity of the equation. 
	Criterion $(4)$ is a standard criterion used to locally continue
	the solution relative to the Cartesian
	coordinates.
	Criteria $(1)$ and $(2)$ and the assumption $(5)$ for $\Upsilon$
	are connected to ruling out the blowup of $u$, 
	degeneracy of the change of variables map,
	and degeneracy of the region $\mathcal{M}_{T_{(Local)},U_0}$.
	In particular, criteria $(1)$ and $(2)$
	play a role in a proving that
	$\sum_{a = 1}^n |\partial_a u|$ is uniformly bounded from
	above and strictly from below away from $0$
	on $\mathcal{M}_{T_{(Local)},U_0}$
	(the proof was essentially given in the proof of Prop.\ \ref{P:CHOVREMAINSADIFFEOMORPHISM}).
	
\end{proof}

\section{The main theorem}
\label{S:MAINTTHM}
We now prove the main result of the paper.

\begin{theorem}[\textbf{Stable shock formation}]
	\label{T:MAINTHM}
	Let $n$ denote the number of spatial dimensions,
	let $\Ntop$ and $\Nmid$ be positive integers verifying \eqref{E:NDEFS},
	and let 
	$\Psiep > 0$,
	$\mathring{\upepsilon} \geq 0$, 
	$\Trandatasize > 0$, 
	and $\TranminusdatasizeWithFactor > 0$
	be the data-size parameters from Subsect.\ \ref{SS:DATASIZE}.
	For each $U_0 \in (0,1]$ (as in \eqref{E:FIXEDPARAMETER}), let 
\begin{align*}
	& T_{(Lifespan);U_0}
		\\
	& := 
	\sup 
	\Big\lbrace 
		t \in [0,\infty) \ | \ \mbox{the solution exists classically on } \mathcal{M}_{t;U_0}
			\\
	& \ \ \ \ \ \ \ \ \ \ \ \ \ \ \ \ \ \ \ \ \ \ \ \
		\mbox{ and $\Upsilon$ is a diffeomorphism from } [0,t) \times [0,U_0] \times \mathbb{T}^{n-1} 
		\mbox{ onto its image}
		\Big\rbrace,
	\end{align*}
	where $\Upsilon$ is the change of variables map from Def.\ \ref{D:CHOV}.
	If $\Psiep$ is sufficiently small relative to $1$ and if
	$\mathring{\upepsilon}$ is sufficiently small relative to
	$1$,
	$\Trandatasize^{-1}$, 
	and $\TranminusdatasizeWithFactor$
	in the sense explained in Subsect.\ \ref{SS:SMALLNESSASSUMPTIONS},
	then the following conclusions hold,
	where all constants can be chosen to be independent of $U_0$
	(see Subsect.\ \ref{SS:NOTATIONANDINDEXCONVENTIONS} for our conventions regarding the dependence of constants
	on the various parameters).

\noindent \underline{\textbf{Dichotomy of possibilities}}.
One of the following mutually disjoint possibilities must occur,
where $\upmu_{\star}(t,u) = \min_{\Sigma_t^u} \upmu$ (as in \eqref{E:MUSTARDEF})
and $\upmu$ is the inverse foliation density of the characteristics $\mathcal{P}_u$ from
Def.\ \ref{D:MUDEF}.
\begin{enumerate}
	\renewcommand{\labelenumi}{\textbf{\Roman{enumi})}}
	\item $T_{(Lifespan);U_0} > 2 \TranminusdatasizeWithFactor^{-1}$. 
		In particular, the solution exists classically on the spacetime
		region $\mbox{\upshape cl} \mathcal{M}_{2 \TranminusdatasizeWithFactor^{-1},U_0}$,
		where $\mbox{\upshape cl}$ denotes closure.
		Furthermore, $\inf \lbrace \upmu_{\star}(s,U_0) \ | \ s \in [0,2 \TranminusdatasizeWithFactor^{-1}] \rbrace > 0$.
	\item $0 <  T_{(Lifespan);U_0} \leq 2 \TranminusdatasizeWithFactor^{-1}$,
		and 
		\begin{align} \label{E:MAINTHEOREMLIFESPANCRITERION}
		T_{(Lifespan);U_0} 
		= \sup 
			\left\lbrace 
			t \in [0, 2 \TranminusdatasizeWithFactor^{-1}) \ | \
				\inf \lbrace \upmu_{\star}(s,U_0) \ | \ s \in [0,t) \rbrace > 0
			\right\rbrace.
		\end{align}
\end{enumerate}
In addition, case $\textbf{II)}$ occurs when $U_0 = 1$, and we have the estimate\footnote{See 
	Subsect.\ \ref{SS:NOTATIONANDINDEXCONVENTIONS} regarding our use of the symbol
	$\mathcal{O}_{\star}$.}
\begin{align} \label{E:CLASSICALLIFESPANASYMPTOTICESTIMATE}
	T_{(Lifespan);1} 
	= 
	\left\lbrace
		1 
		+ 
		\mathcal{O}_{\star}(\Psiep)
		+
		\mathcal{O}(\mathring{\upepsilon})
	\right\rbrace
	\TranminusdatasizeWithFactor^{-1}.
\end{align}

\noindent \underline{\textbf{What happens in Case I)}}.
In case $\textbf{I)}$, 
the energy estimates of Prop.\ \ref{P:MAINAPRIORI}
and the $L^{\infty}$ estimates of Cor.\ \ref{C:IMPROVEMENTOFFUNDAMANETALBOOTSTRAPASSUMPTIONS}
hold on $\mbox{\upshape cl} \mathcal{M}_{2 \TranminusdatasizeWithFactor^{-1},U_0}$.
The same is true for the estimates of Lemma~\ref{L:CONTROLOFTRANSVERSALINTERMSOFTANGENTIAL} and
Prop.\ \ref{P:POINTWISEANDIMPROVEMENTOFAUX}, but
with all factors $\varepsilon$ on the RHS of all inequalities 
replaced by $C \mathring{\upepsilon}$.
Moreover, for the quantities from Def.\ \ref{D:PERTURBEDPART},
the following estimates hold for
$2 \leq i \leq n$,
$1 \leq j \leq n$, 
and $(t,u) \in [0,2 \TranminusdatasizeWithFactor^{-1}] \times [0,U_0]$
(see Subsect.\ \ref{SS:STRINGSOFCOMMUTATIONVECTORFIELDS} regarding the differential operator notation):
\begin{subequations}
		\begin{align}
			\left\|
				\Tanset_*^{[1,\Ntop-1]} \upmu
			\right\|_{L^2(\Sigma_t^u)}
			& \leq C \mathring{\upepsilon},
					\label{E:MAINTHML2UPMUDERIVATIVESBOUNDEDINTERMSOFENERGIES} \\
			\left\|
				\Fullset_*^{[1,\Ntop-1];1} \xi_j^{(Small)}
			\right\|_{L^2(\Sigma_t^u)}
			& \leq C \mathring{\upepsilon},
					\label{E:MAINTHML2XIJDERIVATIVESBOUNDEDINTERMSOFENERGIES} \\
			\left\|
				\Fullset^{[1,\Ntop-1];1} \CoordAngSmallcomp{i}{j}
			\right\|_{L^2(\Sigma_t^u)}
			& \leq C \mathring{\upepsilon}.
				\label{E:MAINTHML2ANGUALARVECTORFIELDJCOMPONENTDERIVATIVESBOUNDEDINTERMSOFENERGIES}
		\end{align}
	\end{subequations}

\noindent \underline{\textbf{What happens in Case II)}}.
In case $\textbf{II)}$, 
the energy estimates of Prop.\ \ref{P:MAINAPRIORI}
and the $L^{\infty}$ estimates of Cor.\ \ref{C:IMPROVEMENTOFFUNDAMANETALBOOTSTRAPASSUMPTIONS}
hold on $\mathcal{M}_{T_{(Lifespan);U_0},U_0}$,
as do the estimates of Lemma~\ref{L:CONTROLOFTRANSVERSALINTERMSOFTANGENTIAL} and Prop.\ \ref{P:POINTWISEANDIMPROVEMENTOFAUX}
with all factors $\varepsilon$ on the RHS of all inequalities 
replaced by $C \mathring{\upepsilon}$.
Moreover, the estimates 
\eqref{E:MAINTHML2UPMUDERIVATIVESBOUNDEDINTERMSOFENERGIES}-\eqref{E:MAINTHML2ANGUALARVECTORFIELDJCOMPONENTDERIVATIVESBOUNDEDINTERMSOFENERGIES}
hold for $(t,u) \in [0,T_{(Lifespan);U_0}) \times [0,U_0]$.
In addition, the scalar functions
$\Fullset^{\leq \Nmid-1;1} \Psi$,
$\Fullset^{\leq \Nmid-2;1} \noshockuparg{J}$,
$\Fullset^{\leq \Nmid-2;1} \diffnoshockdoublearg{\alpha}{J}$,
$\Tanset^{\leq \Nmid-2} \upmu$,
$\Fullset^{\leq \Nmid-2;1} \upxi_j$,
$\Fullset^{\leq \Nmid-2;1}\CoordAngcomp{i}{j}$,
$\Fullset^{\leq \Nmid-2;1} \Lunit^i$,
$\Tanset^{\leq \Nmid-2} \Rad^i$,
and
$\Fullset^{\leq \Nmid-2;1} \Radunit^i$
extend to 
$\Sigma_{T_{(Lifespan);U_0}}^{U_0}$ 
as functions of 
the geometric coordinates $(t,u,\vartheta)$ 
belonging to the space 
$C\left([0,T_{(Lifespan);U_0}],L^{\infty}([0,U_0] \times \mathbb{T}^{n-1}) \right)$.

Moreover,	let $\Sigma_{T_{(Lifespan);U_0}}^{U_0;(Blowup)}$
be the subset of $\Sigma_{T_{(Lifespan);U_0}}^{U_0}$ 
defined by
\begin{align} \label{E:BLOWUPPOINTS}
	\Sigma_{T_{(Lifespan);U_0}}^{U_0;(Blowup)}
	:= \left\lbrace
			(T_{(Lifespan);U_0},u,\vartheta)
			\ | \
			\upmu(T_{(Lifespan);U_0},u,\vartheta)
			= 0
		\right\rbrace.
\end{align}
Then for each point $(T_{(Lifespan);U_0},u,\vartheta) \in \Sigma_{T_{(Lifespan);U_0}}^{U_0;(Blowup)}$,
there exists a past neighborhood containing it such that the following lower bound holds in
the neighborhood:
\begin{align} \label{E:BLOWUPPOINTINFINITE}
	\left| \Radunit \Psi(t,u,\vartheta) \right|
	\geq \frac{1}{4 \TranminusdatasizeWithFactor} 
	\frac{1}{|\widetilde{\blowupcoeff}| \upmu(t,u,\vartheta)},
\end{align}
where $\widetilde{\blowupcoeff} := \blowupcoeff|_{(\Psi,\noshock) = (0,0)}$
is the blowup-coefficient of Def.\ \ref{D:GENIUNELYNONLINEARCONSTANT},
evaluated at the background value of $(\Psi,\noshock) = (0,0)$
(see Remark~\ref{R:FUNCTIONALDEPENDENCEOFBLOWUPCOEFFICIENT}).
In \eqref{E:BLOWUPPOINTINFINITE}, 
$
\displaystyle
\frac{1}{8 |\widetilde{\blowupcoeff}| \TranminusdatasizeWithFactor} 
$
is a \textbf{positive}\footnote{See Remark~\ref{R:BLOWUPCOEFFICIENTISNONZERO}.} 
data-dependent constant,
$\blowupcoeff(0,u,\vartheta)$
and the $\mathcal{T}_{t,u}$-transversal vectorfield $\Radunit$ is of order-unity Euclidean length:
$C^{-1} \leq \delta_{ab} \Radunit^a \Radunit^b \leq C$,
where $\delta_{ij}$ is the standard Kronecker delta.
In particular, 
$\Radunit \Psi$
blows up like $1/\upmu$ at all points in $\Sigma_{T_{(Lifespan);U_0}}^{U_0;(Blowup)}$.
Conversely, at all points in
$(T_{(Lifespan);U_0},u,\vartheta) \in \Sigma_{T_{(Lifespan);U_0}}^{U_0} \backslash \Sigma_{T_{(Lifespan);U_0}}^{U_0;(Blowup)}$,
we have
\begin{align} \label{E:NONBLOWUPPOINTBOUND}
	\left| \Radunit \Psi (T_{(Lifespan);U_0},u,\vartheta) \right|
	< \infty.
\end{align}

\end{theorem}

\begin{proof}
	Let $C' > 1$ be a constant. We will enlarge $C'$ as needed throughout the proof.
	We define
	\begin{align}
		T_{(Max);U_0} &:= \mbox{ The supremum of the set of times } 
		\Tboot \in [0,2 \TranminusdatasizeWithFactor^{-1}] \mbox{ such that:} 
			\label{E:LIFESPANPROOF}	\\
			& \bullet \mbox{$\Psi$, $\noshockuparg{J}$, $\diffnoshockdoublearg{\alpha}{J}$, $u$, $\upmu$, $\xi_j^{(Small)}$, 
			$\CoordAngSmallcomp{i}{j}$,
				and all of the other quantities} 
					\notag \\
			& \ \ \mbox{defined throughout the article
				  exist classically on } \mathcal{M}_{\Tboot,U_0}.
					\notag \\
		& \bullet 
			\mbox{The change of variables map $\Upsilon$ from Def.\ \ref{D:CHOV} 
			is a (global) diffeomorphism}
				\notag \\
		& \ \ 
			\mbox{from $[0,\Tboot) \times [0,U_0] \times \mathbb{T}^{n-1}$
			onto its image $\mathcal{M}_{\Tboot,U_0}$ verifying}
				\notag \\
	& \ \
	\Upsilon^{\alpha},
		\,
	\frac{\partial}{\partial \vartheta^i} \Upsilon^{\alpha}
	\in
	C\left([0,\Tboot),W^{1,\infty}([0,U_0] \times \mathbb{T}^{n-1}) \right)
	\cap
	C^1\left([0,\Tboot),L^{\infty}([0,U_0] \times \mathbb{T}^{n-1}) \right).
	\notag \\
		& \bullet \inf \left\lbrace \upmu_{\star}(t,U_0) \ | \ t \in [0,\Tboot) \right\rbrace > 0,
			\mbox{where } \upmu_{\star} \mbox{ is defined in Def.\ \ref{D:MUSTARDEF}}.
				\notag \\
		& \bullet \mbox{The fundamental } L^{\infty} \mbox{ bootstrap assumptions } 
						\eqref{E:FUNDAMENTALBOOTSTRAP}
						\notag \\
		& \ \ \mbox{ hold with } \varepsilon := C' \mathring{\upepsilon}
								\mbox{ for  } (t,u) \in \times [0,\Tboot) \times [0,U_0].
					\notag 
		\end{align}
	By standard local well-posedness for quasilinear hyperbolic systems,
	if $\Psiep$ and
	$\mathring{\upepsilon}$ are sufficiently small 
	in the sense explained in Subsect.\ \ref{SS:SMALLNESSASSUMPTIONS}
	and $C'$ is sufficiently large,
	then $T_{(Max);U_0} > 0$. Under the same smallness/largeness assumptions, 
	by Cor.\ \ref{C:IMPROVEMENTOFFUNDAMANETALBOOTSTRAPASSUMPTIONS}, 
	the bootstrap assumptions \eqref{E:FUNDAMENTALBOOTSTRAP}
	are not saturated for $(t,u) \in [0,T_{(Max);U_0}) \times [0,U_0]$.
	For this reason, all estimates proved throughout the article on the basis of the bootstrap assumptions
	in fact hold on $\mathcal{M}_{\Tboot,U_0}$ with
	$\varepsilon$ replaced by $C \mathring{\upepsilon}$.
	We use this fact throughout the remainder of the proof without further remark. 
	In particular, the estimates of Prop.\ \ref{P:POINTWISEANDIMPROVEMENTOFAUX} hold 
	for $(t,u) \in [0,T_{(Max);U_0}) \times [0,U_0]$
	with all factors $\varepsilon$ on the RHS of all inequalities 
	replaced by $C \mathring{\upepsilon}$.
	Moreover, by inserting the energy estimates of Prop.\ \ref{P:MAINAPRIORI}
	into the RHSs of the estimates of Lemma~\ref{L:L2ESTIMATESFOREIKONALFUNCTIONQUANTITIES},
	we conclude that the estimates 
	\eqref{E:MAINTHML2UPMUDERIVATIVESBOUNDEDINTERMSOFENERGIES}-\eqref{E:MAINTHML2ANGUALARVECTORFIELDJCOMPONENTDERIVATIVESBOUNDEDINTERMSOFENERGIES}
	hold for $(t,u) \in [0,T_{(Max);U_0}) \times [0,U_0]$.
	
	We now establish the dichotomy of possibilities.
	We first show that if 
	\begin{align} \label{E:CONDITIONALUPMUREMAINSPOSITIVE}
		\inf \left\lbrace \upmu_{\star}(t,U_0) \ | \ t \in [0,T_{(Max);U_0}) \right\rbrace > 0,
	\end{align}
	then $T_{(Max);U_0} = 2 \TranminusdatasizeWithFactor^{-1}$.
	To proceed, we assume for the sake of contradiction that
	\eqref{E:CONDITIONALUPMUREMAINSPOSITIVE} holds but that
	$T_{(Max);U_0} < 2 \TranminusdatasizeWithFactor^{-1}$.
	Then from \eqref{E:CONDITIONALUPMUREMAINSPOSITIVE}
	and Prop.\ \ref{P:CHOVREMAINSADIFFEOMORPHISM},
	we see that if 
	$\Psiep$ and
	$\mathring{\upepsilon}$ are sufficiently small,
	then $\Upsilon$ 
	extends to a global diffeomorphism from
	$[0,T_{(Max);U_0}] \times [0,U_0] \times \mathbb{T}$
	onto its image that enjoys the regularity \eqref{E:CHOVGLOBALDIFFEOMORPHISMREGULARITY}
	(with $\Tboot$ replaced by $T_{(Max);U_0}$ in \eqref{E:CHOVGLOBALDIFFEOMORPHISMREGULARITY}).
	Also using the assumption \eqref{E:POSITIVEDEFMATRICES},
	Definition~\ref{D:EIKONALFUNCTIONONEFORMS},
	definition \eqref{E:UPXIJSMALL},
	and the estimates of Prop.\ \ref{P:POINTWISEANDIMPROVEMENTOFAUX},
	we see that none of the four breakdown scenarios
	of Prop.\ \ref{P:CONTINUATIONCRITERIA} occur
	on $\mathcal{M}_{T_{(Max);U_0},U_0}$.
	Hence, by Prop.\ \ref{P:CONTINUATIONCRITERIA},
	we can classically extend the solution
	to a region of the form $\mathcal{M}_{T_{(Max);U_0} + \Delta,U_0}$,
	with $\Delta > 0$
	and
	$T_{(Max);U_0} + \Delta 
		< 2 \TranminusdatasizeWithFactor^{-1}
	$,
	such that all of the properties defining $T_{(Max);U_0}$ hold
	for the larger time
	$T_{(Max);U_0} + \Delta$.
	This contradicts the definition of $T_{(Max);U_0}$
	and in fact implies that if \eqref{E:CONDITIONALUPMUREMAINSPOSITIVE} holds
	and if $\Psiep$ and $\mathring{\upepsilon}$ are sufficiently small, then
	\textbf{I)} $T_{(Max);U_0} = 2 \TranminusdatasizeWithFactor^{-1}$ and $T_{(Lifespan);U_0} > 2 \TranminusdatasizeWithFactor^{-1}$.
	The only other possibility is:
	\textbf{II)} $\inf \left\lbrace \upmu_{\star}(t,U_0) \ | \ t \in [0,T_{(Max);U_0}) \right\rbrace = 0$.
	
	We now aim to show that case \textbf{II)} corresponds to the formation of a 
	shock singularity in the constant-time hypersurface subset $\Sigma_{T_{(Max);U_0}}^{U_0}$.
	We first derive the statements regarding the quantities that extend to 
	$\Sigma_{T_{(Lifespan);U_0}}^{U_0}$ as 
	elements of the space $C\left([0,T_{(Lifespan);U_0}],L^{\infty}([0,U_0] \times \mathbb{T}) \right)$.
	Here we will prove the desired results with
	$T_{(Max);U_0}$ in place of $T_{(Lifespan);U_0}$;
	in the next paragraph, we will show that $T_{(Max);U_0} = T_{(Lifespan);U_0}$.
	Let $q$ denote any of the quantities
	$\Fullset^{\leq \Nmid-1;1} \Psi$,
	$\cdots$,
	$\Fullset^{\leq \Nmid-2;1} \Radunit^i$
	that, in the theorem, are stated to extend.
	From the estimates of 
	Lemma~\ref{L:CONTROLOFTRANSVERSALINTERMSOFTANGENTIAL} and
	Prop.\ \ref{P:POINTWISEANDIMPROVEMENTOFAUX},
	we deduce that $\| \Lunit q \|_{L^{\infty}(\Sigma_t^{U_0})}$
	is uniformly bounded
	for $0 \leq t < T_{(Max);U_0}$.
	Using this fact,
	the fact that
	$
	\Lunit = \frac{\partial}{\partial t}
	$,
	the fundamental theorem of calculus,
	and the completeness of the space $L^{\infty}([0,U_0] \times \mathbb{T})$,
	we conclude that $q$ extends to $\Sigma_{T_{(Max);U_0}}^{U_0}$
	as a function of the geometric coordinates $(t,u,\vartheta)$ 
	belonging to the space
	$C\left([0,T_{(Max);U_0}],L^{\infty}([0,U_0] \times \mathbb{T}) \right)$,
	as desired.
	
	We now show that the classical lifespan is characterized by \eqref{E:MAINTHEOREMLIFESPANCRITERION}
	and that $T_{(Max);U_0} = T_{(Lifespan);U_0}$.
	To this end, we first use
	\eqref{E:ABSOLUTEVALUECRUCIALESTIMATESFORUPMUANDRADPSI}
	and the continuous extension properties proved in the previous paragraph
	to deduce \eqref{E:BLOWUPPOINTINFINITE}.
	Also using
	Def.\ \ref{D:PERTURBEDPART},
	the schematic relation $\Radunit_{(Small)}^j = \GdVar \smoothfunction(\GdVar)$,
	and the $L^{\infty}$ estimates of Prop.\ \ref{P:POINTWISEANDIMPROVEMENTOFAUX},
	we deduce that
	$C^{-1} \leq \delta_{ab} \Radunit^a \Radunit^b \leq C
	$.
	That is, the vectorfield $\Radunit$ is of order-unity Euclidean length.
	From this estimate and \eqref{E:BLOWUPPOINTINFINITE}, 
	we deduce that at points in $\Sigma_{T_{(Max);U_0}}^{U_0}$
	where $\upmu$ vanishes, $|\Radunit \Psi|$ 
	blows up like $1/\upmu$.
	Hence, $T_{(Max);U_0}$ is the classical lifespan. 
	That is, we have $T_{(Max);U_0} = T_{(Lifespan);U_0}$
	as well as the characterization \eqref{E:MAINTHEOREMLIFESPANCRITERION} of the 
	classical lifespan. 
	The estimate \eqref{E:NONBLOWUPPOINTBOUND}
	follows from the estimate \eqref{E:RADPSILINFTY},
	the fact that $\Rad = \upmu \Radunit$,
	and the continuous extension properties proved in the previous paragraph.
	
	Finally, to obtain \eqref{E:CLASSICALLIFESPANASYMPTOTICESTIMATE},
	we use \eqref{E:MUSTARKEYESTIMATE} to conclude
	that $\upmu_{\star}(t,1)$
	vanishes for the first time when
	$
	t 
	= 
	\left\lbrace
		1 
		+ 
		\mathcal{O}_{\star}(\Psiep)
		+
		\mathcal{O}(\mathring{\upepsilon})
	\right\rbrace
	\TranminusdatasizeWithFactor^{-1} 
	$.
	We have therefore proved the theorem.
	
\end{proof}

\bibliographystyle{amsalpha}
\bibliography{JBib}

\end{document}